\documentclass[11pt]{article}

\usepackage{amsmath,amsthm,amsfonts,latexsym,amssymb, stmaryrd}
\usepackage[all]{xy}
\usepackage{geometry}

\begin{document}

\newtheorem*{theo}{Theorem}
\newtheorem*{pro} {Proposition}
\newtheorem*{cor} {Corollary}
\newtheorem*{lem} {Lemma}
\newtheorem{theorem}{Theorem}[section]
\newtheorem{corollary}[theorem]{Corollary}
\newtheorem{lemma}[theorem]{Lemma}
\newtheorem{proposition}[theorem]{Proposition}
\newtheorem{conjecture}[theorem]{Conjecture}

\theoremstyle{definition}
 \newtheorem{definition}[theorem]{Definition} 
  \newtheorem{example}[theorem]{Example}
   \newtheorem{remark}[theorem]{Remark}
   
\newcommand{\Naturali}{{\mathbb{N}}}
\newcommand{\Reali}{{\mathbb{R}}}
\newcommand{\Complessi}{{\mathbb{C}}}
\newcommand{\Toro}{{\mathbb{T}}}
\newcommand{\Relativi}{{\mathbb{Z}}}
\newcommand{\HH}{\mathfrak H}
\newcommand{\KK}{\mathfrak K}
\newcommand{\LL}{\mathfrak L}
\newcommand{\as}{\ast_{\sigma}}
\newcommand{\tn}{\vert\hspace{-.3mm}\vert\hspace{-.3mm}\vert}
\newcommand{\Mb}{{M^{\rm bim}_0A(\Sigma)}}
\newcommand{\Mbp}{{M^{\rm bim}_0A'(\Sigma)}}
\def\A{{\cal A}}
\def\B{{\cal B}}
\def\D{{\cal D}}
\def\E{{\cal E}}
\def\F{{\cal F}}
\def\H{{\cal H}}
\def\I{{\cal I}}
\def\K{{\cal K}}
\def\L{{\cal L}}
\def\N{{\cal N}}
\def\M{{\cal M}}
\def\gM{{\frak M}}
\def\O{{\cal O}}
\def\P{{\cal P}}
\def\S{{\cal S}}
\def\T{{\cal T}}
\def\U{{\cal U}}
\def\V{{\mathcal V}}
\def\qed{\hfill$\square$}

\title{The Fourier-Stieltjes algebra of a C$^*$-dynamical system}

\author{Erik B\'edos,
Roberto Conti}

\date{\today}
\maketitle
\markboth{Erik B\'edos, Roberto Conti}{
}
\renewcommand{\sectionmark}[1]{}

\vspace{-3ex} \hspace{12ex}{\it  Dedicated to the memory of Ola Bratteli and Uffe Haagerup}

\medskip \begin{abstract}
In analogy with the Fourier-Stieltjes algebra of a group, we associate to a unital discrete twisted C$^*$-dynamical system a Banach algebra whose elements are coefficients of equivariant representations of the system. Building upon our previous work, we show that this Fourier-Stieltjes algebra embeds continuously in the Banach algebra of completely bounded multipliers of the (reduced or full) C$^*$-crossed product of the system.  We introduce a notion of positive definiteness and prove a Gelfand-Raikov type theorem allowing us to describe the Fourier-Stieltjes algebra of a system in a more intrinsic way. We also propose a definition of amenability for C$^*$-dynamical systems and show that it implies regularity. After a study of some natural commutative subalgebras, we end with a characterization of the Fourier-Stieltjes algebra involving C$^*$-correspondences over the (reduced or full) C$^*$-crossed product.


\end{abstract}

\section{Introduction} \label{Intro}

A famous result of Gelfand and Raikov \cite{GR} (see also \cite{G}) says that
a (complex-valued) continuous function on a locally compact group $G$ is positive definite if and only if it arises as 
a ``diagonal" coefficient function of a continuous unitary representation of $G$ on some Hilbert space. 
 The collection $P(G)$ of all such functions forms a cone in the space of continuous, bounded functions $C_{\rm b}(G)$, while its linear span $B(G)$, when equipped with the pointwise product, gives the so-called Fourier-Stieltjes algebra of $G$, as introduced in the seminal work of Eymard \cite{Eym} (where he also introduced the Fourier algebra $A(G)$). It is well-known that $B(G)$ admits a natural Banach space structure for which it is isometrically isomorphic to the 
dual space of the full group C$^*$-algebra $C^*(G)$ associated with $G$. Moreover, any element of $B(G)$ (resp.\ $P(G)$) induces  in a canonical way a completely bounded  (resp.\ completely positive) map  on $C^*(G)$ and on the reduced group C$^*$-algebra  $C^*_{\rm r}(G)$ (see for example \cite{DCHa, Wal1, Pis, Pis2}).
In other words, denoting by $M^{\rm u}_{\rm cb}(G)$ (resp.\ $M_0A(G)$) the space of completely bounded full (resp.\ reduced) multipliers of $G$, this means that $  B(G) \subset M^{\rm u}_{\rm cb}(G)$ and $ B(G) \subset M_0A(G)$.  In fact, we have $  B(G) = M^{\rm u}_{\rm cb}(G)$ \cite{Pis2}, while $B(G) = M_0A(G)$ holds if and only if $G$ is amenable \cite{DCHa, Bo}.

It is also worth mentioning that $A(G)$ and $B(G)$  are important ingredients in Walter's duality theory \cite{Wal}, also available for non abelian groups, that provides an alternative to other perhaps more popular approaches, 
e.g.\  Tannaka-Krein duality in the case of compact groups. Moreover, generalizations of Fourier-Stieltjes algebras (and Fourier algebras) have been introduced in other settings, e.g.\  for Kac algebras in \cite{DCES1, DCES2} and for groupoids in \cite{Re, RW, Oty, Pat}.

In some previous work \cite{BeCo3, BeCo4}, we have developed some aspects of classical Fourier theory
for the reduced C$^*$-crossed product $C_r^*(\Sigma)$ associated with a unital discrete twisted C$^*$-dynamical system $\Sigma=(A,G,\alpha, \sigma)$. We recall in Section 2 the definition of such systems and give a brief outline in subsection 2.1 of the construction of the associated (full and reduced) C$^*$-crossed products using Hilbert C$^*$-modules. As the equivariant representation theory of $\Sigma$ plays a major r{\^o}le in this article, similar to the one played by the unitary representation theory of a group, we give a short introduction to this topic in subsection 2.2, where we also indicate how to form direct sums and tensor products of equivariant representations. 

The main motivation of the present work is to present and discuss a natural candidate for the
Fourier-Stieltjes algebra $B(\Sigma)$ of  $\Sigma$. As a set, $B(\Sigma)$ consists of the $A$-valued coefficients of the equivariant representations of $\Sigma$. In Section 3 we prove that $B(\Sigma)$ may be organized as a unital Banach algebra in  a natural way and that $B(G)$ embeds continuously in $B(\Sigma)$. We also illustrate that  $B(\Sigma)$ is not commutative whenever $A\neq \Complessi$ satisfies some weak assumptions. 

Section 4 contains our main results. In subsection 4.1, we introduce the space $M^{\rm u}_{\rm cb}(\Sigma)$ of completely bounded full multipliers of $\Sigma$,
and show that  $B(\Sigma) \subset M^{\rm u}_{\rm cb}(\Sigma)$, that is, every element of $B(\Sigma)$ 
naturally gives rise to a completely bounded map
on $C^*(\Sigma)$. In the case of $C_r^*(\Sigma)$, the analogous result, saying that $B(\Sigma) \subset M_0A(\Sigma) $,  was already shown in \cite{BeCo4} using a version of Fell's absorption principle. The full case relies on the fact that one may form a kind of tensor product of an equivariant representation of $\Sigma$ with a covariant representation of $\Sigma$ to obtain  another covariant representation of $\Sigma$, as was shown in \cite{BeCo3}.  
In subsection 4.2, we propose a notion of $\Sigma$-positive definiteness for  $A$-valued functions defined on $G\times A$ that are linear in the second variable. This notion fits well with the general scheme, in the sense that we show that $T$ is $\Sigma$-positive definite if and only if $T$ arises as a ``diagonal" coefficient function of an equivariant representation of $\Sigma$ on some Hilbert $A$-module. 
This characterization provides a generalization of Gelfand and Raikov's result for  positive definite functions on $G$ (which is recovered by setting $A=\Complessi$)
and of the GNS-construction for completely positive maps from $A$ to itself \cite{Pas} (which follows by setting  $G=\{e\}$). It also gives that $B(\Sigma)$ coincides with the span of  $\Sigma$-positive definite functions, as in the group case. One should note that our concept of $\Sigma$-positive definiteness differs from the notion of 
positive definiteness (with respect to $\alpha$) defined earlier by Anantharaman-Delaroche \cite{AD1, AD2} for functions from $G$ to $A$. Such a function is positive definite in her sense if and only if it arises
as a ``diagonal" coefficient function of an $\alpha$-compatible action of $G$ on
some Hilbert $A$-module. However, in her approach, the link with completely positive maps on 
the associated crossed products remained somewhat elusive, except in some
special situations, as the one considered by Dong and Ruan \cite{DoRu} for functions from $G$ into the center of $A$. Our results clarify this connection. For example, if $T:G\times A \to A$ is of the form
$T(g,a) = \varphi(g)\, a$ for some function $\varphi:G\to A$, we get that $T$ is $\Sigma$-positive definite
if and only if $\varphi$ takes its values in the center of $A$ and $\varphi$ is positive definite  (with respect to $\alpha$), if and only if $\varphi$ induces a completely positive map on $C^*(\Sigma)$ (resp.\ $C_r^*(\Sigma)$). (See Proposition \ref{L-phi} for a precise statement).
We end this subsection by proposing a definition of amenability for $\Sigma$ and showing that it implies 
regularity, i.e.,
that $C^*(\Sigma)$ is canonically isomorphic to $C_{\rm r}^*(\Sigma)$. 
In subsection 4.3, we discuss some commutative subalgebras of $B(\Sigma)$, notably its center and also its subalgebra arising from coefficients of equivariant representations of $\Sigma$ associated with central vectors. On the other hand, as $B(\Sigma)$ itself sits in the space $M_0A(\Sigma)$ of completely bounded (reduced) multiplers of $\Sigma$, we show in subsection 4.4 that  $M_0A(\Sigma)$ has a natural structure of a Banach algebra with conjugation (in the sense of \cite{Wal1}) and that $B(\Sigma)$ is a subalgebra of $M_0A(\Sigma)$ that is closed under conjugation. We also consider some other subalgebras of $M_0A(\Sigma)$. In this subsection,   $M_0A(\Sigma)$ can be replaced with the space of completely bounded full multipliers $M_{\rm cb}^{\rm u}(\Sigma)$ if wishable.

In Section 5, we show that $B(\Sigma)$ may alternatively be described as ``localized'' coefficients of C$^*$-correspondences over the reduced crossed product $C_r^*(\Sigma)$ if one uses the canonical conditional expectation from $C_r^*(\Sigma)$ onto $A$ as a localization map. A similar result is also true if one considers the full crossed product $C^*(\Sigma)$ instead of $C_r^*(\Sigma)$. This means that $B(\Sigma)$ is fully determined by the reduced (or the full) C$^*$-crossed product associated with $\Sigma$, a fact which is not evident from the outset. 

In a subsequent paper \cite{BeCo6}, we plan to study how different notions of equivalence for C$^*$-dynamical systems are reflected in the associated Fourier-Stieltjes algebras. We also have in mind to study possible candidates for 
 the Fourier algebra of $\Sigma$ (see Remark \ref{fourier} for a tentative definition of $A(\Sigma)$). 
In another direction, one may wonder if it is possible to define a notion of conditionally negative definiteness
(w.r.t.\ $\Sigma$) and establish a Schoenberg type theorem, as in the classical case.
Our work leaves many other interesting questions open for future investigations and some of these problems are mentioned throughout the text. To avoid many technicalities, we only consider systems $\Sigma=(A,G,\alpha, \sigma)$ where $A$ is unital and $G$ is discrete in this paper. When $A$ is nonunital and  the cocycle $\sigma$ is assumed to take its values in the multiplier algebra of $A$, by making use of an approximative unit for $A$ 
whenever appropriate, it should not be problematic to generalize our results to this setting (for $G$ discrete). We also believe that the case where $A$ is a separable C$^*$-algebra and $G$ is a second countable locally compact group (as considered in \cite{PaRa, PaRa1}) should be possible to handle, although this will require a non-negligible amount of work. Our guiding thought in this respect has been that we had better demonstrate that our approach leads to a valuable theory in the case of unital discrete systems before eventually dealing with more general systems.  
 
As general references for the theory of C$^*$-algebras used in this article, we recommend
\cite{Bl} and \cite{BrOz}. Concerning notation and terminology, we list below a few items specific for this paper. If $G$ is a discrete group and $X$ is a (complex) vector space,  we will let $C_c(G,X)$ denote the vector space of functions from $G$ into $X$ with finite support. If $g\in G$ and $x\in X$, we will let $x\odot g$ denote the function in $C_c(G,X)$ which takes the value $x$ at $g$ and is equal to 0 otherwise.
We will only consider {\it unital}  C$^*$-algebras in this paper, and by a homomorphism 
between two such objects we will always mean a homomorphism that is unital and $*$-preserving.  Isomorphisms and automorphisms are therefore 
 also assumed to be $*$-preserving. The group of unitary elements in a C$^*$-algebra $A$ will be denoted by $\U(A)$, the center of $A$ by $Z(A)$, and the group of automorphisms of $A$  by ${\rm Aut}(A)$. The identity map on $A$ will be denoted by ${\rm id}$ (or ${\rm id}_A$). 

By a Hilbert C$^*$-module, we will always mean a  {\it right} Hilbert C$^*$-module. The reader should consult \cite{La1} for unexplained terminology and notation about such modules.  All inner products are assumed to be linear in the second variable, $\L(X, Y)$  will denote the space of all adjointable operators between two Hilbert C$^*$-modules $X$ and $Y$ over a C$^*$-algebra $B$, and $\L(X) = \L(X,X)$. A representation $\pi$ of a C$^*$-algebra $A$ on a Hilbert $B$-module $Y$ is then a homomorphism from $A$ into the C$^*$-algebra $\L(Y)$. 
Moreover, if $\mathcal{H}$ is a Hilbert space, then we will regard $Y\otimes \mathcal{H}$ as a Hilbert $B$-module, and $\pi\otimes \iota$ will denote the representation of $A$ on  $Y\otimes \mathcal{H}$ satisfying $(\pi\otimes \iota)(a) (y \otimes \xi) = \pi(a) \otimes \xi$ for $a\in A,\, y \in Y$ and $\xi \in \mathcal{H}$.

\section{Preliminaries}\label{Preliminaries} 

\subsection{C$^*$-dynamical systems and twisted crossed products}
Throughout this paper, the quadruple $\Sigma = (A, G, \alpha,\sigma)$ will  denote a {\it twisted 
unital discrete
C$^*$-dynamical system}.  This means that
$A$ is a  C$^*$-algebra with unit $1$, 
$G$ is a discrete group with identity $e$
and $(\alpha,\sigma)$ is a {\it twisted
action} of $G$ on $A$ (sometimes called a cocycle $G$-action on $A$), that is,
$\alpha$ is a map from $G$ into ${\rm Aut}(A)$ 
and  $\sigma$ is a map from $G \times G$ into $\, \U(A)$,
satisfying
\begin{align*}
\alpha_g \, \alpha_h & = {\rm Ad}(\sigma(g, h)) \,  \alpha_{gh} \\
\sigma(g,h) \sigma(gh,k) & = \alpha_g(\sigma(h,k)) \sigma(g,hk) \\
\sigma(g,e) & = \sigma(e,g) = 1 \ , 
\end{align*}
for all $g,h,k \in G$. Of course,  ${\rm Ad}(v)$ denotes here the (inner) automorphism  of $A$ implemented by some unitary $v$ in $\U(A)$.

 Some general references about such twisted systems and their associated C$^*$-crossed products are for example \cite{PaRa, PaRa1, BeCo3}.  If $\sigma$ is  trivial, that is, $\sigma(g,h)=1$ for all $g,h \in G$, then $\Sigma$ is an ordinary C$^*$-dynamical system (see e.g.\ \cite{Wi, BrOz, Ec}).
If $\sigma$ is {\it central}, that is, it takes values in  $\U(Z(A))$, then $\alpha$ is still an ordinary action of $G$ on $A$, and  this case is studied in \cite{ZM}. If $A = {\mathbb C}$, then we have $\alpha_g={\rm{id}}$ for all $g \in G$ and $\sigma$ is a 2-cocycle on $G$ with values in the unit circle $\mathbb{T}$,
 (see e.g.\ \cite{BeCo2} and references therein).    

We will use the cocycle equations listed above in various forms a number of times. The reader should be able to fill in the necessary steps that we often will omit. As a sample, we include the following result.

\begin{lemma} \label{cocy}
Let $g, h, h' \in G$. Then we have
$$\sigma(h, h^{-1}h') = \alpha_g^{-1}\big(\sigma(g, h') \, \sigma(gh, h^{-1}h')^*\, \sigma(g,h)^*)\big)\,.$$

\end{lemma}
\begin{proof}
We have
$$ \sigma(g,h) \, \sigma(gh, h^{-1}h') = \alpha_g \big(\sigma(h, h^{-1}h')\big)\, \sigma(g, h\,h^{-1}h')= \alpha_g \big(\sigma(h, h^{-1}h')\big)\, \sigma(g,h')\,.$$
Hence, $ \alpha_g \big(\sigma(h, h^{-1}h')\big) =  \sigma(g,h) \, \sigma(gh, h^{-1}h')\, \sigma(g,h')^*$\,, and the assertion follows by applying $\alpha_g^{-1}$.

\end{proof}
We give below a short exposition on the full twisted crossed product C$^*$-algebra
$C^*(\Sigma)$ 
and its  reduced version 
$C^*_{\rm r}(\Sigma)$
 that makes use of  Hilbert C$^*$-modules (see  e.g.\ \cite{BeCo3} for more details on this approach). 

A  {\it covariant homomorphism} of $\Sigma$ is a pair     
$(\pi,u)$, where $\pi$ is a homomorphism of $A$ into a
$C^*$-algebra $D$ and  $u$ is a map of $G$ into ${\mathcal U}(D)$, satisfying
$$u(g)\, u(h) = \pi(\sigma(g,h))\,  u(gh)$$ and
the covariance relation       
\begin{equation} 
\pi(\alpha_g(a)) = u(g) \, \pi(a) \, u(g)^*
\end{equation}
 for all $g,h \in G$ and $a \in A$.
If $D={\mathcal L}(Y)$ for some Hilbert $C^*$-module $Y$, then
$(\pi, u)$ is called a {\it covariant representation} of $\Sigma$ on $Y$.

\smallskip
For example, let $A^\Sigma$ be the Hilbert $A$-module 
$$A^\Sigma= \Big\{\xi: G \to A \ | \ \sum_{g \in G}  \alpha_g^{-1}\big(\xi(g)^*\,\xi(g)\big) \
\mbox{is norm-convergent in  $A$} \Big\} \, $$
where the $A$-valued inner product is given by
$$\langle\xi,\eta \rangle_\alpha = \sum_{g\in G} \alpha_g^{-1}\big(\xi(g)^* \eta(g)\big),$$
and  the  right action of $A$ is given by $(\xi \cdot a )(g) = \xi(g)\,\alpha_g(a)$. The (left) {\it regular covariant representation} $( \ell_\Sigma\, ,  \lambda_\Sigma)$ of  $\Sigma$ on $A^\Sigma$ is then given by
$$\big[\ell_\Sigma(a)\, \xi\big](h)  = a\; \xi(h)\,,$$ 
$$\big[\lambda_\Sigma(g)\, \xi\, \big](h) =
\alpha_g\big(\xi(g^{-1}h)\big)\, \sigma(g, g^{-1}h)\,$$
for $\xi \in A^\Sigma$ and $h\in G$. Other (unitarily equivalent) ways to define regular covariant representations of $\Sigma$ are discussed in \cite{BeCo3}.

The vector space $C_c(G,A)$ becomes a unital $*$-algebra, denoted by $C_c(\Sigma)$, when equipped with the operations
$$
(f_1 \ast f_2)\, (h)
= \sum_{g \in G} f_1(g) \, \alpha_g\big(f_2(g^{-1}h)\big) \, \sigma(g,g^{-1}h) \,,
$$
$$f^*(h) = \sigma(h, h^{-1})^*\, \alpha_h\big(f(h^{-1})\big)^*  \,. $$
Whenever  $(\pi, u)$ is  a covariant homomorphism of $\Sigma$ into $D$, the map
$\pi \times u : C_c(\Sigma) \to D$  defined by
$$
(\pi \times u)(f) = \sum_{g \in G} \pi(f(g))\, u(g) \quad \text{for all } f \in C_c(\Sigma) \,
$$
gives a homomorphism from $C_c(\Sigma)$ into $D$. 
In particular,  $\Lambda_\Sigma := \ell_\Sigma \times \lambda_\Sigma$ is  a representation of $C_c(\Sigma)$ on $A^\Sigma$, which is easily seen to be faithful.  Identifying $A$ with $\ell_\Sigma(A)$, as we will do in the sequel, we have
$$\Lambda_\Sigma(f) = \sum_{g\in G} \, f(g) \, \lambda_\Sigma(g)\, \quad \text{for}\,\,f \in C_c(\Sigma)\,.$$
The {\it full $C^*$-algebra} $C^*(\Sigma)$  is the C$^*$-algebra obtained by completing $C_c(\Sigma)$ with respect to the norm on $C_c(\Sigma)$ given by 
$$\|f\|_{\rm u} = \sup\big\{ \|(\pi\times u)(f)\| : (\pi, u)\, \text{ is a covariant homomorphism of }\Sigma\big\}\,.$$
We will identify $C_c(\Sigma)$ with its canonical copy inside $C^*(\Sigma)$.
 
 Whenever $(\pi, u)$ is a covariant homomorphism of $\Sigma$ into $D$, we may extend $\pi \times u$
in a unique way to a homomorphism from $C^*(\Sigma)$ into $D$, that will also be denoted by $\pi \times u$.
Let $i_A:A\to C^*(\Sigma)$ be the homomorphism given by $i_A(a) = a\odot e$ for $a \in A$ and let $i_G:G\to \U(C^*(\Sigma))$ be given by $i_G(g) = 1 \odot g$ for $g\in G$. Then $(i_A, i_G)$ is  a covariant homomorphism of $\Sigma$ into $C^*(\Sigma)$,  satisfying  $i_A\times i_G = {\rm id}_{C^*(\Sigma)}$. Moreover, if
  $\phi: C^*(\Sigma)\to D$ is a homomorphism, then $(\pi, u) := (\phi\circ i_A\,,\, \phi\circ i_G)$ is a covariant homomorphism of $\Sigma$ into $D$,  satisfying $\pi\times u = \phi$.

The {\it reduced $C^*$-algebra $C_{\rm r}^*(\Sigma)$}  is  defined as the $C^*$-subalgebra of ${\mathcal L}(A^\Sigma)$ given by
 $$C_{\rm r}^*(\Sigma) = \Lambda_\Sigma\big(C^*(\Sigma)\big) \,.$$
 The system  $\Sigma$ is called {\it regular} when $\Lambda_\Sigma$ is faithful on $C^*(\Sigma)$, in which case it provides an isomorphism from $C^*(\Sigma)$ onto $C_{\rm r}^*(\Sigma)$. 
 As is well-known, this happens when $G$ is amenable. For some more general conditions, see e.g.\ \cite{BeCo3} and references therein. 

\smallskip 
Let $\xi_0 \in A^\Sigma$ be defined by 
$\xi_0(e) = 1$ and $\xi_0(g) = 0$ for $g \neq e$. 
 For $x \in C_{\rm r}^*(\Sigma)$, we set $$\widehat{x} = x\, \xi_0\,  \in A^\Sigma\,.$$ 
The canonical conditional expectation $E$ from  $C_{\rm r}^*(\Sigma)$ onto $A$         
is  then given by $$E(x) = \widehat{x}(e) \, $$
 for  $x \in  C_{\rm r}^*(\Sigma)$.  In particular, we have $E\big(\Lambda_\Sigma(f)\big)= f(e)$
 for $f\in C_c(\Sigma)$. As is well-known, $E$ is faithful. It is also equivariant, that is, we have
$$E\big(\lambda_\Sigma(g) x \lambda_\Sigma(g)^*\big) = \alpha_g(E(x))\,$$
for $g\in G$ and $x \in  C_{\rm r}^*(\Sigma)$.

\subsection{Equivariant representations}
We recall from \cite{BeCo3, BeCo4} that  an {\it equivariant representation} 
of $\Sigma$ on a Hilbert $A$-module $X$ 
is  a pair $(\rho, v)$ where 
 $\rho : A \to \L(X)$ is a representation of $A$ on $X$ and  $v$ is a map from $G$ into the group $\mathcal{I}(X)$ consisting of all $\mathbb{C}$-linear, invertible, bounded maps from $X$ into itself, which  satisfy:
\begin{itemize}
\item[(i)]  \quad $\rho(\alpha_g(a))  = v(g) \, \rho(a) \, v(g)^{-1}\, , \quad  \quad g\in G\,, \,  a \in A$
\item[(ii)]  \quad$v(g)\, v(h)  = {\rm ad}_\rho(\sigma(g,h)) \, v(gh) \, , \quad   \quad g, h \in G$
\item[(iii)] \quad $\alpha_g\big(\langle x \, ,\, x' \rangle\big)   = \langle v(g) x\, ,\, v(g) x' \rangle\, , \quad  \quad g\in G\, , \, x,\,  x' \in X\,$ 
\item[(iv)]  \quad$v(g)(x \cdot a)  = (v(g) x)\cdot \alpha_g(a)\, , 
\quad  \quad \, g \in G,\, x\in X,\, a \in A$. 
\end{itemize}
In (ii) above, $ {\rm ad}_\rho(\sigma(g,h)) \in \mathcal{I}(X) $ is defined by
$${\rm ad}_\rho(\sigma(g,h)) \,x = \big(\rho(\sigma(g,h))\, x \big)\cdot \sigma(g,h)^* \, , \quad g, h \in G,\, x \in X. $$

The {\it central part of} $X$ (w.r.t.\ $\rho$) is defined by 
$$Z_X = \{z \in X \mid \rho(a)z = z\cdot a \, \, \text{for all} \, \, a \in A\}\, .$$
So ${\rm ad}_\rho(\sigma(g,h))$ is the identity operator when restricted to $Z_X$. 
We also note that property (iii) implies that each $v(g)$ is an isometry on $X$, and that $v$ is simply a unitary representation of $G$ on the Hilbert space $X$ when $A=\Complessi$.
Equivariant representations of $\Sigma$ may alternatively be described via $(\alpha,\sigma)$-$(\alpha, \sigma)$ compatible actions of $G$ on C$^*$-correspondences over $A$, in the spirit of  \cite{EKQR-0, EKQR} (see also Remark \ref{eqrep-corresp}).

The  {\it trivial equivariant representation} of $\Sigma$ is  the pair $(\ell, \alpha)$ acting on $A$, considered as an $A$-module over itself in the canonical way. 
If $(\rho,v)$ is an equivariant representation of $\Sigma$ on a Hilbert $A$-module $X$ and $w$ is a unitary representation of $G$ on some         
Hilbert space $\mathcal H$, then we can form the equivariant representation $(\rho \otimes \iota, v \otimes w)$ of $\Sigma$ on $X \otimes {\mathcal H}$, where $(v \otimes w)(g) \in \mathcal{I}(X \otimes {\mathcal H})$ is 
 determined by
$$[(v \otimes w)(g)](x\otimes \xi) = v(g)x \otimes w(g)\xi$$
for all $g\in G$, $x\in X$ and $\xi \in \mathcal{H}$.
For example, if $\lambda$ denotes the (left) regular representation of $G$ on $\ell^2(G)$, then $(\ell \otimes \iota, \alpha \otimes \lambda)$ gives the (left) {\it regular equivariant representation} of $\Sigma$ on $A^G:= A \otimes \ell^2(G)$. 

Tensoring an equivariant representation with a covariant representation is possible, see \cite[Section 4]{BeCo3}: 
if $(\rho,v)$ is an equivariant representation of $\Sigma$ on a Hilbert $A$-module $X$ and
$(\pi,u)$ is a covariant representation of $\Sigma$ on a Hilbert $B$-module $Y$,
then  $(\rho\dot\otimes\pi\,,\, v\dot\otimes u)$ is a covariant representation of
$\Sigma$ on the internal tensor product  Hilbert $B$-module $X\, {\otimes_\pi}Y$. We recall that 
on simple tensors in $X\, {\otimes_\pi}Y$, we have
$$(x \cdot a) \,\dot\otimes\, y = x \,\dot\otimes\, \pi(a) y\,,\quad (x \,\dot\otimes\, y)\cdot b = x \,\dot\otimes \,(y \cdot b)\,, $$
$$ \big\langle x \,\dot\otimes\, y\,,\, x' \,\dot\otimes\, y' \big\rangle 
= \big\langle y\,,\, \pi \big(\langle x,x'\rangle\big) y' \big\rangle\, 
 $$
 for $x \in X, y \in Y, a\in A$ and $b\in B$. We also note for further use that
 $\|x\, \dot\otimes\, y\| \leq \|x\|\, \|y\|$.
 The covariant representation $(\rho\dot\otimes\pi\,,\, v\dot\otimes u)$ is determined by
$$ \big[(\rho\dot\otimes\pi)(a)\big]\, (x\dot\otimes\, y) 
= \rho(a) x \, \dot\otimes \, y \ , \quad
\big[(v\dot\otimes u)(g)\big] \, (x\dot\otimes\, y) = v(g) x \, 
\dot\otimes \, u(g)y\, .$$

One can also form the tensor product of equivariant representations. Indeed,  
assume that $(\rho_1,v_1)$ and $(\rho_2,v_2)$ are equivariant representations 
 of $\Sigma$ on some Hilbert $A$-modules $X_1$ and $X_2$, respectively.
We will define their tensor product $(\rho_1,v_1) \otimes (\rho_2,v_2)$, 
as an equivariant representation of $\Sigma$ on the  internal tensor product Hilbert $A$-module 
$X_1 \otimes_{\rho_2} X_2$. This is achieved as follows. 

For  $a \in A$, we let $(\rho_1 \otimes \rho_2)(a)\in \L(X_1 \otimes_{\rho_2} X_2) $ be determined on simple tensors by
$$\big[(\rho_1 \otimes \rho_2)(a)\big] (x_1 \,\dot\otimes\, x_2) = \rho_1(a)x_1 \,\dot\otimes \, x_2 \quad \text{for } x_1 \in X_1,\, \text{and } x_2 \in X_2\,.$$
In other words, $(\rho_1 \otimes \rho_2)(a) = (\rho_2)_* (\rho_1(a))$ in the notation of  \cite{La1}. It is easily checked
that the associated map $\rho_1 \otimes \rho_2: A \to \L(X_1 \otimes_{\rho_2} X_2)$ is a representation 
of $A$ on $X_1 \otimes_{\rho_2} X_2$.

Next, for each $g \in G$, it is straightforward to verify that there exists a map
$(v_1 \otimes v_2)(g)$ in $\I(X_1 \otimes_{\rho_2} X_2)$ determined on simple tensors by
$$\big[(v_1 \otimes v_2)(g)\big](x_1 \dot\otimes x_2) = v_1(g)x_1 \,\dot\otimes \,v_2(g)x_2\quad  \text{for } x_1 \in X_1\text{  and } x_2 \in X_2\, .$$

\begin{proposition}\label{tenprodeqrep}
The pair $(\rho_1,v_1)\otimes(\rho_2,v_2):=(\rho_1 \otimes \rho_2, v_1 \otimes v_2)$ is an equivariant representation of $\Sigma$ on the Hilbert $A$-module $X_1 \otimes_{\rho_2} X_2$.
\end{proposition}

\begin{proof} 
We have to check that the properties (i)-(iv) hold for $(\rho_1 \otimes \rho_2,v_1 \otimes v_2)$. By linearity and density, it suffices to verify  these  on simple tensors.

(i): 
For all $a \in A$, $g \in G$, $x \in X_1$ and $ y \in X_2$ we have
\begin{align*}
(\rho_1 \otimes \rho_2)(\alpha_g(a))(v_1 \otimes v_2)(g) (x \,\dot\otimes\, y) & = 
\rho_1(\alpha_g(a))v_1(g)x \,\dot\otimes\, v_2(g)y \\
& = v_1(g)\rho_1(a)x \,\dot\otimes\, v_2(g)y\\
& = (v_1 \otimes v_2)(g) (\rho_1 \otimes \rho_2)(a)(x \,\dot\otimes\, y) \ ,
\end{align*}
where we have used property (i) for $(\rho_1,v_1)$;

(ii): 
For all $g,h \in G$, $x \in X_1$ and $y \in X_2$ we have
\begin{align*}
(v_1 \otimes v_2)(g) (v_1 \otimes v_2)(h) (x \,\dot\otimes\, y)
& = v_1(g)v_1(h)x \,\dot\otimes \,v_2(g)v_2(h)y) \\
& =  (\rho_1(\sigma(g,h))v_1(gh)x) \cdot \sigma(g,h)^* 
      \,\dot\otimes\, 
      (\rho_2(\sigma(g,h))v_2(gh)y) \cdot \sigma(g,h)^* \\
& =  \rho_1(\sigma(g,h))v_1(gh)x  
       \,\dot\otimes \,
      v_2(gh)y \cdot \sigma(g,h)^* \\
& = \big((\rho_1 \otimes \rho_2)(\sigma(g,h))(v_1 \otimes v_2)(gh)
        (x \,\dot\otimes\, y)\big) \cdot \sigma(g,h)^* \\
 & = {\rm ad}_{\rho_1 \otimes \rho_2}(\sigma(g,h))(v_1 \otimes v_2)(gh)
        (x \,\dot\otimes\, y)\,;
\end{align*}

(iii):
For all $g \in G$, $x,x' \in X_1$ and $y,y' \in X_2$ we  have
\begin{align*}
\alpha_g\big(\big\langle x \,\dot\otimes\, y\,,\, x' \,\dot\otimes\, y'\big\rangle\big)
& = 
\alpha_g\big( \big\langle y, \rho_2 \big(\langle x,x'\rangle\big) y' 
         \big\rangle\big) \\
& = \big\langle v_2(g)y\,,\, v_2(g)\rho_2 \big(\langle x,x'\rangle\big) y'
         \big\rangle \\
& = \big\langle v_2(g)y\,,\, \rho_2 \big(\alpha_g(\langle x,x'\rangle)\big) 
         v_2(g)y' \big\rangle \\
& = \big\langle v_2(g)y\,,\, \rho_2 \big(\langle v_1(g)x,v_1(g)x'\rangle) 
         v_2(g)y' \big\rangle \\
& = \big\langle v_1(g)x \,\dot\otimes\, v_2(g)y\,,\,v_1(g)x' \,\dot\otimes\, v_2(g)y' 
        \big\rangle \\
& = \big\langle (v_1 \otimes v_2)(g)(x \,\dot\otimes\, y)\,,\,(v_1 \otimes v_2)(g)
        (x' \,\dot\otimes\, y')\big\rangle \,,
\end{align*}
where we have used both (iii) and (i) for $(\rho_2,v_2)$ 
and (iii) for $(\rho_1,v_1)$;

(iv) 
For all $a \in A$, $g \in G$, $x \in X_1$ and $ y \in X_2$ we have
\begin{align*}
(v_1 \otimes v_2)(g) \big((x \,\dot\otimes\, y)\cdot a\big) 
& = (v_1 \otimes v_2)(g) (x \,\dot\otimes\, y \cdot a) \\ 
& =  v_1(g)x \,\dot\otimes\, v_2(g)(y \cdot a) \\ 
& =  v_1(g)x \,\dot\otimes\, (v_2(g)y) \cdot \alpha_g(a) \\ 
& = \big((v_1 \otimes v_2)(g)(x \,\dot\otimes\, y)\big) \cdot \alpha_g(a) \, ,
\end{align*}
where we have used (iv) for $(\rho_2,v_2)$.
\end{proof}

Finally, we will need to form the direct sum of equivariant representations. Assume $\{(\rho_i, v_i)\}_{i \in I}$ is an indexed family of equivariant representations of $\Sigma$, with each $(\rho_i,v_i)$ acting on a Hilbert $A$-module $X_i$. Let then $X=\oplus_{i\in I} X_i$ denote the direct sum of the $X_i$'s, as defined for example in \cite{La1}. In particular, the inner product on $X$ is given by  
$$\langle x \, ,\, x' \rangle =  \sum_{i\in I}  \, \langle x_i \, ,\, x'_i \rangle $$
for $x=(x_i)\, , \, x'=(x'_i)  \in X$, the sum being norm-convergent in $A$.
We can then define $\rho=\oplus_{i\in I}\,\rho_i:A\to \L(X)$ and $v=\oplus_{i\in I}\,v_i:G\to\I(X)$ by 
$$\rho(a) \,x = (\rho_i(a)\, x_i)\, , \quad v(g) \, x = (v_i(g)\, x_i)$$ 
for $a \in A$, $g \in G$ and $x=(x_i) \in X$. It is easy to check that $\rho$ and $v$ are well defined and satisfy all the required properties. For example, for $x=(x_i)\, , \, x'=(x'_i)  \in X$ and $g\in G$, using continuity of $\alpha_g$ and property (iii) for each $(\rho_i, v_i)$, we have
\begin{align*}
\alpha_g\big(\langle x \, ,\, x' \rangle\big) &= \sum_{i\in I} \, \alpha_g\big(\langle x_i \, ,\, x_i \rangle\big) 
= \sum_{i\in I} \, \langle v_i(g) x_i\, ,\, v_i(g) x_i' \rangle\,\\ 
&= \langle v(g) x\, ,\, v(g) x' \rangle\, .
\end{align*}
which shows that property (iii) holds for $(\rho, v)$.

\section{The Fourier-Stieltjes algebra of $\Sigma$}\label{FS-algebra}

 We recall \cite{Eym, Pis} that $B(G)$ is a commutative unital Banach $*$-algebra w.r.t.\ to the norm given by letting $\|\varphi\|$ denote the infimum of the set of values  $\|\xi\|\, \|\eta\|$ obtained from the possible decompositions of $\varphi$ of the form $\varphi(g) = \langle \xi,v(g)\, \eta\rangle $ where $v$ is a unitary representation of $G$ on a Hilbert space $\H$ and $\xi\,,\,\eta \in \H$. 
 
 To define the Fourier-Stieltjes algebra $B(\Sigma)$ in a similar manner, we first introduce some notation. 
Let  $(\rho,v)$ be an equivariant representation of $\Sigma$ on a Hilbert $A$-module $X$ and let $x, y \in X$.    
Then we define $T_{\rho,v,x,y}:G\times A \to A$ by
$$T_{\rho,v,x,y}(g,a) = \big\langle x,\rho(a) \, v(g) \, y \big \rangle  \quad \text{for} \,\,a \in A, \, g \in G \,,$$ 
and think of $T_{\rho, v, x, y}$ as an  $A$-valued coefficient function associated with $(\rho, v)$. 
\begin{definition}
We let $B(\Sigma)$  denote the set of all maps from $G\times A$ into $A$ of the form $T_{\rho, v, x, y}$ for some equivariant representation $(\rho,v)$ of $\Sigma$ on a Hilbert $A$-module $X$ and  $x, y \in X$\,. 
\end{definition}
As  we are going to organize $B(\Sigma)$  as an algebra, we call $B(\Sigma)$ for the  {\it Fourier-Stieltjes algebra  of $\Sigma$}. 
Let $L(\Sigma)$ consist of all the maps from $G\times A$ into $A$ that are linear in the second variable, and equip $L(\Sigma)$ with its natural algebra structure:
for $T, T' \in L(\Sigma)$ and $\lambda \in \Complessi$,  we let 
$T+T', \, \lambda T$\,, $ T\times T'$  and $I$ be the maps in $L(\Sigma)$  defined by 
\begin{align*}
(T+T')(g,a) & := T(g,a)+T'(g,a) \\
(\lambda T) (g,a) & := \lambda \, T(g,a) \\
(T \times T')(g,a) & := T(g,T'(g,a)) \\
I(g,a) & := a
\end{align*}
for $g\in G$ and $a \in A$. Given $T\in L(\Sigma)$ and $g \in G$, we will often write $T_g$ for the linear map from $A$ into itself given by $T_g(a)=T(g, a)$ for $a\in A$.
Note that, for instance, we have $(T+T')_g = T_g + T'_g$ and $(T\times T')_g = T_g \circ T'_g$ for all $g\in G$, so it is almost obvious that $L(\Sigma)$ becomes a unital algebra with respect to these operations.

It is clear that $B(\Sigma)$ is a subset of $L(\Sigma)$. In fact, we have:
\begin{lemma} \label{alg}
  $B(\Sigma)$ is a unital subalgebra of $L(\Sigma)$.
\end{lemma}
\begin{proof}
We first  note that 
$I = T_{\ell, \alpha, 1, 1} \in B(\Sigma)$. 
Next, it is evident that  $B(\Sigma)$ is closed under multiplication by scalars. Moreover, if for $i=1,2$,
 $(\rho_i,v_i)$ 
are equivariant representations of $\Sigma$ on Hilbert $A$-modules $X_i$ and  $x_i,\, y_i \in X_i$, then we have
\begin{align*}
T_{\rho_1, v_1, x_1, y_1} \,+ \,T_{\rho_2, v_2, x_2, y_2} &\,= \,T_{\rho_1\oplus\rho_2,\, v_1\oplus v_2, \,x_1\oplus x_2, y_1\oplus y_2}\,,\\
T_{\rho_1, v_1, x_1, y_1}\, \times\, T_{\rho_2, v_2, x_2, y_2} &\,=\, T_{\rho_2\otimes\rho_1,\, v_2\otimes v_1, \,x_2\dot\otimes x_1, y_2\dot\otimes y_1}\,,
 \end{align*}
which clearly implies that $B(\Sigma)$ is closed under addition and multiplication. 
For example, for $g \in G$ and $a \in A$, we have
\begin{align*} \big(T_{\rho_1, v_1, x_1, y_1}\, \times\, T_{\rho_2, v_2, x_2, y_2}\big)(g,a) &=
T_{\rho_1, v_1, x_1, y_1}\big(g, \,  T_{\rho_2, v_2, x_2, y_2}(g,a)\big)\\
&=  \big\langle x_1\,,\, \rho_1\big(\langle x_2\,,\, \rho_2(a)v_2(g)y_2 \rangle\big)v_1(g)y_1 \big\rangle \\
& = \big\langle x_2 \,\dot\otimes\, x_1\,, \,\rho_2(a)v_2(g)y_2 \,\dot\otimes \,v_1(g)y_1\big\rangle\\
& =\big\langle x_2 \,\dot\otimes x_1\,,(\rho_2 \otimes \rho_1)(a)(v_2 \otimes v_1)(g)       
(y_2 \,\dot\otimes\, y_1) \big\rangle \\ 
&= \big(T_{\rho_2\otimes\rho_1,\, v_2\otimes v_1\,, \,x_2\dot\otimes x_1, y_2\dot\otimes y_1}\big)(g,a)\,.
\end{align*}
\end{proof}
As with $B(G)$, it is not difficult to see that we get a norm on $B(\Sigma)$ by letting  $\|T\|$ denote the infimum of  the set of values $\|x\|\, \|y\|$ associated with the possible decompositions of $T$ of the form  $T = T_{\rho, v, x,y}$. 
Moreover, we have:

\begin{proposition} $B(\Sigma)$ is a unital Banach algebra w.r.t.\ $\|\cdot\|$.
\end{proposition}

\begin{proof} 
From Lemma \ref{alg} and our comment above, we know that $B(\Sigma)$ is a unital  algebra and a normed space w.r.t.\ $\|\cdot\|$. We now show that $B(\Sigma)$ is complete. Assume that $\{T_i\}_{i\in \Naturali}$ is a sequence of non-zero elements in $B(\Sigma)$ such that $\sum_{i=1}^\infty \, \|T_i\| \, < \, \infty$. We have to show that $\sum_{i=1}^\infty \, T_i $ is norm-convergent in $B(\Sigma)$.  For each $i$, we may pick an equivariant representation $(\rho_i, v_i) $ of $\Sigma$ on $X_i$ and $x_i, y_i\in X_i$ such that $T_i = T_{\rho_i, v_i, x_i, y_i}$ and $$\|x_i\|\, \|y_i\|  \, <\, \|T_i\|+ 1/2^i\,.$$
Without loss of generality, we may assume that $\|x_i\| = \|y_i\| $ (by replacing $x_i$ with $\sqrt{\frac{\|y_i\|}{\|x_i\|}}\, x_i$ and $y_i$ with $\sqrt{\frac{\|x_i\|}{\|y_i\|}} \,y_i$ if necessary), so we have
\begin{equation} \label{ineq-Ti}
\|x_i\|^2 = \|y_i\|^2 \, <\, \|T_i\|+ 1/2^i\
\end{equation} for each $i$. Let $X=\oplus_{i=1}^\infty \, X_i$ and note that 
$x=(x_i)$ and $y=(y_i)$ both belong to $X$ since, for instance, we have
$$\sum_{i=1}^\infty \, \|\langle x_i, x_i\rangle\| = \sum_{i=1}^\infty \|x_i\|^2 \, \leq \,  \sum_{i=1}^\infty \, \big(\|T_i\| + 1/2^i\big)\, =\,  \Big(\sum_{i=1}^\infty \,\|T_i\|\Big) + 1 \,<\, \infty\,,$$
so the series $\sum_{i=1}^\infty \, \langle x_i, x_i\rangle$ is (absolutely) convergent in $A$. 

\smallskip We may now let $(\rho, v)$ be the equivariant representation of $\Sigma$ on $X$ given by  $\rho = \oplus_{i=1}^\infty\,\rho_i $\,,\, $v= \oplus_{i=1}^\infty \, v_i$\,,  and set $T= T_{\rho, v, x, y} \in B(\Sigma)$. Then we claim that $\sum_{i=1}^\infty \, T_i $ converges to $T$ in $B(\Sigma)$. Indeed, for each $n \in \Naturali$, setting $$X'_n = \oplus_{i=n+1}^\infty \, X_i\, , \, \rho'_n =\oplus_{i=n+1}^\infty \,\rho_i\, , \, v'_n = \oplus_{i=n+1}^\infty \, v_i\, , \, x'_n= ( x_i)_{i=n+1}^\infty \,, \, y'_n = ( y_i)_{i=n+1}^\infty\,,$$ one easily checks that   
$T-\sum_{i=1}^n \, T_i \,= T_{\rho'_n, v'_n, x'_n, y'_n}\,.$ Hence, we get
$$\big\|T-\sum_{i=1}^n \, T_i \big\|^2 \, \leq \, \|x'_n\|^2\, \|y'_n\|^2 = \big\|\sum_{i=n+1}^\infty \langle x_i, x_i\rangle\big\|\, \,  \big\|\sum_{i=n+1}^\infty \langle y_i, y_i\rangle\big\|$$
$$\leq \sum_{i=n+1}^\infty \|\langle x_i, x_i\rangle\|\, \cdot \,  \sum_{i=n+1}^\infty \|\langle y_i, y_i\rangle\|
=  \sum_{i=n+1}^\infty \| x_i\|^2\, \cdot \,  \sum_{n+1}^\infty \| y_i\|^2 $$
$$\leq \,  \Big(   \sum_{i=n+1}^\infty \, \big(\|T_i\| + 1/2^i\big) \Big)^2 \,\, \to \, 0 \, \, \text{    as } n \to \infty$$
since  $\sum_{i=1}^\infty \, (\|T_i\| + 1/2^i) $ is convergent.

\smallskip Next we show that $\|\cdot \| $ is an algebra-norm. Let $T_1, \, T_2 \in B(\Sigma)$. If $T_j= T_{\rho_j, v_j, x_j, y_j}$ for $j=1, 2$, then $T_1\times T_2 = T_{\rho_2\otimes\rho_1,v_2\otimes v_1 , x_2\otimes x_1, y_2\otimes x_1}$, so $$\|T_1\times T_2\| \leq \|x_2\otimes x_1\|\, \|y_2\otimes y_1\|\, \leq \|x_1\| \|y_1\|\, \| x_2\|\|y_2\|$$
Hence, taking the infimum, first over all possible decompositions of $T_1$\,, next over all possible decompositions of $T_2$, we get $\|T_1\times T_2\|\,\leq \, \|T_1\|\,\| T_2\|$\,, as desired. 

\smallskip As $I= T_{\ell, \alpha, 1, 1}$\,, we have $\|I\| \leq \|1\| \,\|1\| \leq 1$, hence $\|I\|=1$ since the converse inequality always holds. 
\end{proof}

There exists a canonical way of embedding $B(G)$ into $B(\Sigma)$.

\begin{proposition} For $\varphi \in B(G)$, define $T^\varphi \in L(\Sigma)$ by $T^\varphi(g, a) = \varphi(g) \, a$ for $g\in G$ and $a \in A$. 
Then $T^\varphi \in B(\Sigma)$ and the map $\varphi \to T^\varphi $ gives an injective, continuous, algebra-homomorphism of $B(G)$ into $B(\Sigma)$.
\end{proposition}
\begin{proof} Let us pick  a unitary representation $w$ of $G$ on a Hilbert space $\mathcal H$ and $\xi,\eta \in {\mathcal H}$ such that $\varphi(g) = \langle \xi\, ,\, w(g)\, \eta\rangle$ for all $g\in G$. 
Considering $(\rho, v) = (\ell \otimes \iota, \alpha \otimes w)$ on $X=A \otimes {\mathcal H}$ and
 $x=1\otimes \xi\, , \, y = 1 \otimes \eta$, we get
$$
T_{\rho, v , x, y} (g,a) = \big\langle 1\, ,\,  a \, \alpha_g(1) \big\rangle \, \langle \xi, w(g)\eta\rangle 
= \langle \xi, w(g)\eta \rangle \, a = \varphi(g)\, a\,.
$$
Hence, $ T^\varphi = T_{\rho, v , x, y}  \in B(\Sigma)$. Moreover, this implies that $\|T^\varphi\| \leq \|x\| \|y\| = \|\xi\|\|\eta\|$. Taking the infimum over all possible decompositions of $\varphi$, we get $\|T^\varphi\| \leq \|\varphi\|$. As the map $\varphi \to T^\varphi$ is obviously an injective algebra-homomorphism, we are done.
\end{proof}

\begin{remark} \label{centvect} 
For later use, we note that the vectors $x=1\otimes \xi$ and $y=1\otimes \eta$ in the above proof are both central, i.e., they lie in the central part of $X=A\otimes \mathcal{H}$ (with respect to $\rho = \ell \otimes \iota$). Indeed, $\rho(a) (1\otimes \xi) = a\otimes \xi = (1\otimes \xi)\cdot a$ for all $a \in A$.
\end{remark}

\begin{remark}
If $\Sigma = (\Complessi, G, {\rm id}, \sigma)$ where $\sigma \in Z^2(G,\Toro)$ is a normalized 2-cocycle on $G$, then it is easy to see that the map $\varphi \to T^\varphi$ is an isomorphism from $B(G)$ onto $B(\Sigma)$. We will usually identify $B(G)$ with $B(\Sigma)$ in this case.  
\end{remark}

\begin{remark}
We note that if $A$ is noncommutative, then  $B(\Sigma)$ is noncommutative as well. Indeed, 
let us consider the trivial equivariant representation $(\ell,\alpha)$ on $X=A$ and let $b, \, b' \in A$. Then we have 
$T_{\ell, \alpha, b, 1}(g, a) = b^*a\alpha_g(1) = b^*a$, so 
$$(T_{\ell, \alpha, b', 1} \times T_{\ell, \alpha, b, 1})(g, a) = T_{\ell, \alpha, b', 1}(g, b^*a)= b'^*b^* a = (bb')^*a\,$$
for all $g \in G $ and $a \in A$. Similary, we get
$$(T_{\ell, \alpha, b, 1} \times T_{\ell, \alpha, b', 1})(g, a) = (b'b)^*a\,$$
for all $g \in G $ and $a \in A$. Hence we see that $T_{\ell, \alpha, b', 1} $ and $T_{\ell, \alpha, b, 1}$ commute in $B(\Sigma)$ if and only if $bb'=b'b$.

On the other hand, one may wonder whether $B(\Sigma)$ can be commutative when $A$ is commutative (and nontrivial). We show below that this is often not the case. 

Let $(\rho, v)$ denote an equivariant representation of $\Sigma$ on a Hilbert $A$-module $X$ and assume there exists $\beta \in {\rm Aut}(A)$ such that $\beta\circ\alpha_g=\alpha_g\circ\beta$ for all $g\in G$ and $\beta(\sigma(g,h)) = \sigma(g,h)$ for all $g,h \in G$. Then it is easy to check that $(\rho\circ \beta, v)$ is also an equivariant representation of $\Sigma$ on $X$. Now, if  $\beta$ is nontrivial, so there exists $b \in A$ such that $\beta(b)\neq b$, we may set $T= T_{\ell, \alpha, b, 1}$\,,\, $T'= T_{\ell\circ\beta, \alpha, 1, 1} $. For all $g\in G$, we get $(T\times T') (g,1) = b^*\,,$  while $\,(T'\times T) (g,1) = \beta(b)^*\,.$ Hence, $T\times T'\neq T'\times T$ in this case. So $B(\Sigma)$ will be noncommutative whenever some $\beta$ satisfying the above assumptions exists, even if $A$ is commutative. 
This happens for example when $\alpha$ and $\sigma$ are trivial and $A$ has a nontrivial automorphism. 
However, if we take $G= {\rm Aut}(A)$ and $\sigma$ is trivial, it may happen that $G$ has a trivial center (e.g.\ $A= C(\Omega)$ where $\Omega $ denotes the Cantor set), so this observation can not be applied. 
\end{remark}

\section{Multipliers and positive definiteness}

\subsection{Multipliers} 

Let $T\in L(\Sigma)$. 
For each $f \in C_c(\Sigma)$,  define $ T\cdot f \in C_c(\Sigma)$ by 
$$\big(T\cdot f\big)(g) = T(g,f(g)) \,  \quad\text{for all } g\in G\,.$$
In other words, $\big(T\cdot f\big)(g)= T_g\big(f(g)\big)$ when $g\in G$.
We recall from \cite{BeCo2} that $T$ is called a ({\it reduced}) {\it multiplier} of $\Sigma$ whenever
 there exists a  bounded linear map
$M_T: C_{\rm r}^*(\Sigma) \to C_{\rm r}^*(\Sigma)$ such that
$$M_T\big(\Lambda_\Sigma(f)\big) = \Lambda_\Sigma(T\cdot f)\,$$ 
for all $ f \in C_c(\Sigma)$. 
   
We let  $MA(\Sigma)$ denote the set consisting of all (reduced) multipliers of $\Sigma$.
The subset of $MA(\Sigma)$ consisting of all  {\it completely bounded} (reduced) multipliers,  that is, of multipliers satisfying $\|M_T\|_{\rm cb} < \infty$, is denoted by $M_0A(\Sigma)$.

\smallskip As an example, consider $\varphi: G \to {\mathbb C}$ and define $T^\varphi(g,a)=\varphi(g) a$ for $g\in G $ and $a \in A$.
If $T^\varphi \in MA(\Sigma)$, then $\varphi \in MA(G)$. The converse statement holds when $A=\Complessi$ \cite{BeCo5}, but we don't know whether it is true in general. Anyhow, it can be shown \cite[Corollary 4.7]{BeCo3} that
$T^\varphi \in M_0A(\Sigma)$ if and only if  $\varphi \in M_0A(G)$.

\smallskip  One may associate completely bounded (reduced) multipliers to elements of $B(\Sigma)$ (cf.\ \cite[Theorem 4.8]{BeCo3}):

\begin{theorem}\label{equiv-coeff}
Let  $(\rho,v)$ be an equivariant representation of $\Sigma$ on a Hilbert $A$-module $X$ and let $x, y \in X$.    
Set $T= T_{\rho,v,x,y} \in B(\Sigma)$. Then $T \in M_0A(\Sigma)$,  with $\| M_T \|_{\rm cb}   \leq \|x\| \, \|y\|$.
Moreover, if $x=y$, then $M_T$ is completely positive, with $\| M_T \|_{\rm cb} = \|x\|^2\,.$
 \end{theorem}

It follows that if $T\in B(\Sigma)$, then $T \in M_0A(\Sigma)$, with $\|M_T\|_{\rm cb} \leq \|T\|$.
The following partial converse of Theorem \ref{equiv-coeff} holds.

\begin{theorem} \label{conv-red-mult}
Assume $T$ is a $($reduced$)$ multiplier of $\Sigma$ such that $M_T$ is completely positive. Then $T= T_{\rho, v, x, x}$ for some 
equivariant representation $(\rho, v)$ of $\Sigma$ on a Hilbert $A$-module $X$ and $x \in X$. 
\end{theorem}

\begin{proof}
We will use the KSGNS construction for completely positive maps \cite[Theorem 5.6]{La1}. We set $B= C_{\rm r}^*(\Sigma)$ and consider $B$ as a (right) Hilbert  module over itself in the natural way. Let $L: B \to \L(B)$ denote the left multiplication map and set $\widetilde{M_T} = L\circ M_T: B \to \L(B)$. Then $\widetilde{M_T}$ is completely positive, so the KSGNS construction provides us with a Hilbert $B$-module $Y$  and a representation $\pi: B \to \L(Y)$: one first defines a  $B$-valued semi-inner product on the (right) $B$-module $B\odot B$, given on simple tensors by 
$$\langle b_1\otimes c_1 \, , \, b_2 \otimes c_2\rangle_B  = c_1^{\,*} \,\big[\widetilde{M_T}\big( b_1^{\,*} \,b_2 \big)\big]\, (c_2) = 
c_1^{\,*} \,M_T( b_1^{\,*} \,b_2)\, c_2$$
for $b_1, b_2, c_1, c_2 \in B$, and let $Y$ be the completion of  the pre-Hilbert  $B$-module $(B\odot B)/N$, where $N=\{ z \in B\odot B \mid \langle z, z\rangle_B=0\}.$  Writing $b\,\dot\otimes \,c$ for the coset $(b\otimes c) + N$, the representation $\pi$ is then determined by $\pi (d)(b\,\dot\otimes \,c) = db\,\dot\otimes \,c$  for $b,c, d \in B$. Now, we may localize $Y$ using $E$, that is, we consider $Y$ as a pre-Hilbert $A$-module with respect to $\langle y, z \rangle_A = E\big(\langle y, z \rangle_B\big)$ and let $X$ denote the Hilbert $A$-module obtained after completion. Then we let $\rho:A\to \L(X)$ be the representation  determined by $\rho(a)\, y = \pi(a) \, y$ for $a\in A$ and $ y \in Y$, and let $v:G\to \I(X)$ be determined by
$$v(g) \,y = \pi(\lambda_\Sigma(g)) \big(y\cdot \lambda_\Sigma(g)^*\big)\,.$$
It is easy to check that $\rho$ and $v$ are well-defined, and that $(\rho, v)$ is an equivariant representation of $\Sigma$ on $X$. (A more general result is proven in Proposition \ref{rv}). 

 \smallskip Set $x = 1\dot\otimes 1 \in X$ and let $a\in A$ and $g\in G$. Then  we have 
$$v(g)\, x =   \pi(\lambda_\Sigma(g)) \big((1\dot\otimes 1) \cdot \lambda_\Sigma(g)^*\big)= \lambda_\Sigma(g)\dot\otimes\lambda_\Sigma(g)^*\,,$$
so 

\vspace{-3 ex} $$  \rho(a)\,v(g)\, x = \pi(a) \big(\lambda_\Sigma(g)\dot\otimes\lambda_\Sigma(g)^*\big) = (a\,\lambda_\Sigma(g))\dot\otimes\lambda_\Sigma(g)^*\,,$$ which gives
\begin{align*} \big\langle x\,, \,\rho(a)\,v(g)\, x\big\rangle_B &=
 \langle 1\dot\otimes 1\, , (a\,\lambda_\Sigma(g))\dot\otimes\lambda_\Sigma(g)^* \rangle_B\\
 &= 
 M_T\big(a\lambda_\Sigma(g)\big)\,\lambda_\Sigma(g)^* = T_g(a) \, \lambda_\Sigma(g)\,\lambda_\Sigma(g)^* \\
 &= T_g(a)\,.
 \end{align*}
 
 \vspace{-2 ex}\noindent Hence we get

\vspace{-2 ex}
$$\big\langle x\,, \,\rho(a)\,v(g)\, x\big\rangle_A = E \big(\langle x\,, \,\rho(a)\,v(g)\, x\rangle_B\big)= E(T_g(a)) = T_g(a)\,, $$
as desired.
\end{proof}

\begin{definition}
Let $T\in L(\Sigma)$.  We say that $T$ is a  {\it full multiplier} of $\Sigma$ whenever
 there exists a  bounded linear map
$\Phi_T: C^*(\Sigma) \to C^*(\Sigma)$ such that
$$\Phi_T\big(f\big) =T\cdot f\,$$
for every $f \in C_c(\Sigma)$. If $\Phi_T$ is completely bounded, then we say that $T$ is  a {\it completely bounded full multiplier of $\Sigma$}. The set of all full  (resp.\  completely bounded full) multipliers of $\Sigma$ will be denoted by $M^{\rm u}(\Sigma)$ (resp.\ $M^{\rm u}_{\rm cb}(\Sigma)$).
\end{definition}

The following result is known for full group C$^*$-algebras \cite{Wal1,Pis2}, i.e.\ when $A=\Complessi$, $\alpha={\rm id}$ and $\sigma=1$. 

\begin{theorem}\label{equiv-coeff-full}
Let  $(\rho,v)$ be an equivariant representation of $\Sigma$ on a Hilbert $A$-module $X$ and let $x, y \in X$.    
Set $T=T_{\rho,v, x, y}\in B(\Sigma)$. Then $T$ is a completely bounded full multiplier of $\Sigma$ satisfying $\| \Phi_T \|_{\rm cb}   \leq \|x\| \, \|y\|$. 
Moreover, if $x=y$, then $\Phi_T$ is completely positive, with $\| \Phi_T \|_{\rm cb} = \|x\|^2\,.$
 \end{theorem} 
 
 \begin{proof} 
 Let $\Pi$ be a faithful representation of $C^*(\Sigma)$ on a Hilbert space $\H$ and write $\Pi = \pi\times u$ for a   covariant representation $(\pi, u)$ of $\Sigma$ on $\H$. Note that the pair
 $(\rho\dot\otimes\pi\,,\, v\dot\otimes u)$ is then a  covariant representation of
$\Sigma$ on  the Hilbert space $X\, {\otimes_\pi}\H$. 

For $z\in X$, let $V_z: \H \to X\, {\otimes_\pi}\H$ denote the bounded operator
determined  by $$V_z(\xi) = z \dot\otimes \xi$$ for $\xi \in \H$ and note that its adjoint operator $V_z^*$ satisfies $$V_z^*(z'\dot\otimes\eta) = \pi(\langle z, z'\rangle)\, \eta$$ for $z' \in X$ and $\eta \in \H$. 

Let then $\psi: C^*(\Sigma) \to \L(X\, {\otimes_\pi}\H)$ be the completely bounded linear map given by 
$$\psi(b)= V_x^*\,[(\rho\dot\otimes\pi)\times (v\dot\otimes u)] (b)\, V_y$$ 
for $b\in C^*(\Sigma)$, which satisfies
$$\| \psi \|_{\rm cb} \leq \|V_x^*\| \, \|V_y\| = \|x\|\, \|y\|\,,$$
see for example \cite{Pau}.

Consider  now $f\in C_c(\Sigma)$. We claim that $\Pi(T\cdot f)=\psi(f)$. Indeed, for all $\xi, \eta \in \H$, we have
 \begin{align*}  
 \big\langle \xi\,, \, \Pi(T\cdot f)\,\eta\big\rangle &= \big\langle \xi\,, \, (\pi\times u)(T\cdot f)\,\eta\big\rangle  \\
&= \sum_{g\in G} \, \big\langle \xi \,,\, \pi\big(\langle x, \rho(f(g))v(g)y\rangle\big)u(g) \eta\big\rangle \\
 & =\sum_{g\in G} \, \big\langle x\dot\otimes \xi \, , \,\rho(f(g))v(g)y\,\dot\otimes \,u(g) \eta\big\rangle \\
 &=  \sum_{g\in G} \, \big\langle x\dot\otimes \xi \, , \, (\rho\dot\otimes\pi) (f(g)) \, (v\dot\otimes u)(g) (y\,\dot\otimes \eta)\big\rangle \\
& =\big\langle x\dot\otimes \xi \, ,\, [(\rho\dot\otimes\pi)\times (v\dot\otimes u)](f) \,(y\dot\otimes \eta)\big\rangle \\
&=\big\langle  \xi \, ,\, V_x^*\,[(\rho\dot\otimes\pi) \times (v\dot\otimes u)](f)\, V_y\,\eta \big\rangle \\ 
&=\big\langle  \xi \, ,\, \psi(f)\,\eta \big\rangle\,.
\end{align*}
Hence, we get that $$\|T\cdot f\|_{\rm u} = \|\Pi(T\cdot f)\| = \| \psi(f)\| \leq \|\psi\|\,\|f\|_u\,.$$ This shows that the linear map $f \to T\cdot f$ is bounded, and can therefore be extended to a bounded linear map $\Phi_T$ from $C^*(\Sigma)$ into itself. By continuity of the involved maps and density of $C_c(\Sigma)$ in $C^*(\Sigma)$,  we get $\Pi\circ \Phi_T = \psi$. It follows that the range of $\psi$ is contained in the range of $\Pi$, and we may therefore write $\Phi_T = \Pi^{-1}\circ \psi$, where $\Pi^{-1}$ denotes the inverse of the isomorphism $\Pi:C^*(\Sigma)\to \Pi(C^*(\Sigma))$. As $\Pi^{-1}$ and $\psi$ are both completely bounded,  $\Phi_T$ is also completely bounded, with 
$$\| \Phi_T\|_{\rm cb}  \leq \|\Pi^{-1}\|_{\rm cb}\, \| \psi \|_{\rm cb} \leq  \|x\|\, \|y\|\,,$$
as asserted. 

If $x=y$, then $\Pi^{-1}$ and $\psi$  are both completely positive, so $\Phi_T$ is also completely positive, with 
$$\|\Phi_T\|_{\rm cb} = \|\Phi_T(I)\| = \|V_x^*V_x\| = \|\pi(\langle x, x\rangle)\| = \|\langle x, x\rangle\| = \|x\|^2$$
(as $\pi= \Pi\circ i_A$ is faithful).
 \end{proof}
 
The following partial converse of Theorem \ref{equiv-coeff-full} holds.

\begin{theorem} \label{conv-full-mult}
Assume $T$ is a full multiplier of $\Sigma$ such that $\Phi_T$ is completely positive. Then $T= T_{\rho, v, x, x}$ for some equivariant representation $(\rho, v)$ of $\Sigma$ on a Hilbert $A$-module $X$ and $x \in X$. 
\end{theorem}

The proof is similar to the one for Theorem \ref{conv-red-mult}. One may for example use the KSGNS construction for the completely positive map $\widetilde{\Phi_T}=L\circ \Lambda_\Sigma\circ\Phi_T:C^*(\Sigma) \to \L\big(C_{\rm r}^*(\Sigma)\big)$ (where  $L: C_{\rm r}^*(\Sigma) \to \L\big(C_{\rm r}^*(\Sigma)\big)$ is the left multiplication map).
\begin{remark} For completeness, we mention that one may also consider so-called {\it rf-multipliers} of $\Sigma$ (cf.\ \cite{BeCo3}): if $T \in L(\Sigma)$, then $T$ is a  {\it rf-multiplier} of $\Sigma$ whenever
 there exists a  bounded linear map
$\Psi_T: C_{\rm r}^*(\Sigma) \to C^*(\Sigma)$ such that
$\Psi_T\big(\Lambda_\Sigma(f)\big) = T\cdot f\,$
for every $f \in C_c(\Sigma)$. 
These are used in \cite{BeCo3} in the formulation of the weak approximation property for $\Sigma$, which ensures that $\Sigma$ is regular.
\end{remark}

 \subsection{Positive definiteness} 

 \begin{definition}\label{pd}
 Let $T\in L(\Sigma)$. We say that $T$ is {\it positive definite} (w.r.t.\ $\Sigma$), or that $T$ is {\it $\Sigma$-positive definite}, when for any $n\in \Naturali$, $g_1, \ldots, g_n \in G$ and $a_1, \ldots, a_n \in A$, the matrix
  $$\Big[ \,\alpha_{g_i}\Big(T_{g_i^{-1}g_j}\big(\alpha_{g_i}^{-1}\big(a_i^*\,a_j\,\sigma(g_i, g_i^{-1}g_j)^*\big)\,\big)\Big)\, \sigma(g_i, g_i^{-1}g_j)\,\Big]$$
is positive in $M_n(A)$ (the $n\times n$ matrices over $A$). 
 \end{definition}
 
When $A=\Complessi$, any $T\in L(\Sigma)$ satisfies  $T_g(\lambda) = \lambda\,\varphi(g)$, where $\varphi(g) = T_g(1)$, for $g \in G$ and $\lambda\in \Complessi$, and we see that $T$ is $\Sigma$-positive definite  if and only if $\varphi$ is positive definite. Another motivation for Definition \ref{pd} comes from the following example.
 
  \begin{example} \label{EQRPD} Let  $(\rho,v)$ be an equivariant representation of $\Sigma$ on a Hilbert $A$-module $X$ and let $x \in X$.    
Set $T=T_{\rho,v, x, x}\in B(\Sigma)$. Then $T$ is positive definite (w.r.t.\ $\Sigma$).

\smallskip 
To check this, we first consider $a,\, b \in A$ and $g, \,h \in G$. Then, using the properties of $(\rho, v)$, we get
 
 \medskip\noindent \hspace{2ex}
 $\alpha_{g}\Big(T_{g^{-1}h} \big(\alpha_{g}^{-1}\big(a^*\,b\,\sigma(g, g^{-1}h)^*\big)\,\big)\Big)\, \sigma(g, g^{-1}h)
 $
 
 \vspace{-4ex}
 \begin{align*}
\hspace{4ex} &= \alpha_{g}\big(\big\langle \,x , \, \rho\big(\alpha_{g}^{-1}\big(a^*\,b\,\sigma(g, g^{-1}h)^*\big)\big)\, 
v(g^{-1}h)\, x\,\big\rangle \big)\, \sigma(g, g^{-1}h)\\
&=\big\langle v(g)\,x\, , \, \rho\big(a^*\,b\,\sigma(g, g^{-1}h)^*\big)\,v(g)\, 
v(g^{-1}h)\, x\,\big\rangle \, \sigma(g, g^{-1}h)\\
&=\big\langle \rho(a)\, v(g)\,x\, , \, \rho(b)\,\rho\big(\sigma(g, g^{-1}h)^*\big)\,v(g)\, 
v(g^{-1}h)\, x\,\big\rangle \, \sigma(g, g^{-1}h)\\
&=\big\langle \rho(a)\, v(g)\,x\, , \,
\big(\rho(b)\,v(h)\, x\big)\cdot\sigma(g, g^{-1}h)^*\,\big\rangle \, \sigma(g, g^{-1}h)\\
&=\big\langle \rho(a)\, v(g)\,x\, , \,\rho(b)\,v(h)\, x\,\big\rangle\,.
\end{align*}
Now, let $n\in \Naturali$, $g_1, \ldots, g_n \in G$,  $a_1, \ldots, a_n \in A$ and set 
$$T_{ij} = \alpha_{g_i}\Big(T_{g_i^{-1}g_j}\big( \alpha_{g_i}^{-1}\big(a_i^*\,a_j\,\sigma(g_i, g_i^{-1}g_j)^*\big)\,\big)\Big)\, \sigma(g_i, g_i^{-1}g_j)\,$$
for each $1\leq i, j \leq n$.
Then, by the above computation, we get
$$\big[T_{ij}\big] = \big[\langle x_i\,, x_j\rangle\big]\,,$$ where $x_i = \rho(a_i)\, v(g_i)\, x$. Hence, $[\,T_{ij}\,]$ is positive in $M_n(A)$ (cf.\ \cite[Lemma 4.2]{La1}). 
\end{example}

\begin{remark}\label{PD-conj}
Assume $T \in L(\Sigma)$ is positive definite (w.r.t.\ $\Sigma$). Then  we have $$T_g(a)^* = \Big(\big(\alpha_g\circ T_{g^{-1}}\circ \alpha_g^{-1}\big)(a^*\sigma(g, g^{-1})^*\big)\Big)\, \sigma(g, g^{-1})$$ for all $g\in G$ and $a\in A$. This follows easily after plugging $g_1=e, \, g_2=g, \, a_1=1$ and $a_2 = a$ in the definition and looking at the off-diagonal terms. 
\end{remark}

  \begin{remark}
 Assume $T \in L(\Sigma)$ is positive definite (w.r.t.\ $\Sigma$) and set $\psi =T_e$\,. Thus, $\psi:A\to A$ is the linear map given by $\psi(a) = T(e, a)$ for $a \in A$. Then $\psi$ is completely positive. 
 Indeed, for any given $a_1, \ldots, a_n \in A$,  plugging $g_1=\cdots=g_n=e$ in the definition of positive definiteness of $T$ gives
 that $\big[\psi( a_i^*a_j)\big] = \big[T(e, a_i^*a_j)\big] $ is positive in $M_n(A)$. As is well known, this is equivalent to $\psi$ being completely positive. 
  \end{remark}

  \begin{example}\label{cpequi}
  Let $\theta: A\to A$ be a completely positive map which is $\alpha$-equivariant, i.e., satisfies that $\theta\circ\alpha_g = \alpha_g \circ\theta$ for all $g\in G$.   Let $\Theta: G\times A \to A $ be given by $\Theta(g,a) = \theta(a)$ for all $(g,a)$ in $G\times A$. 
  
  Assume first that $\sigma$ is scalar-valued. Then $\Theta$ is $\Sigma$-positive definite. Indeed, for all $g,h\in G$, we have
   $$\alpha_{g}\Big(\Theta_{g^{-1}h} \big(\alpha_{g}^{-1}\big(a^*\,b\,\sigma(g, g^{-1}h)^*\big)\,\big)\Big)\, \sigma(g, g^{-1}h) = \theta\big(a^*\,b \, \sigma(g, g^{-1}h)^*\big)\sigma(g, g^{-1}h)= \theta\big(a^*\,b )\,,$$
   so the $\Sigma$-positive definiteness of $\Theta$ follows readily from the complete positivity of $\theta$.

If $\sigma$ is not scalar-valued, let us assume that $\theta$ also satisfies that $\theta\circ\sigma = \sigma$. Then $\Theta$ is $\Sigma$-positive definite. Indeed, the above computation can still be carried out, now using that all the $\sigma(g,h)$'s lie in the multiplicative domain of $\theta$, as easily follows from the extra assumption. 
  \end{example}
  
   Another connection to completely positive maps is the following: 
   
  \begin{proposition} \label{M-cp} Assume $M:C_{\rm r}^*(\Sigma) \to C_{\rm r}^*(\Sigma)$ is a completely positive linear map. Define $T_M \in L(\Sigma)$ by 
 $$T_M(g, a) = E\big(\,M(a\lambda_\Sigma(g))\, \lambda_\Sigma(g)^*\big)$$
 for $g \in G$ and $a \in A$. Then $T_M$ is positive definite $($w.r.t.\ $\Sigma$$)$.
 \end{proposition}
 
 \begin{proof}
 Let $n\in \Naturali$,\, $g_1, \ldots, g_n \in G$,\, $a_1, \ldots, a_n \in A$ and set 
 $$T_{ij} = \alpha_{g_i}\Big(T_M\big( g_i^{-1}g_j\,, \, \alpha_{g_i}^{-1}\big(a_i^*\,a_j\,\sigma(g_i, g_i^{-1}g_j)^*\big)\big)\Big)\, \sigma(g_i, g_i^{-1}g_j)\,.$$
 Using the definition of $T_M$ and the properties of $E$, we get
 $$T_{ij} = E\Big(\lambda_\Sigma(g_i) \, M\Big(\alpha_{g_i}^{-1}\big(a_i^*\,a_j\,\sigma(g_i, g_i^{-1}g_j)^*\big) \lambda_\Sigma(g_i^{-1}g_j)\Big)\,\lambda_\Sigma(g_i^{-1}g_j)^*\lambda_\Sigma(g_i)^*\Big)\, \sigma(g_i, g_i^{-1}g_j)\,$$
 $$= E\Big(\lambda_\Sigma(g_i) \, M\Big(\alpha_{g_i}^{-1}\big(a_i^*\,a_j\,\sigma(g_i, g_i^{-1}g_j)^*\big) \lambda_\Sigma(g_i^{-1}g_j)\Big)\,\lambda_\Sigma(g_j)^*\sigma(g_i, g_i^{-1}g_j)^*\Big)\, \sigma(g_i, g_i^{-1}g_j)\,$$
$$= E\Big(\lambda_\Sigma(g_i) \, M\Big(\alpha_{g_i}^{-1}\big(a_i^*\,a_j\,\sigma(g_i, g_i^{-1}g_j)^*\big) \lambda_\Sigma(g_i^{-1}g_j)\Big)\,\lambda_\Sigma(g_j)^*\Big)\,,$$
that is, $T_{ij} =  E\Big(\lambda_\Sigma(g_i) \, M( A_{ij})\, \lambda_\Sigma(g_j)^*\Big)\,,$ where 
 \begin{align*}
  A_{ij} &= \alpha_{g_i}^{-1}\big(a_i^*\,a_j\,\sigma(g_i, g_i^{-1}g_j)^*\big) \lambda_\Sigma(g_i^{-1}g_j)\\
  &= \lambda_\Sigma(g_i)^*a_i^*\,a_j\,\sigma(g_i, g_i^{-1}g_j)^*\, \lambda_\Sigma(g_i)\, \lambda_\Sigma(g_i^{-1}g_j)\, \\
 &= \lambda_\Sigma(g_i)^*a_i^*\,a_j\,\lambda_\Sigma(g_j)\,.
\end{align*}
Hence, $[A_{ij}] = [z_i^*\, z_j]$\,, where $z_i = a_i\,\lambda_\Sigma(g_i) \in C_{\rm r}^*(\Sigma)$. So $[A_{ij}]$ is positive in $M_n(C^*_{\rm r}(\Sigma))$. As $M$ is completely positive, this implies that
$B=[ M(A_{ij})]$ is also positive in $M_n(C^*_{\rm r}(\Sigma))$. Then, setting $C = {\rm diag}\big(\lambda_\Sigma(g_1), \ldots, \lambda_\Sigma(g_1)\big)$, we get that
$$\Big[\lambda_\Sigma(g_i) \, M\big( A_{ij}\big)\, \lambda_\Sigma(g_j)^*\Big]= C\,  B\, C^*$$ 
 is positive in $M_n(C^*_{\rm r}(\Sigma))$. Finally, as $E$ is completely positive, we get that 
$$[\,T_{ij}\,] =  \Big[\,E\big(\lambda_\Sigma(g_i) \, M\big( A_{ij}\big)\, \lambda_\Sigma(g_j)^*\big)\,\Big]\,$$
 is positive in $M_n(C^*_{\rm r}(\Sigma))$, as desired.
\end{proof}

 \begin{corollary} \label{M-cpCor} Assume that $T$ is a (reduced) multiplier of $\Sigma$ such that  $M_T$ is completely positive. Then $T$ is  positive definite $($w.r.t.\ $\Sigma$$)$.
 \end{corollary}
 
 \begin{proof} For all $g\in G$ and $a \in A$, we have
$$T_{M_T}(g,a) = E\big(M_T(a\lambda_\Sigma(g))\lambda_\Sigma(g)^*\big) = E(T_g(a)) =T_g(a)\,.$$
Hence, using Proposition \ref{M-cp}, we get that $T = T_{M_T}$ is positive definite $($w.r.t.\ $\Sigma$$)$.
\end{proof}

We also have:

\begin{proposition} \label{phi-cp} Assume $\Phi:C^*(\Sigma) \to C^*(\Sigma)$ is a completely positive linear map. Define $T_\Phi \in L(\Sigma)$ by 
 $$T_\Phi(g, a) = (E\circ\Lambda_\Sigma)\big(\,\Phi\big(i_A(a)i_G(g)\big)\, i_G(g)^*\big)$$
 for $g \in G$ and $a \in A$. Then $T_\Phi$ is positive definite $($w.r.t.\ $\Sigma$$)$.
 \end{proposition}
 
The proof of Proposition \ref{phi-cp} is similar to the proof of Proposition \ref{M-cp}, so we skip it.

\begin{corollary} \label{phi-cpCor} Assume that $T$ is a full multiplier of $\Sigma$ such that $\Phi_T$ is completely positive. Then $T$ is  positive definite $($w.r.t.\ $\Sigma$$)$.
 \end{corollary}
 
 \begin{proof} For all $g\in G$ and $a \in A$, we have
$$T_{\Phi_T}(g,a) = (E\circ\Lambda_\Sigma)\big(\Phi_T(i_A(a)i_G(g))\,i_G(g)^*\big) = 
E\Big(\Lambda_\Sigma\big(i_A(T_g(a))\big)\Big) = E(T_g(a)) =T_g(a)\,.$$
Hence, using Proposition \ref{phi-cp}, we get that $T = T_{\Phi_T}$ is positive definite (w.r.t.\ $\Sigma$).
\end{proof}

Here is a Gelfand-Raikov type theorem, showing that the converse of Example \ref{EQRPD} holds. 

\begin{theorem} \label{GR} Let $T\in L(\Sigma)$ be  positive definite $($w.r.t.\ $\Sigma$$)$. 
    
   \smallskip  Then there exist  an equivariant representation $(\rho,v)$ of $\Sigma$ on a Hilbert $A$-module $X$ and $x \in X$ such that   
$T=T_{\rho,v, x, x}$.

 \smallskip The vector $x \in X$ may  be chosen to be cyclic for $(\rho, v)$, that is, in such a way that $$\text{{\rm Span}}\big\{ \big(\rho(a)\, v(g) \,x\big)\cdot b\mid a, \,b\in A\, ,\, g\in G\,\big\}$$ is dense in $X$. 
 If we then also have $T=T_{\rho',v', x', x'}$ for some  equivariant representation $(\rho',v')$ of $\Sigma$ on a Hilbert $A$-module $X'$ and some $x' \in X'$ which is cyclic for $(\rho', v')$, then the triple $(\rho', v', x')$ is unitary equivalent to the triple $(\rho, v, x)$, in the sense that there exists a unitary $u \in \L(X, X')$ such that $$\rho'(a) = u\, \rho(a)\, u^*\, \, , \, v'(g) = u\, v(g) \, u^*\, \, , \text{ and }\,\, u\, x = x'$$ for all $a\in A$ and $g\in G$.      
    \end{theorem}
    
     \begin{proof}  
For $g, h\in G$ and $a,b \in A$, we define $\big[(g,a), (h,b)\big]_T \in A$ by
    $$ \big[(g,a), (h,b)\big]_T=  \alpha_{g}\Big(T_{g^{-1}h} \big(\alpha_{g}^{-1}\big(a^*\,b\,\sigma(g, g^{-1}h)^*\big)\,\big)\Big)\, \sigma(g, g^{-1}h)\,.$$
    The assumption that $T$ is positive definite (w.r.t.\ $\Sigma$) gives that the $A$-valued map $$\big((g,a), (h,b)\big) \to \big[(g,a), (h,b)\big]_T$$  is  a positive definite kernel on $G\times A$ in the sense of \cite{BBLS}, so it has a Kolgomorov decomposition on a certain inner product $A$-module (see \cite[Section 3]{BBLS}). As this module plays an important r{\^o}le in our proof, we provide the details of
its construction in our situation.   

We first note that 
\begin{equation}\label{eq-conjsym}
 \big(\big[(g,a), (h,b)\big]_T\big)^* =  \big[(h,b), (g,a)\big]_T\,.
 \end{equation}
 Indeed,  choosing $g_1=g, \, g_2=h, \, a_1=a$ and $a_2 = b$ in the definition of the positive definiteness of $T$ gives a positive matrix in $M_2(A)$ whose  off-diagonal terms are $\big[(g,a), (h,b)\big]_T$ and $\big[(h,b), (g,a)\big]_T$\,. 

\smallskip Moreover, we observe that
\begin{equation}\label{eq-adj}
\big[(g,ac), (h,b)\big]_T\big] = 
 \alpha_{g}\Big(T_{g^{-1}h} \big(\alpha_{g}^{-1}\big(c^*a^*\,b\,\sigma(g, g^{-1}h)^*\big)\,\big)\Big)\, \sigma(g, g^{-1}h)\,
 = \big[(g,c), (h,a^*b)\big]_T\,.
 \end{equation}
 
  \smallskip 
  Now, we set $X_0 = C_c(G\times A\,, A)$ 
  where the first copy of $A$ carries the discrete topology
  and consider $X_0$ as an $A$-module w.r.t.\  the right action of $A$ on $X_0$ given by $$(F\cdot c)(h,b) = F(h,b)\,c$$ for $F\in X_0, \, h \in G$ and $ b, \,c \in A$.  For $F, F' \in X_0$, we define $\big\langle \,F, F'\,\big\rangle_T \in A$ by 
  $$ \big\langle \,F,F'\,\big\rangle_T \, = \sum_{(g,a) , \,(h,b) \,\in\, G\times A} \, F(g,a)^*\, \big[(g,a), (h,b)\big]_T\, F'(h,b)\,.$$
  Clearly, the map $(F, F')\to \big\langle \,F,F'\,\big\rangle_T$ is linear in the second variable.
  Further, using (\ref{eq-conjsym}), we get 
  $$ \big(\big\langle \,F,F'\,\big\rangle_T\big)^* = \sum_{(g,a) , \,(h,b) \,\in\, G\times A} \, F'(h,b)^*\, \big(\big[(g,a), (h,b)\big]_T\big)^*\, F(g,a)\,$$
  $$= \sum_{(g,a) , \,(h,b) \,\in\, G\times A} \, F'(h,b)^*\, \big[(h,b), (g,a)\big]_T\, F(g,a) = \big\langle \,F',F\,\big\rangle_T\,. $$
  Moreover, for each $c \in A$, we have
$$ \big\langle \,F,F'\cdot c\,\big\rangle_T = \sum_{(g,a) , \,(h,b) \,\in\, G\times A} \, F(g,a)^*\, \big[(g,a), (h,b)\big]_T\, F'(h,b)\, c = \big\langle \,F,F'\,\big\rangle_T \ c\,.$$
  When $g\in G$ and $a\in A$,  we will let $\delta_{(g,a)} \in X_0$ denote the function that takes the value 1 at $(g,a)$ and is 0 otherwise. Consider $0\neq F \in X_0$, with supp$(F)=\{(g_1, a_1), \ldots, (g_n,a_n)\}$ (without repetition) and $F(g_i, a_i) = b_i$ for $i=1, \ldots, n$.  
  We then have $$F = \sum_{i=1}^n \, \delta_{(g_i\,,\, a_i)}\cdot b_i\,.$$ 
  Somewhat pedantically, we will say that this is the {\it standard decomposition} of $F$.
    If $F=0$, we just take $0=\delta_{(e,1)} \cdot 0$ as its standard decomposition. Clearly, we have
$$\big\langle \,F,F\,\big\rangle_T = \sum_{i,j=1}^n b_i^* \, \big[(g_i,a_i)\, , \, (g_j, a_j)\big]_T\, b_j \,.$$
As $T$ is positive definite, the matrix 
$$\Big[ \,\big[(g_i,a_i)\, , \, (g_j, a_j)\big]_T \,\Big] = \Big[\alpha_{g_i}\Big(T_{g_i^{-1}g_j}\big(\alpha_{g_i}^{-1}\big(a_i^*\,a_j\,\sigma(g_i, g_i^{-1}g_j)^*\big)\,\big)\Big)\, \sigma(g_i, g_i^{-1}g_j)\Big]$$ is positive in $M_n(A)$. Hence,  it follows that $\big\langle \,F,F\,\big\rangle_T$ is a positive element of $A$\,.

\smallskip Thus, $X_0$ becomes a semi-inner-product $A$-module w.r.t.\ $ \big\langle \,\cdot, \cdot\,\big\rangle_T$\,.   
   Setting $$N = \big\{ F\in X_0\mid \big\langle \,F, F\,\big\rangle_T=0\big\}\,,$$ the quotient $X_0/N$ becomes an inner-product $A$-module w.r.t.\ $$ \big\langle \,F + N\,, F' + N\,\big\rangle:= \big\langle \,F,F'\,\big\rangle_T\,.$$
   We let $X$ denote the Hilbert $A$-module obtained by completing $X_0/N$.
  
   \smallskip 
   For each $a\in A$ and $F \in X_0$ with standard decomposition $F = \sum_{i=1}^n \, \delta_{(g_i\,,\, a_i)}\cdot b_i$\,,
  we define $\rho_0(a)\, F \in X_0$ by 
   $$\rho_0(a) \,F = \sum_{i=1}^n \, \delta_{(g_i\,,\, a\,a_i)}\cdot b_i\,.$$
   It is a  routine exercise to check that the map  $\rho_0(a):X_0\to X_0$ we obtain in this way for each $a \in A$ is linear (in fact $A$-linear). Moreover,  $\rho_0(a)$ is adjointable, with $\rho_0(a)^* = \rho_0(a^*)$. Indeed, for $F$ as above and $F' \in X_0$ with standard decomposition $F'= \sum_{j=1}^m \, \delta_{(g'_j\,,\, a'_j)}\cdot b'_j$\,, we use our previous observation (cf.\ equation (\ref{eq-adj})) to get
  $$\big\langle \rho_0(a) \,F\, ,\, F'\big\rangle_T = \sum_{i=1}^n\sum_{j=1}^{m} \, b_i^* \, \big[(g_i, aa_i), (g'_j, a'_j)\big]_T\, b'^*_j$$
  $$= \sum_{i=1}^n\sum_{j=1}^{m} \, b_i^* \, \big[(g_i, a_i), (g'_j, a^*a'_j)\big]_T\, b'^*_j= \big\langle F\, ,\, \rho_0(a^*) \,F'\big\rangle_T\,.$$
  In particular, we have $\big\langle \rho_0(a) \,F\, ,\,  \rho_0(a) \,F\big\rangle_T = \big\langle F\, ,\,  \rho_0(a^*) \,\rho_0(a) \,F\big\rangle_T $, and it then follows from the Cauchy-Schwarz inequality that $\rho_0(a)\, F$ lies in $N$ whenever $F\in N$. Hence, for each $a\in A$, we may define an adjointable linear map $\tilde\rho_0(a):X_0/N\to X_0/N$ by 
  $$\tilde\rho_0(a) (F+N) = (\rho_0(a)\,F) + N$$ 
  for each $F\in X_0$, which satisfies $\tilde\rho_0(a)^* = \tilde\rho_0(a^*)$.  One easily  checks that $\tilde\rho_0(1) = I$ and $\tilde\rho_0(aa') = \tilde\rho_0(a)\, \tilde\rho_0(a') $ for all $a, a' \in A$.
  
  \smallskip  We now show that $\tilde\rho_0(a)$ is bounded for each $a\in A$. Let $a\in A$ and choose $b \in A$ be such that $b^*b = \|a\|^2\cdot 1 - a^*a $\,. For $ \, F \in X_0$, let us write $\dot{F}= F+N$. Then we have
$$ \|a\|^2 \, \big\langle \dot{F}, \dot{F}\big \rangle- \big\langle \tilde\rho_0(a)\,\dot{F},\, \tilde\rho_0(a)\,\dot{F}\big\rangle \ = \big\langle \dot{F}, \,\tilde\rho_0\big(\|a\|^2\cdot 1 - a^*a\big)\, \dot{F}\big\rangle = \big\langle \tilde\rho_0(b)\,\dot{F}, \tilde\rho_0(b)\,\dot{F}\big\rangle \geq 0\,.
$$
So $\big\langle \tilde\rho_0(a)\,\dot{F}, \tilde\rho_0(a)\,\dot{F}\big\rangle \, \leq \, \|a\|^2 \,\big \langle \dot{F}, \dot{F}\big\rangle$, which implies that $$\|\tilde\rho_0(a)\, \dot{F}\|\,  \leq \, \|a\| \, \|\dot{F}\|\,.$$
Hence, $\tilde\rho_0(a)$ is bounded on $X_0/N$ for each $a\in A$ and it is straightforward to see that $\tilde\rho_0$ extends to a representation $\rho:A \to \L(X)$. 

\smallskip
Next, for $g\in G$ and $F \in X_0$ with standard decomposition $F = \sum_{i=1}^n \, \delta_{(g_i\,,\, a_i)}\cdot b_i\,$, we define $v_0(g)\, F \in X_0$ by 
    $$v_0(g)\, F = \sum_{i=1}^n \, \delta_{\big(g\,g_i\,,\, \alpha_g(a_i)\,\sigma(g,g_i)\big)}\cdot \big(\sigma(g, g_i)^*\,\alpha_g(b_i)\big)\,.$$
    It is routine to verify that the map $v_0(g): X_0\to X_0$ we obtain in this way is $\Complessi$-linear.   
    Moreover, for any $F\in X_0$, we have 
\begin{equation} \label{v0} \alpha_g\big(\langle F\,,\, F \rangle_T\big)  = \big\langle\, v_0(g)\,F,\, v_0(g)\,F\,\big\rangle_T 
\,. 
\end{equation}

\smallskip To prove this, using linearity, it suffices to check that 
\begin{equation} \label{v0-2} 
 \alpha_g\big(\,\big\langle \delta_{(h,a)}\cdot b\,,\, \delta_{(h',a')}\cdot b' \big\rangle_T\,\big)
=\big\langle\, v_0(g)\,\big(\delta_{(h,a)}\cdot b)\,,\, v_0(g)\,\big(\delta_{(h',a')}\cdot b'\big)\,\big\rangle_T 
 \end{equation}
for $g,h, h'\in G$ and $a, a', b , b' \in A$. We have

\medskip $ \alpha_g\big(\,\big\langle \delta_{(h,a)}\cdot b\,,\, \delta_{(h',a')}\cdot b' \big\rangle_T\,\big)
=  \alpha_g\big(\, b^*\,\big[ (h,a) \,,\, (h',a')\big]_T\,b'\,\big)$
$$= \alpha_g(b)^*\, \alpha_g\Big(\big((\alpha_h\, T_{h^{-1}h'}\, \alpha_h^{-1})(a^*a'\,\sigma(h, h^{-1}h')^*\big) \, \sigma(h, h^{-1}h')\Big)\, \alpha_g(b')$$ 
$$= \alpha_g(b)^*\, \sigma(g,h)\,\Big((\alpha_{gh}\, T_{h^{-1}h'}\, \alpha_h^{-1})\big(a^*a'\,\sigma(h, h^{-1}h')^*\big)\Big)\, \sigma(g,h)^* \, \alpha_g(\sigma(h, h^{-1}h'))\, \alpha_g(b')$$
$$= \alpha_g(b)^*\, \sigma(g,h)\,\Big((\alpha_{gh}\, T_{h^{-1}h'}\, \alpha_h^{-1})\big(a^*a'\,\sigma(h, h^{-1}h')^*\big)\Big)\, \sigma(gh, h^{-1}h') \, \sigma(g, h')^*\, \alpha_g(b')\,,$$
while

\medskip $\big\langle\, v_0(g)\,\big(\delta_{(h\,,\,a)}\cdot\, b)\,,\, v_0(g)\,\big(\delta_{(h'\,,\,a')}\cdot \,b'\big)\,\big\rangle_T $
$$= \big\langle\,\delta_{(gh,\,\alpha_g(a)\,\sigma(g,h))}\cdot\, \sigma(g,h)^*\alpha_g(b)\,,\,\delta_{(gh',\,\alpha_g(a')\,\sigma(g,h'))}\cdot\,\sigma(g,h')^*\, \alpha_g(b')\,\big\rangle_T $$
$$= \alpha_g(b)^*\, \sigma(g,h)\, \Big[\,\big(gh,\,\alpha_g(a)\,\sigma(g,h)\big)\,,\,\big(gh',\,\alpha_g(a')\,\sigma(g,h')\big)\Big]_T \, \sigma(g, h')^*\, \alpha_g(b')$$
$$= \alpha_g(b)^*\, \sigma(g,h)\, \Big(\big(\alpha_{gh}\,T_{h^{-1}h'}\,\alpha_{gh}^{-1}\big)\big(\sigma(g,h)^*\,\alpha_g(a^*a')\,\sigma(g, h')\, \sigma(gh, h^{-1}h')^*\big)\Big) \,\sigma(gh, h^{-1}h')\, \sigma(g, h')^*\, \alpha_g(b')\,.$$
Now, using Lemma \ref{cocy} at the last step, we have
\begin{align*}
\alpha_{gh}^{-1}\big(\sigma(g,h)^*\,\alpha_g(a^*a')\,\sigma(g, h')\, \sigma(gh, h^{-1}h')^*\big)&
= 
\big(\alpha_{h}^{-1}\alpha_{g}^{-1} {\rm Ad}(\sigma(g,h))\big)\big(\sigma(g,h)^*\,\alpha_g(a^*a')\,\sigma(g, h')\, \sigma(gh, h^{-1}h')^*\big)\\
&= \alpha_{h}^{-1}\big(\alpha_{g}^{-1} \big(\alpha_g(a^*a')\,\sigma(g, h')\, \sigma(gh, h^{-1}h')^*\sigma(g,h)^*\big)\\
&= \alpha_{h}^{-1}\big(a^*a'\,\alpha_{g}^{-1}(\sigma(g, h')\, \sigma(gh, h^{-1}h')^*\sigma(g,h)^*)\big)\\
&=\alpha_{h}^{-1}\big(a^*a'\, \sigma(h, h^{-1}h')^*\big)\,
\end{align*}
and we see that both sides of Equation (\ref{v0-2}) are equal, hence that Equation (\ref{v0}) holds.

\smallskip Equation (\ref{v0}) implies in particular that $v_0(g)\,F \in N$ whenever $F \in N$. 
We may therefore define  a linear map $\tilde{v}_0(g): X_0/N \to X_0/N$ by $$\tilde{v}_0(g)\, (F+N) = (v_0(g) \, F) + N$$
for each $F\in X_0$,  satisfying 
$$\alpha_g\big(\langle F+N\, ,\, F + N \rangle\big) 
= \big\langle \tilde{v}_0(g) (F+N) \, , \,\tilde{v}_0(g) (F+N) \big\rangle$$
for each $F\in X_0$. It clearly follows that $\tilde{v}_0(g)$ is isometric on $X_0/N$, 
so we may extend it to an isometry $v(g):X\to X$. By continuity, we get 
\begin{equation} \label{iii}
\alpha_g\big(\langle x, x \rangle\big) 
= \big\langle v(g)\,x,v(g)\,x\big\rangle 
\end{equation} 
for all $ x \in X$\,.

\smallskip 
Let $g\in G$ and $a\in A$. For $F \in X_0$ with standard decomposition $F = \sum_{i=1}^n \, \delta_{(g_i\,,\, a_i)}\cdot b_i$,  we  have
$$v_0(g)\,(F \cdot a) = \sum_{i=1}^n \, \delta_{\big(g\,g_i\,,\, \alpha_g(a_i)\,\sigma(g,g_i)\big)}\cdot \big(\sigma(g, g_i)^*\,\alpha_g(b_i\,a)\big) = \big(v_0(g)\, F\big) \cdot \alpha_g(a)\,.$$
Hence
\begin{align*}
\tilde{v}_0(g)\big((F+N)\cdot a\big) &= \tilde{v}_0(g)\big((F\cdot a) +N\big)  = \big(v_0(g)\,(F \cdot a)\big) + N \\
&= \big(\big(v_0(g)\, F\big) \cdot \alpha_g(a)\big) + N =  \big(\big(v_0(g)\, F\big)+ N\big)  \cdot \alpha_g(a)\\
& = \big(\tilde{v}_0(g)\,( F+ N)\big)  \cdot \alpha_g(a)
\end{align*}
and it follows by continuity  that 
\begin{equation} \label{iv}
v(g) (x\cdot a) = (v(g)\, x) \cdot \alpha_g(a)
\end{equation} 
 for all $x \in X$.

\smallskip 
Let $g, h \in G$. For $F \in X_0$, we  have
$$v_0(g) \,v_0(h) \,F = \big(\rho_0(\sigma(g,h))\, v_0(gh)\, F\big)\cdot \sigma(g,h)^*\,.$$
By linearity, it suffices to verify this formula for $F= \delta_{(k, b)}\cdot c$\,, where $k\in G$ and $b, c\in A$: 
\begin{align*}
v_0(g) \,v_0(h) \,\big( \delta_{(k, b)}\cdot c\big) &= v_0(g)\, \big(  \delta_{(hk\,, \,\alpha_h(b)\,\sigma(h,k))}\cdot \sigma(h,k)^*\,\alpha_h(c)\big)  \\
&=  \delta_{(ghk\,,\, \alpha_g(\alpha_h(b)\,\sigma(h,k))\,\sigma(g, hk))}\cdot \sigma(g, hk)^*\,\alpha_g\big(\sigma(h,k)^*\,\alpha_h(c)\big)  \\
&=  \delta_{(ghk\,,\, \sigma(g,h)\,\alpha_{gh}(b)\,\sigma(g,h)^*\,\alpha_g(\sigma(h,k))\,\sigma(g, hk))}\cdot \sigma(g, hk)^*\,\alpha_g(\sigma(h,k))^*\,\sigma(g,h)\,\alpha_{gh}(c)\,\sigma(g,h)^* \\
&=  \delta_{(ghk\,,\, \sigma(g,h)\,\alpha_{gh}(b)\,\sigma(gh, k))}\cdot \sigma(gh, k)^*\,\alpha_{gh}(c)\,\sigma(g,h)^*  \\
&= \rho_0(\sigma(g,h))\, \big(\delta_{(ghk\,,\, \alpha_{gh}(b)\,\sigma(gh, k))}\cdot \sigma(gh, k)^*\,\alpha_{gh}(c)\,\sigma(g,h)^*\big)\\
&= \rho_0(\sigma(g,h))\, \big(\big(v_0(gh)\,(\delta_{(k\,,\,b)}\cdot c\big)\cdot \sigma(g,h)^*\big)\\
&= \big(\rho_0(\sigma(g,h))\, v_0(gh)\, \big( \delta_{(k, b)}\cdot c\big)\big)\cdot \sigma(g,h)^*\,.
\end{align*}

\smallskip 
It follows easily that 
\begin{equation} \label{ii}
v(g) \,v(h) \,x = \big(\rho(\sigma(g,h))\, v(gh)\, x\big)\cdot \sigma(g,h)^* = {\rm ad}_\rho (\sigma(g,h))\, v(gh)\, x
\end{equation}
for all $x \in X$.  
As $v(e) = I$, one deduces readily from this equation that each $v(g)$ is invertible.  

\smallskip Let $g\in G$ and $a\in A$. For $F \in X_0$, we  have
$$v_0(g) \, \rho_0(a)\, F = \rho_0(\alpha_g(a)) \,v_0(g)\, F\,,$$
since, for all $h\in G$ and $b,c \in A$: 
$$v_0(g)\,\rho_0(a)\, \big(\delta_{(h,b)}\cdot c\big) = v_0(g)\, \big(\delta_{(h,\, ab)}\cdot c\big)
$$
$$=\delta_{(gh,\,\alpha_g(ab)\,\sigma(g,h))}\cdot \sigma(g,h)^*\,\alpha_g(c)
=\delta_{(gh,\,\alpha_g(a)\,\alpha_g(b))\,\sigma(g,h))}\cdot \sigma(g,h)^*\,\alpha_g(c)
$$
$$=\rho_0(\alpha_g(a))\,\big(\delta_{(gh,\,\alpha_g(b))\,\sigma(g,h))}\cdot \sigma(g,h)^*\,\alpha_g(c)\big)
=\rho_0(\alpha_g(a))\,v_0(g)\,\big(\delta_{(h,b)}\cdot c\big)\,.
$$
It follows that $\,v(g) \, \rho(a)  = \rho(\alpha_g(a)) \,v(g)$\,. Hence
\begin{equation}\label{i}
 \rho(\alpha_g(a)) =v(g) \, \rho(a) \,v(g)^{-1} \,.
\end{equation}
   
   \medskip Taking into account equations (\ref{iii}), (\ref{iv}), (\ref{ii}) and (\ref{i}), we see that we have shown that $(\rho,v)$ is an equivariant representation of $\Sigma$ on $X$.  We set $x = \delta_{(e,1)}+N \in X$. Then we have $$\big\langle x\, , \, \rho(a)\, v(g) \,x\big\rangle = \langle \delta_{(e,1)}\, , \, \rho_0(a)\, v_0(g)\, \delta_{(e,1)}\rangle_T = \langle \delta_{(e,1)}\, , \delta_{(g,a)}\rangle_T = \big[(e,1)\,,\, (g,a)\big]_T = T_g(a)\,,$$
 and the proof of the first assertion is completed.
 
\medskip  For $ a, \,b\in A$ and $ g\in G $, we have $\big(\rho(a)\, v(g) \,x\big)\cdot b = \delta_{(g,a)}\cdot b + N$. Hence, we get
 $$\text{{\rm Span}}\big\{ \big(\rho(a)\, v(g) \,x\big)\cdot b\mid a, \,b\in A\, ,\, g\in G\,\big\} = X_0/N,$$ 
 which is dense in $X$. Thus, 
   $x$ is cyclic for $(\rho, v)$ and the second assertion is proven.   
   
 \medskip Finally, assume that we also have $T=T_{\rho',v', x', x'}$ for some  equivariant representation $(\rho',v')$ of $\Sigma$ on a Hilbert $A$-module $X'$ and some $x' \in X'$ which is cyclic for $(\rho', v')$.
 Define a map $u_0: X_0 \to X'$ by 
 $$u_0 \, F =  \sum_{i=1}^n\, \big(\rho'(a_i)\, v'(g_i) \,x'\big)\cdot b_i$$
 whenever $F\in X_0$ has standard decomposition  $F = \sum_{i=1}^n \, \delta_{(g_i\,,\, a_i)}\cdot b_i$.
 One checks without much trouble that $u_0$ is $A$-linear. Moreover, we
 have $$\langle u_0 \,F, \,u_0\, F'\rangle = \langle  F,\, F'\rangle_T$$ for all $F, F' \in X_0$. 
 
 Indeed, consider $F$  as above and $F'\in X_0$ with   $F' = \sum_{j=1}^m \, \delta_{(g'_j\,,\, a'_j)}\cdot b'_j$ (standard decomposition).  The computation in Example \ref{EQRPD} (with $T=T_{\rho',v', x', x'}$) gives that 
 $$\big\langle\,\rho'(a)\, v'(g) \,x'\,, \, \rho'(b)\, v'(h) \,x'\,\big\rangle = \alpha_{g}\Big(T_{g^{-1}h} \big(\alpha_{g}^{-1}\big(a^*\,b\,\sigma(g, g^{-1}h)^*\big)\,\big)\Big)\, \sigma(g, g^{-1}h)
 = \big[(g,a)\, , \, (h,b)\, \big]_T$$
 for all $a, b \in A$ and $g, h\in G$. Hence,
  $$\langle u_0 \,F, \,u_0\, F'\rangle = \sum_{i=1}^n\sum_{j=1}^m\, \big\langle\,\big(\rho'(a_i)\, v'(g_i) \,x'\big)\cdot b_i\,, \, \big(\rho'(a'_j)\, v'(g'_j) \,x'\big)\cdot b'_j\,\big\rangle$$
$$= \sum_{i=1}^n\sum_{j=1}^m\, b_i^*\,\big\langle\,\rho'(a_i)\, v'(g_i) \,x'\,, \, \rho'(a'_j)\, v'(g'_j) \,x'\,\big\rangle\, b'_j$$ 
$$= \sum_{i=1}^n\sum_{j=1}^m\, b_i^*\,\big[(g_i,a_i)\, , \, (g'_j,a'_j)\, \big]_T\, b'_j = \langle  F,\, F'\rangle_T\,.$$ 
 In particular, it follows that $N$ is contained in the kernel of $u_0$. Hence, there exists a unique  $A$-linear isometry $u: X\to X'$ determined by $u (F+N) = u_0\, F$ for $F\in X_0$. Since the range  of $u$ is dense in $X'$ (as $x'$ is cyclic for $(\rho', v')$)  and closed (since $u$ is isometric), $u$ is surjective. Thus, by \cite[Theorem 3.5]{La1}, $u$ is a unitary in $\L(X,X')$. 
 
 Now, for $a, b, c \in A$ and $ \, g, h\in G$, we have
 $$\big[u\, \rho(a)\big]\big(\delta_{(h,b)}\cdot c  + N\big)= u_0 \,\big (\delta_{(h,ab)}\cdot c \big) = \big(\rho'(ab)\, v'(h)\, x'\big)\cdot c$$
 $$= \rho'(a)\big((\rho'(b)\, v'(h)\, x')\cdot c\big) = \big[\rho'(a)\, u\big](\delta_{(h,b)}\cdot c  + N\big)\,,$$
 and 
 $$\big[u\, v(g)\big]\big(\delta_{(h,b)}\cdot c  + N\big) = u_0\, \big(\delta_{(gh,\alpha_g(b)\sigma(g,h))}\cdot(\sigma(g,h)^*\alpha_g(c)) \big)$$
$$ =\big( \rho'(\alpha_g(b)\sigma(g,h))\, v'(gh)\, x'\big)\cdot(\sigma(g,h)^*\alpha_g(c))$$
$$ = 
\big( \rho'(\alpha_g(b))\, (\rho(\sigma(g,h))\, v'(gh)\, x')\cdot \sigma(g,h)^*)\big) \cdot \alpha_g(c)$$
$$ = \big( \rho'(\alpha_g(b))\, v'(g)\, v'(h)\, x'\big) \cdot \alpha_g(c)= \big( v'(g)\,\rho'(b)\, v'(h)\, x'\big) \cdot \alpha_g(c)$$
$$= v'(g)\,\big((\rho'(b)\, v'(h)\, x') \cdot c\big) = [v'(g)\, u]\big(\delta_{(h,b)}\cdot c  + N\big)\,,$$
so it follows that    $\rho'(a) = u\, \rho(a)\, u^*$ and $v'(g) = u\, v(g) \, u^*$ for all $a\in A$ and $g\in G$. 
Finally, we have $$u\, x = u_0\, \delta_{(e,1)} = \rho'(1)\, v'(e) \, x' = x'\,.$$
     \end{proof}
 
    A well-known property of a positive definite function $\varphi$ on $G$ is that $\varphi$ is bounded, with $\|\varphi\|_\infty =
 \varphi(e)$. Similary, we have:
  
 \begin{corollary} \label{bounded}
 Assume $T\in L(\Sigma)$ is positive definite $($w.r.t.\ $\Sigma)$. Then $T$ is bounded, in the sense that 
 $$\|T\|_\infty:=\sup \{\,\|T_g\|\mid g\in G\}  < \infty\,,$$
 and we have $\|T\|_\infty = \|T_e(1)\|$.
 \end{corollary}  
 \begin{proof}
 Write $T = T_{\rho, v, x, x}$ as in Theorem \ref{GR}. For each $g\in G$ and $a \in A$, we  have
 $$\|T_g(a)\| = \|\,\langle x, \rho(a)v(g) x\rangle\,\|\, \leq\, \|\rho(a)\| \,\|v(g)\|\, \|x\|^2  \,\leq \,\|a\| \,\|x\|^2\,.$$
 Hence, $\|T_g\| \leq \|x\|^2$ for all $g\in G$, so $T$ is bounded with $\|T\|_\infty \leq \|x\|^2\,.$ Moreover, since
 $T_e(1) = \langle x, x\rangle$, we get 
 $$\|T_e(1)\| \leq \|T_e\| \leq \|T\|_\infty \leq \|x\|^2 =  \|T_e(1)\|\,$$
 and the last assertion follows.
  \end{proof}
 \begin{corollary}\label{GR-Cor}
 Let $T\in L(\Sigma)$. Then the following statements are equivalent:
 
 \begin{itemize}
 \item[$a)$] $T = T_{\rho, v, x, x}$ for some equivariant representation $(\rho, v)$ of \,$\Sigma$ on some Hilbert $A$-module $X$ and some $x\in X$.
 
 \item[$b)$] $T$ is  positive definite $($w.r.t.\ $\Sigma)$.
 
 \item[$c)$] $T$ is a  full multiplier of $\Sigma$ such that $\Phi_T$ is completely positive.
 
 \item[$d)$] $T$ is a  reduced multiplier of $\Sigma$ such that $M_T$ is completely positive.
 
 \end{itemize}
\end{corollary}
\begin{proof}   $a) \Rightarrow b)$ is Example \ref{EQRPD}, $b) \Rightarrow a)$ is Theorem \ref{GR}, $a) \Rightarrow c)$ follows from Theorem \ref{equiv-coeff-full} , $a) \Rightarrow d)$ follows from Theorem \ref{equiv-coeff}, $c) \Rightarrow b)$ is Corollary \ref{phi-cpCor} and $d) \Rightarrow b)$ is Corollary \ref{M-cpCor}. 
\end{proof}
Note that $c) \Rightarrow a)$ is Theorem \ref{conv-full-mult}, while $d) \Rightarrow a)$ is Theorem \ref{conv-red-mult}. Thus, Corollary \ref{GR-Cor} provides alternative proofs of these two theorems. 

  \begin{example} \label{cpequi2}
  Let $\theta: A\to A$ be a linear map and let $\Theta: G\times A \to A $ be given by $\Theta(g,a) = \theta(a)$. Assume that $\theta$ is completely positive and  $\alpha$-equivariant.
  If $\sigma$ is scalar-valued or satisfies $\theta\circ\sigma = \sigma$, then we know from Example \ref{cpequi} that $\Theta$ is $\Sigma$-positive definite, so we conclude from Corollary \ref{GR-Cor} that $T_\Theta$ gives a reduced (resp.\ full) multiplier of $\Sigma$, hence that $M_\theta:=M_{T_\Theta}$ (resp.\ $\Phi_\theta:=\Phi_{T_\Theta}$) is a completely positive linear map on $C_{\rm r}^*(\Sigma)$ (resp.\ $C^*(\Sigma)$) such that $M_\theta\big(a\,\lambda_\Sigma(g)\big) = \theta(a)\,\lambda_\Sigma(g)$ (resp.\ $\Phi_\theta\big(i_A(a)\,i_G(g)\big) = i_A(\theta(a))\,i_G(g)$) for all $a\in A$ and $g\in G$. 
When $\sigma$ is trivial, this fact has long been a part of the folklore (e.g.\ it is mentioned in \cite[p.\ 173]{KW}); it can be deduced from \cite[Theorem 4.9 and Corollary 4.18]{BEW}. 
\end{example}

\begin{definition} We set $P(\Sigma) =\big\{T\in L(\Sigma)\mid T \, \text{is positive definite $($w.r.t.} \, \Sigma )\big\}$.
\end{definition}

It is clear from Corollary \ref{GR-Cor} that $P(\Sigma) \subset B(\Sigma)$ and that $P(\Sigma)$ is closed under the product in $B(\Sigma)$. Moreover, we have: 

\begin{corollary} \label{span-pd} 
$$B(\Sigma) = \text{\rm Span} \, P(\Sigma)\,.$$
\end{corollary}

\begin{proof} If $T \in B(\Sigma)$, say $T= T_{\rho, v, x, y}$, then usual polarization gives $T = \frac{1}{4}\, \sum_{k=0}^3 \, i^k\, T_{\rho, v, x_k, x_k}\,,$ where $x_k = x + i^k\, y\, , \, k= 0,\ldots, 3$. Hence, the result follows from Corollary \ref{GR-Cor}.
\end{proof} 

Let $B$, $C$ be  C$^*$ algebras. We recall from \cite{Haa85} that a bounded linear map from  $B$ to $C$ is called {\it decomposable} (in the sense of Haagerup) if it is a  finite linear combination of  completely positive linear maps from $B$ to $C$. Clearly, such a map is completely bounded. Hence we may define a subspace $M_{\rm dec}(\Sigma)$ of $M_0A(\Sigma)$ and a subspace $M^{\rm u}_{\rm dec}(\Sigma)$ of $M^{\rm u}_{\rm cb}(\Sigma)$ by
$$M_{\rm dec}(\Sigma)=\big\{ T \in MA(\Sigma) \mid M_T:C_{\rm r}^*(\Sigma)\to C_{\rm r}^*(\Sigma)\,\, \text{is decomposable} \big\}\,,$$
$$M^{\rm u}_{\rm dec}(\Sigma)=\big\{ T \in M^{\rm u}(\Sigma) \mid \Phi_T:C^*(\Sigma)\to C^*(\Sigma)\,\, \text{is decomposable} \big\}\,.$$

\begin{corollary}\label{dec} We have  $$B(\Sigma) = M_{\rm dec}(\Sigma)= M^{\rm u}_{\rm dec}(\Sigma)\,.$$
\end{corollary}
\begin{proof} Let $T \in B(\Sigma)$. Using Corollary \ref{span-pd}, we may write $T= \sum_{i=1}^n c_i\,T_i$\,,  where 
$c_1,\ldots, c_n \in \Complessi$ and $ T_1, \ldots, T_n \in P(\Sigma)$. Then $M_T = \sum_{i=1}^n c_i\, M_{T_i} $ is a linear combination of completely positive maps from $C^*_{\rm r}(\Sigma)$ to itself, using now Corollary \ref{GR-Cor}. Hence, $M_T$ is decomposable, so $T\in M_{\rm dec}(\Sigma)$. Thus, $B(\Sigma) \subset M_{\rm dec}(\Sigma)$. 

Assume now that $T\in M_{\rm dec}(\Sigma)$, and write $M_T = \sum_{i=1}^n c_i\,M_i$ \,,  where 
$c_1,\ldots, c_n \in \Complessi$ and $ M_1, \ldots, M_n$ are completely positive maps from $C^*_{\rm r}(\Sigma)$ to itself. Then, using the notation introduced in Proposition \ref{M-cp}, we have $T= T_{M_T} = \sum_{i=1}^n c_i \,T_{M_i}$, the last equality being easily checked. Now, using Corollary \ref{M-cpCor}, each $T_{M_i}$ is positive definite (w.r.t. $\Sigma)$, so we get that $T \in {\rm Span}\, P(\Sigma) = B(\Sigma)$. Thus, $B(\Sigma) \supset M_{\rm dec}(\Sigma)$.

Together, we have shown the first equality that was to be proved. The proof that $B(\Sigma) = M^{\rm u}_{\rm dec}(\Sigma)$ is similar and left to the reader.
\end{proof} 

\begin{remark} \label{PiBoDCH} When $\Sigma=(\Complessi, G, {\rm id}, 1)$, we have  $B(\Sigma)=B(G)$ , $M_{\rm cb}^{\rm u}(\Sigma) = M_{\rm cb}^{\rm u}(G)$ and $M_0A(\Sigma)=M_0A(G)$,  and it is known that $B(G) = M_{\rm cb}^{\rm u}(G)$ (cf.\ \cite[Theorem 1]{Wal1}, \cite[Corollary 8.7]{Pis2}), while the equality $B(G) = M_0A(G)$ holds if and only if $G$ is amenable (cf.\ \cite[Corollary 1.8]{DCHa}, \cite[Theorem]{Bo}). In the general case, one may wonder what kind of conditions are sufficient to ensure that the equality $B(\Sigma) = M_{\rm cb}^{\rm u}(\Sigma)$ (resp.\  $B(\Sigma) = M_0A(\Sigma)$) holds. Corollary \ref{dec} says that one may equally well consider the question whether the equality $M^{\rm u}_{\rm dec}(\Sigma)= M^{\rm u}_{\rm cb}(\Sigma)$ (resp.\ $M_{\rm dec}(\Sigma)= M_0A(\Sigma)$) holds. 

Properties of $\Sigma$ will have to enter the picture. For example, assume that $G$ is amenable and $\sigma$ is trivial. Assume moreover that there exists a completely bounded $\alpha$-equivariant linear map  $\theta: A\to A$ that is {\it not} decomposable. Then  $C^*(\Sigma) = C^*_{\rm r}(\Sigma)$ and $\theta$ extends to a completely bounded linear map $\theta \times \iota:C_{\rm r}^*(\Sigma) \to C_{\rm r}^*(\Sigma)$ satisfying $(\theta \times \iota)\big(a\lambda_\Sigma(g)\big) = \theta(a) \lambda_\Sigma(g)$ for all $a\in A$ and $g\in G$ (cf.\ \cite[Theorem 3.5]{RSW}). This means that the map $\Theta \in L(\Sigma)$ associated to $\theta$ as in Example \ref{cpequi2} belongs to $M_0A(\Sigma)$, and we have $M_\Theta = \theta \times {\iota}$. However, 
$\Theta$ does not lie in $ M_{\rm dec}(\Sigma)$,  because decomposability of $M_\Theta$ would imply decomposability of $\theta$ (as is easily verified by using that $\theta(a) = (E\circ M_\Theta)(a)$ for all $a\in A$). Thus, we have \begin{equation*}
\Theta \in \, M_0A(\Sigma)\setminus M_{\rm dec}A(\Sigma) = M_0A(\Sigma)\setminus B(\Sigma) = 
M_{\rm dec}^{\rm u}(\Sigma)\setminus B(\Sigma)
\end{equation*}
in this case. Now, a deep result of Haagerup (cf.\ \cite[Corollary 2.8]{Haa85}) says that when $A$ is a von Neumann algebra, then there exists a completely bounded linear map  $\theta: A\to A$ that is not   decomposable if (and only if) $A$ is non injective. Hence,  when $G$ is amenable, $\alpha$ and $\sigma$ are trivial, and $A$ is a non injective von Neumann algebra, we can conclude that
$ B(\Sigma) \neq M_{\rm dec}^{\rm u}(\Sigma)=M_0A(\Sigma)$.
\end{remark}

 \medskip  Turning to another application of Theorem \ref{GR}, let us consider a map $\varphi: G \to A$ and let $L^\varphi \in L(\Sigma)$ be given by
 $$L^\varphi(g,a) = \varphi(g) \,a$$ for $g \in G$ and $ a \in A$.  
Such maps (and their right-handed versions) naturally arise when considering coefficients of an equivariant representation $(\rho, v)$ of $\Sigma$ on $X$  associated with central vectors. Indeed, as in \cite[Example 4.11]{BeCo4}, we have   
$$T_{\rho, v, x, y}(g,a) = \big\langle x,  v(g)y \big\rangle\, a \quad \text{ if } y\in Z_X,$$
while 
$$T_{\rho, v, x, y}(g,a) = a \, \big\langle x ,  v(g)y \big\rangle\quad \text{ if }  x\in Z_X\,.$$

\smallskip \noindent In particular, when $x, y \in Z_X$, we have $\big\langle x,  v(g)y \big\rangle \in Z(A)$ and
\begin{equation}\label{eqcentr}
T_{\rho, v, x, y}(g,a) = \big\langle x,  v(g)y \big\rangle\, a =  a \, \big\langle x ,  v(g)y \big\rangle\,.
\end{equation}

\smallskip We note that if $\varphi: G \to A$, then $L^\varphi$ is positive definite (w.r.t.\ $\Sigma$)
 if and only if  the matrix 
 $$\Big[\alpha_{g_i}\big(\varphi(g_i^{-1}g_j)\big) \, a_i^*\, a_j\Big]$$
 is 
 positive in $M_n(A)$ for all $n\in \Naturali$, $g_1, \ldots, g_n \in G$ and $a_1, \ldots, a_n \in A$. This is an immediate consequence of the following computation:
  
  \quad $\alpha_{g_i}\Big(L^\varphi_{g_i^{-1}g_j}\big( \alpha_{g_i}^{-1}( a_i^*\, a_j\,\sigma(g_i, g_i^{-1}g_j)^*\big)\Big)\,\sigma(g_i, g_i^{-1}g_j)$
 
 \vspace{-4ex} \begin{align*} &= \alpha_{g_i}\Big(\varphi(g_i^{-1}g_j) \,\alpha_{g_i}^{-1}\big( a_i^*\, a_j(\sigma(g_i, g_i^{-1}g_j)^*)\,\big)\Big)\, \sigma(g_i, g_i^{-1}g_j)\\
 &= \alpha_{g_i}\big(\varphi(g_i^{-1}g_j)\big) \, a_i^*\, a_j\,.
 \end{align*}
Now, following \cite{AD1, AD2},  the function  $\varphi$ is said to be {\it AD-positive definite}
(w.r.t.\ $\Sigma$) if,
for any $g_1,\ldots,g_n \in G$, the matrix $$\Big[\alpha_{g_i}\big(\varphi(g_i^{-1}g_j)\big)\Big]$$ is positive in $M_n(A)$.

When $\sigma =1$, Anantharaman-Delaroche  establishes in \cite[Proposition 2.3]{AD1} a Gelfand-Raikov type of result for AD-positive-definite functions (w.r.t.\ $\alpha$), involving so-called  $\alpha$-equivariant actions of $G$ on Hilbert $A$-modules.  Her result is related to Theorem \ref{GR} when $T=L^\varphi$, but one should note that these two results are not equivalent in this case.  
Indeed, we have:  

\begin{proposition} \label{L-phi} Let $\varphi: G\to A$. Then the following conditions are equivalent:
\begin{itemize}
\item[a)] $L^\varphi$ is  positive definite (w.r.t.\ $\Sigma$), that is, $L^\varphi \in P(\Sigma)$;
\item[b)] $\varphi$ takes its values in $Z(A)$, the center of $A$, and $\varphi$ is  AD-positive definite $($w.r.t. $\Sigma)$;
\item[c)] $L^\varphi = T_{\rho, v, x,x}$ for some equivariant representation $(\rho,v)$ of $\Sigma$ on a Hilbert $A$-module $X$ and some $x \in Z_X$ $($the central part of $X$ w.r.t.\ $\rho$$)$; 
\item[d)] $L^\varphi $ is a $($reduced$)$ multiplier of $\Sigma$ such that $M_{L^\varphi}$ is completely positive; 
\item[e)] $L^\varphi $ is a full multiplier of $\Sigma$ such that $\Phi_{L^\varphi}$ is completely positive. 
\end{itemize}
\end{proposition}
\begin{proof}
$a) \Rightarrow b)$: Assume $L^\varphi$ is  positive definite (w.r.t.\ $\Sigma$). Choosing $a_i=1$ for $i=1, \ldots, n$ in the definition gives immediately that  
 $\varphi$ is  AD-positive definite (w.r.t.\ $\Sigma$). Moreover, using Remark \ref{PD-conj} with $T=L^\varphi$, we get that 
 $$\big(\varphi(g)\, a\big)^* = \alpha_g\big(\varphi(g^{-1})\, \alpha_g^{-1}(a^*) \,\alpha_g^{-1}(\sigma(g, g^{-1})^*)\big)\, \sigma(g, g^{-1}) =  \alpha_g\big(\varphi(g^{-1})\big)\, a^*$$
 for all $g\in G$ and $a \in A$. In particular, we have $\varphi(g)^* = \alpha_g\big(\varphi(g^{-1})\big)$, and it then follows that $$a^*\, \varphi(g)^* = \big(\varphi(g)\, a\big)^* = \varphi(g)^*\, a^*$$
 for all $g\in G$ and $a\in A$. Hence $\varphi(g) \in Z(A)$ for every $g\in G$. 
 
 \smallskip 
 $b) \Rightarrow a)$: Assume   $\varphi : G\to Z(A)$ is  
  AD-positive definite (w.r.t. $\Sigma$). Let $g_1, \ldots, g_n \in G$ and $a_1, \ldots, a_n \in A$. Then we have
 \begin{align*}\Big[\alpha_{g_i}\big(\varphi(g_i^{-1}g_j)\big) \, a_i^*\, a_j\Big] &= \Big[ a_i^*\, \alpha_{g_i}\big(\varphi(g_i^{-1}g_j)\big)\,a_j\Big] \\
 &=D^*\, \Big[ \alpha_{g_i}\big(\varphi(g_i^{-1}g_j)\big)\Big]\, D\,,
 \end{align*}
 where $D=$ diag$\big(a_1,\, \ldots,\, a_n\big)$. As $B:=\big[\alpha_{g_i}\big(\varphi(g_i^{-1}g_j)\big)\big]$ is assumed to be positive in $M_n(A)$, 
 $D^*BD$ is positive in $M_n(A)$. Thus, it follows that $L^\varphi$ is  $\Sigma$-positive definite. 

\smallskip 
 $a) \Rightarrow c)$: Assume $L^\varphi$ is  positive definite (w.r.t.\ $\Sigma$). From Theorem \ref{GR}, we know that
 c) holds, except for the fact that $x$ may be chosen in $Z_X$. Let us set $T= L^\varphi$ and use the notation introduced in the proof of this theorem. We have then that $x = \delta_{(e,1)} + N \in X$. To check that $x \in Z_X$, we have to show that $\rho(a) x = x\cdot a$ for all $a\in A$. Let $a\in A$. As 
$$\Big\langle \delta_{(e,a)} -\delta_{(e,1)}\cdot a\, , \, \delta_{(e,a)} -\delta_{(e,1)}\cdot a\Big\rangle_{L^\varphi}= 
\varphi(e)\, a^*a - a^*\varphi(e) \, a - \varphi(e)\,  a^* a + a^*\varphi(e) \, a = 0\, ,$$
we have that $\delta_{(e,a)} -\delta_{(e,1)}\cdot a \in  N$. Hence, 
we get that $$\rho(a) x = \delta_{(e,a)} + N = \delta_{(e,1)}\cdot a + N = x\cdot a$$
as desired.

\smallskip Finally, $c) \Rightarrow a)$ follows from Example \ref{EQRPD}, while $a) \Leftrightarrow d)$ and $a) \Leftrightarrow e)$ follows from Corollary \ref{GR-Cor}.

\end{proof}

\begin{remark} \label{DongR} When $\sigma=1$, then the implication $b) \Rightarrow d)$ in Proposition \ref{L-phi} is due to Dong and Ruan \cite[Theorem 3.2]{DoRu} (although they use some different conventions). They also point out that if $b)$ holds, then the map $M_{L^\varphi}$ is an $A$-bimodule map, i.e., we have $M_{L^\varphi}(axb) = a \,M_{L^\varphi}(x)\, b$ for all $a, b \in A$ and $x \in C_{\rm r}^*(\Sigma)$. 
\end{remark}

\begin{remark} At the end of \cite[Example 4.11]{BeCo4}, we wondered whether functions from $G$ to $A$ that are AD-positive definite give rise to (bounded) multipliers of $\Sigma$ in general. The equivalence of $b)$ and $d)$ in Proposition \ref{L-phi} provides a partial answer to this question. However, if $\varphi:G\to A$ is AD-positive definite (w.r.t.\ $\Sigma$), but does not take all of its values in $Z(A)$, we can not exclude for the moment the fact that $L^\varphi$ could still lie in $M_0A(\Sigma)$, or in $MA(\Sigma)$.
\end{remark}

\begin{remark}
Of course, when $\varphi : G\to A$ is given, one may also consider $R^\varphi \in L(\Sigma)$ defined by $R^\varphi(g,a) = a \,\varphi(g)$. It is easily checked that  $R^\varphi$ is positive definite (w.r.t.\ $\Sigma$)
 if and only if  the matrix 
  $$\Big[a_i^*\, a_j\, {\rm Ad}\big(\sigma(g_i,g_i^{-1}g_j)^*\big)\, \alpha_{g_i}\big(\varphi(g_i^{-1}g_j)\big)\Big]$$
 is 
 positive in $M_n(A)$ for all $n\in \Naturali$, $g_1, \ldots, g_n \in G$ and $a_1, \ldots, a_n \in A$. In view of our discussion of $L^\varphi$,  
it is natural to say that $\varphi$ is {\it $\sigma$-AD-positive definite} (w.r.t.\ $\Sigma$) if, for any $g_1,\ldots,g_n \in G$, we have
$$\Big[{\rm Ad}\big(\sigma(g_i,g_i^{-1}g_j)^*\big)\, \alpha_{g_i}\big(\varphi(g_i^{-1}g_j)\big)\Big] \, $$
 is positive in $M_n(A)$. It is obvious that this concept coincides with AD-positive definiteness when $\varphi$ take values in $Z(A)$. It is therefore not difficult to check that Proposition \ref{L-phi} also holds if we replace $L^\varphi$ with $R^\varphi$ everywhere.
\end{remark}

\begin{remark}\label{pdclassic}
Consider a function $\varphi:G\to \Complessi$. In this case, we have 
$ T^\varphi=L^\varphi=R^\varphi$. Moreover, identifying $\Complessi$ with $\Complessi\cdot 1 \subset A$, one immediately sees that $\varphi$ is a positive definite function on $G$, i.e., $\varphi \in P(G)$,  if and only if $\varphi$ is AD-positive definite (w.r.t.\ $\Sigma$). Hence, it follows from Proposition \ref{L-phi} that $\varphi \in P(G)$ if and only if $T^\varphi \in P(\Sigma)$, if and only 
$T^\varphi = T_{\rho, v, x,x}$ for some equivariant representation $(\rho,v)$ on a Hilbert $A$-module $X$ and some $x \in Z_X$, if and only if $T^\varphi$ is a reduced (resp.\ full) multiplier of $\Sigma$ such that $M_{T^\varphi}$ (resp.\ $\Phi_{T^\varphi}$) is completely positive. To our knowledge, the first result in this direction goes back to Haagerup \cite{Ha} in the setting of crossed products of von Neumann algebras.  

\end{remark} 
\begin{remark}\label{fourier}
Following \cite{Eym}, one way to define the Fourier algebra of $G$ is to let $A(G)$ be the closure in $B(G)$ of the span of all positive definite functions on $G$ that have finite support. For $T\in L(\Sigma)$, let us say that $T$ has {\it finite support} if the set $\{ g\in G \mid T_g \neq 0\}$ is finite. Then one possible definition of the Fourier algebra of $\Sigma$ is to let $A(\Sigma)$ be the closure in $B(\Sigma)$ of the span of all elements in $P(\Sigma)$ that have finite support. One easily checks that $A(\Sigma)$ is then a two-sided ideal in $B(\Sigma)$.
A natural problem is to investigate whether this definition of $A(\Sigma)$ coincides with the other candidates that make sense in our setting (see \cite[Proposition (3.4) and the Th{\' e}or{\`e}me on p.\ 218]{Eym}). For example, is it true that $A(\Sigma)$, if defined as above, is equal to the set of coefficients of the regular equivariant representation of $G$ on $A^G$ ?  
\end{remark}

Among the many existing characterizations of the amenability of $G$ (see e.g.\ \cite[Theorem 2.6.8]{BrOz} for a few of these), one says that $G$ is amenable if and only if there exists a net $\{\varphi_i\}$ of normalized finitely supported positive definite functions on $G$ such that $\varphi_i \to 1$ pointwise. We propose below an analogous definition of amenability for $\Sigma$. We will call a net $\{T^i\}$ in $P(\Sigma)$ {\it uniformly bounded} when $\sup_i \|T^i\|_\infty = \sup_i\|T_e^i(1)\| < \infty$ (cf.\ Corollary \ref{bounded}). 

\begin{definition} We will say that  
$\Sigma$ is {\it amenable} whenever there exists a uniformly bounded net $\{T^i\} $ of  finitely supported elements in $P(\Sigma)$ such that $\lim_i \|T^i_g(a) - a\| = 0$ for every $g \in G$ and $a\in A$. 
\end{definition}
\begin{example}
Recall that $\Sigma$ is said to have {\it the weak approximation property} \cite{BeCo3} if there exists
an equivariant representation $(\rho, v)$ of $\Sigma$ on some Hilbert $A$-module $X$ and
nets $\{\xi_i\}, \{\eta_i\} $ in $C_c(G, X) \subset X^G$ 
satisfying
\begin{itemize}
\item[(a)] there exists some $M > 0$ such that $\|\xi_i\| \cdot \|\eta_i\| \leq M $ for all $i $\,;
\item[(b)]  
$\lim_i \big\| \big\langle \xi_i\,,\,\check\rho(a)\check{v}(g)\eta_i \big\rangle  - a\big\| = 0\,$ for all $g \in G$ and $a \in A$.
\end{itemize}
In this definition, $X^G$ denotes the Hilbert $A$-module given by  $$X^G= \Big\{ \xi: G \to X\mid \,\sum_{g\in G} \, \langle\xi(g),\xi(g)\rangle \,\, \text{is norm-convergent in}\, \,A\Big\} \, \simeq \, X\otimes \ell^2(G)\,,$$
the $A$-valued inner product being given by $ \langle\xi,\eta\rangle = \sum_{g\in G} \, \langle\xi(g),\eta(g)\rangle $ for $\xi, \eta \in X^G$, while $\check\rho := \rho \otimes \iota$ and $ \check{v} := v \otimes \lambda$. 

If the conditions (a) and (b) above hold with $\eta_i = \xi_i$ for all $i$, then $\Sigma$ is said to have the {\it positive weak approximation property}. In this case, if we set $T^i := T_{\check\rho, \,\check{v}, \,\xi_i, \,\xi_i}$ for each $i$, then it follows from Example \ref{EQRPD} that we get a finitely supported net $\{T^i\}$ in $P(\Sigma)$ such that $$\sup_i \|T^i\|_\infty = \sup_i \|T^i_e(1)\| = \sup_i \|\langle \xi_i, \xi_i\rangle \| = \sup_i \|\xi_i\|^2 \leq M$$ and
$$\lim_i \|T^i_g(a) - a\|= \lim_i \big\| \big\langle \xi_i\,,\,\check\rho(a)\check{v}(g)\eta_i \big\rangle  - a\big\| = 0\,,$$
hence we see that $\Sigma$ is then amenable. 

More concretely, let us assume that  the positive weak approximation property is achieved with $(\rho, v) = (\ell, \alpha)$, i.e., that there exists a net  $\{\xi_i\}$ in $C_c(G,A)$ such that  
\begin{itemize}
\item[(a')] $\sup_i \big\| \sum_{g\in G} \, \xi_i(g)^*\xi_i(g)\big\| < \infty\,$;

\item[(b')] 
$\lim_i \big\| \, \sum_{h\in G} \, \xi_i(h)^*a\,\alpha_g\big(\xi_i(g^{-1}h)\big)  - a\big\| = 0\,
$ for all $g \in G$ and $a \in A$.
\end{itemize}
These conditions say that $\Sigma$ satisfies {\it Exel's positive approximation property} \cite{Ex, ExNg}.
Then $\Sigma$ is amenable, and a net $\{T^i\}$ satisfying the required properties is given by $$T^i_g(a) = \sum_{h\in G} \, \xi_i(h)^*a\,\alpha_g\big(\xi_i(g^{-1}h)\big)$$ for $g\in G$ and $a\in A$.   
Note that if all $\xi_i$'s take their values in $Z(A)$, then (b') is equivalent to  $$\lim_i \big\| \, \sum_{h\in G} \, \xi_i(h)^*\,\alpha_g\big(\xi_i(g^{-1}h)\big)  - 1\big\| = 0\,
$$ for all $g \in G$. It is then straightforward to see that in the case where $\sigma=1$, then $\Sigma$ is amenable whenever the action $\alpha$ is amenable in the sense of \cite{BrOz}. 

\end{example}

As shown in \cite[Theorem 5.11]{BeCo3}, $\Sigma$ is regular whenever it has the weak approximation property.
We also have:
 
\begin{theorem}\label{amen}
Assume that $\Sigma$ is amenable. Then $\Sigma$ is regular. Hence, $C^*(\Sigma) \simeq C_{\rm r}^*(\Sigma)$. Moreover,
$C^*(\Sigma) \simeq C_{\rm r}^*(\Sigma)$ is nuclear if and only if $A$ is nuclear.
\end{theorem}
\begin{proof}
To prove the first assertion we will follow the approach used  in the proof of \cite[Theorem 2.6.8]{BrOz} in the case where $A$ and $\sigma$ are trivial, and $G$ is amenable. 

We first observe that if $S$ is a finite subset of $G$, then the subspace $C_S$ of $C_c(\Sigma)$ given by  $C_S:= {\rm Span}\,\{ a \odot g\mid a \in A, \,g\in S\}$ is closed in $C^*(\Sigma)$. Indeed, let $\{f_n\}$ be a sequence in $C_S$ converging to some $x\in C^*(\Sigma)$ and set $\widetilde{E} := E\circ \Lambda_\Sigma: C^*(\Sigma) \to A$, which is clearly continuous. For each $g\in S$, we then have $$f_n(g) = E\big(\Lambda_\Sigma(f_n)\,\lambda_\Sigma(g)^*\big) = \widetilde{E}\big(f_n\, i_G(g)^*\big) \to  \widetilde{E}\big(x \,i_G(g)^*\big) \,\, \text{as}\,\, n\to \infty\,.$$ So if we define $f \in C_S$ by $f =\sum_{g \in S} f(g)\odot g$\,, where $f(g) =  \widetilde{E}\big(x\, i_G(g)^*\big)$ for each $g \in S$, we get
$$\| x -f\|_{\rm u} \leq \| x -f_n\|_{\rm u} + \| f_n -f\|_{\rm u} \leq  \| x -f_n\|_{\rm u} + \sum_{g\in S} \|f(g) -f_n(g)\| \to 0 \,\, \text{as}\, \, n\to \infty\,.$$
Hence, $x = f \in C_S$.

Next, consider a  finitely supported $T\in P(\Sigma)$ and let $x \in C^*(\Sigma)$. Then $\Phi_T(x) \in C_c(\Sigma)$. Indeed, letting $S$ denote the support of $T$ and $\{f_n\}$ be a sequence in $C_c(\Sigma)$ converging to $x$, we have $\Phi_T(f_n) =\sum_{g\in S} T_g\big(f_n(g)\big)\odot g\, \in C_S$ for each $n$, so $\Phi_T(x) = \lim_n \Phi_T(f_n)$ belongs to $ \overline{C_S}=C_S \subset C_c(\Sigma)$.

 Let now $x \in C^*(\Sigma)$ and suppose that $\Lambda_\Sigma(x)=0$. Let  $\{T^i\}$ be a net as guaranteed by the amenability of $\Sigma$. Corollary \ref{GR-Cor} gives a net $\{ \Phi_{T^i}\}$ (resp.\ $\{ M_{T^i}\}$) of completely positive maps on $C^*(\Sigma)$ (resp.\  $C_r^*(\Sigma)$). Note that for each $i$ we have $\Lambda_\Sigma \circ \Phi_{T^i} = M_{T^i}\circ \Lambda_\Sigma$ on $C^*(\Sigma)$ since the two maps are continuous and agree on $C_c(\Sigma)$. So for each $i$ we get 
 $$\Lambda_\Sigma \big(\Phi_{T^i}(x)\big) = M_{T^i}\big(\Lambda_\Sigma(x)\big) = 0\,.$$
Since each $T^i$ is finitely supported, we have $\Phi_{T^i}(x) \in C_c(\Sigma)$. As $\Lambda_\Sigma$ is injective on $C_c(\Sigma)$, we obtain that  $\Phi_{T^i}(x) = 0$ for each $i$\,.

Observe that for  $f\in C_c(\Sigma)$ with support $F$, we have 
 $$\|\Phi_{T^i}(f) - f\|_{\rm u} = \big\|\sum_{g\in F} \,\big[T_g^i(f(g)) -f(g)\big]\odot g\big\|_{\rm u} \, \leq\, \sum_{g\in F} \,\big\|T_g^i(f(g)) -f(g)\|\,.$$
 Hence, using the assumption that $\lim_i \|T^i_g(a) - a\| = 0$ for every $g \in G$ and $a\in A$,  we get that $\lim_i \|\Phi_{T^i}(f) - f\|_{\rm u}=0$.  Since $\|\Phi_{T^i}\| = \|\Phi_{T^i}(1\odot e) \|= \|T_e^i(1)\| $ for each $i$ and $\{T^i\}$ is  uniformly bounded by assumption, we get that $\sup_i  \|\Phi_{T^i}\| < \infty$. Hence, by a density argument, it follows that
$\lim_i \Phi_{T^i}(x) =x\,$. Since $\Phi_{T^i}(x) = 0$ for each $i$\,, we conclude that $x =0$. This shows that $\Lambda_\Sigma $ is injective, i.e., $\Sigma$ is regular, as desired.   

The final assertion may be shown by adapting the proof of similar statements in the existing literature  (see e.g.\ \cite{AD2, BrOz, Ec}). 
For completeness, we sketch a proof. Assume that $A$ is nuclear and let $B$ be a  C$^*$-algebra. Without loss of generality, we may assume that $B$ is unital. We can then form the product system $\Sigma' = (A\otimes B, G, \alpha \otimes {\rm id}, \sigma \otimes 1)$, and it is almost  immediate that the amenability of $\Sigma$ passes to $\Sigma'$. Hence, using the first assertion with $\Sigma'$ instead of $\Sigma$, we get
$$C_r^*(\Sigma)\otimes B\, \simeq \,C_r^*(\Sigma') \,\simeq \,C^*(\Sigma') \,\simeq \,C^*(\Sigma) \otimes_{\rm max} B \,\simeq \,C_r^*(\Sigma) \otimes_{\rm max} B\,.$$ 
Thus, $C_r^*(\Sigma)\simeq C^*(\Sigma)$ is nuclear. Conversely, nuclearity of $A$ is necessary for this to hold since there exists a conditional expectation from $C_r^*(\Sigma)$ onto $A$. 
\end{proof}

\begin{remark}
In view of Theorem \ref{amen}, several questions arise. Does the regularity of $\Sigma$ imply its amenability ? Does the amenability of $\Sigma$ imply that $\Sigma$ has the weak approximation property ?   Do we get a strictly stronger notion of amenability by requiring that  $ T_e^i(1) = 1$ for all $i$  instead of just saying that the net $\{T^i\}$ is uniformly bounded ? 
\end{remark}

\subsection{On some commutative subalgebras of the Fourier-Stieltjes algebra}

It is clear that $B(\Sigma)$ contains several interesting commutative subalgebras, e.g.\ the canonical copy of $B(G)$, which is obviously contained in the center $ZB(\Sigma)$ of $B(\Sigma)$. We note that  
\begin{align*}
ZB(\Sigma)&=\{T \in B(\Sigma) \ | \ T \times T' = T' \times T \ \text{for all } T' \in B(\Sigma)\}\\
&= \{T \in B(\Sigma) \ | \ T \times T' = T' \times T \ \text{for all } T' \in P(\Sigma)\} \,. 
\end{align*}

We describe some properties of elements in $ZB(\Sigma)$. 

\begin{proposition}\label{centprop}
Let $T=T_{\rho,v,x,y} \in ZB(\Sigma)$. Then one has
\begin{itemize}
\item[$(i)$] $\big\langle x,y \big\rangle \in Z(A)$.
\item[$(ii)$] More generally, $T(g,a) = \big\langle x, \rho(a)v(g) y \big\rangle \in Z(A)$ for all $a \in Z(A)$ and $g \in G$.
\item[$(iii)$]  $T(g,a) = \big\langle x, \rho(a)v(g) y \big\rangle = a \,\big\langle x, v(g)y \big\rangle = \big\langle x, v(g)y \big\rangle\, a$ for all $a \in A$ and $g \in G$. 

\smallskip In other words, $T = L^\varphi= R^\varphi$ where $\varphi: G \to Z(A)$ is given by $\varphi(g) := \langle x, v(g) y \rangle$.
\item[$(iv)$] $T_{\rho,v,y,x} \in ZB(\Sigma)$.
\end{itemize}
\end{proposition}

\begin{proof}
Let $T' = T_{\ell,\alpha,b,c}$ be any coefficient of the trivial equivariant representation,
i.e. $T'(g,a) = b^* a \,\alpha_g(c)$ for some $b,c \in A$. 
As $T$ commutes with $T'$, we get  that
\begin{equation}\label{center}
\big\langle x, \rho(b^* a \alpha_g(c)) v(g) y \big\rangle = b^* \, \big\langle x, \rho(a)v(g)y \big\rangle \,\alpha_g(c) \ , 
\end{equation}
for all $a,b,c \in A$ and $g \in G$.\\
$(i)$ Plugging into the equation (\ref{center}) $a=1$, $b = c =:u \in \mathcal{U}(A)$ and $g=e$, we get
$$\langle x, y \rangle = \langle x, \rho(u^* u) y \rangle = u^*\, \langle x,y \rangle \,u\,.$$
This implies
that $\langle x, y \rangle \in A$ commutes with the whole of $\U(A)$, 
and thus with $A$.\\
$(ii)$ The argument is similar. Given $a\in Z(A)$ and $g\in G$, we now choose $c\in \mathcal{U}(A)$ and set $b = \alpha_g(c) \in \U(A)$, and plug this into equation (\ref{center}). Then we let $c$ range over $\U(A)$.\\
$(iii)$ Now, plugging   $a = 1 = c$ and $b = a^*$ into equation  (\ref{center}),
we get the first equality. The second one follows from $(ii)$.\\ 
$(iv)$ This can be checked by brute force. Alternatively, one can use that $ZB(\Sigma)$ is closed under the conjugation in $B(\Sigma)$ (see Remark \ref{BM0}), and that $(T_{\rho,v,x,y})^c = T_{\rho,v,y,x} $\,.

\end{proof}

\begin{corollary}\label{strcent}
Let $T \in ZB(\Sigma)$. Then $\varphi: G\to A$ given by $\varphi(g) = T(g,1)$ takes its values in $Z(A)$, and we have $T = L^{\varphi}= R^{\varphi}$. For any equivariant representation $(\rho,v)$ of $\Sigma$ on some Hilbert $A$-module $X$ and $x, y \in X$ such that  $T= T_{\rho,v,x,y} $, we have $\varphi (g) = \langle x, v(g) y \rangle$. 
\end{corollary}

We may also consider
$$BZ(\Sigma) = \{T \in L(\Sigma) \ | \ T = T_{\rho,v,x,y} \ \mbox{ for some equiv.\ rep.\ $(\rho,v)$ of $\Sigma$ on $X$ and $x, y \in Z_X$}\}\,.$$
 It is not difficult to check that $BZ(\Sigma)$ is a commutative subalgebra of $B(\Sigma)$,
which contains the canonical copy of $B(G)$, as follows immediately from Remark \ref{centvect}.
 Similar to the case where $T \in ZB(\Sigma)$, we have:
 
 \begin{proposition}\label{BZprop}
 Let
$T \in BZ(\Sigma)$. Then $\varphi: G\to A$ given by $\varphi(g) = T(g,1)$ takes its values in $Z(A)$, and we have $T = L^{\varphi}= R^{\varphi}$. For any equivariant representation $(\rho,v)$ of $\Sigma$ on some Hilbert $A$-module  $X$ and $x, y \in Z_X$ such that  $T= T_{\rho,v,x,y} $, we have $\varphi (g) = \langle x, v(g) y \rangle$.
\end{proposition}
\begin{proof}
Assume $T = T_{\rho,v,x,y}$ with $x,y \in Z_X$.
Then, using equation (\ref{eqcentr}), we get that  $T = L^\varphi = R^\varphi$, where $\varphi: G \to Z(A)$ is given by 
$\varphi(g) = \langle x, v(g)y \rangle$. But then $T(g,1)= \varphi(g)$ for all $g\in G$, so the proposition follows. 
\end{proof} 

The two previous results show that $ZB(\Sigma)$ and $BZ(\Sigma)$ have some common features. 
It seems to us that it should be interesting to understand the structure of these algebras, and also to investigate their relative positions. We include below a few results concerning these issues. We first introduce
$$PZ(\Sigma) = \{T \in L(\Sigma) \ | \ T = T_{\rho,v,x,x} \ \mbox{ for some equiv.\ rep.\ $(\rho,v)$ of $\Sigma$ on $X$ and $x \in Z_X$}\}\,, $$
which is a cone in $BZ(\Sigma)$ containing the canonical image of $P(G)$, as the reader will easily verify, using Remark \ref{pdclassic} for the last part.  

\begin{proposition} \label{PZspan}
We have $BZ(\Sigma) = {\rm Span}\,PZ(\Sigma)$. Moreover, 
\begin{align*}
PZ(\Sigma)&=  \{ L^\varphi\, | \, \mbox{$\varphi: G \to Z(A)$ is AD-positive $($w.r.t. $\Sigma)$}\} \\
&=P(\Sigma) \cap BZ(\Sigma)\,.
\end{align*}
\end{proposition}

\begin{proof}
The first statement  is a consequence of the polarization identity, 
after noticing that if $x,y \in Z_X$ then $x + i^k y \in Z_X$, for $k=0,1,2,3$.
The second one is a consequence of Proposition \ref{L-phi} and Proposition \ref{BZprop}.
\end{proof}

\begin{lemma}\label{xcentral}
Assume \,$T \in ZB(\Sigma)$ can be written as \,$T=T_{\rho,v,x,y}$ for some equivariant representation $(\rho,v)$ on a Hilbert $A$-module $X$ and $x,y \in X$, such that, in addition, $y$ is cyclic for $(\rho,v)$, namely
$${\rm Span} \Big\{ \big(\rho(a)v(g)y\big) \cdot a' \ | \ a,a' \in A,\, g \in G\Big\}$$
is dense in $X$. Then $x \in Z_X$.
\end{lemma}

\begin{proof}
Let $T' = T_{\ell,\alpha,b,c}$ be any coefficient of the trivial equivariant representation,
i.e. $T'(g,a) = b^* a \,\alpha_g(c)$ for some $b,c \in A$. 
As $T$ commutes with $T'$, we get  that
$$\big\langle x, \rho\big(b^* a \alpha_g(c)\big) v(g) y \big\rangle = b^* \,\big\langle x, \rho(a)v(g)y \big\rangle\, \alpha_g(c)\,,$$
 hence 
$$\big\langle \rho(b)x, \rho(a)v(g) \rho(c)y \big\rangle = \big\langle x \cdot b, \rho(a)v(g)(y \cdot c) \big\rangle \, $$
for all $g \in G$ and $ a, b,c \in A$.

In particular, substituting $c = 1$, we obtain 
$ \big\langle \rho(b)x, \rho(a)v(g) y \big\rangle = \big\langle x \cdot b, \rho(a)v(g)y \big\rangle $ and thus
$$\big\langle \rho(b)x, (\rho(a)v(g) y) \cdot a' \,\big\rangle = \big\langle x \cdot b, (\rho(a)v(g)y)\cdot a' \,\big\rangle $$
for every $g \in G$ and $ a,a', b \in A$.
Using the cyclicity of $y$, this identity in turn implies that $\rho(b)x = x \cdot b$ for all $b \in A$, i.e., $x \in Z_X$, as claimed.

\end{proof}

\begin{proposition}

$$P(\Sigma) \cap ZB(\Sigma) \subset PZ(\Sigma) \,. $$
\end{proposition}

\begin{proof}
This follows at once from Theorem 4.16 and Lemma \ref{xcentral}.

\end{proof}

All in all, we have the following pattern of inclusions: 

$$
\begin{array}{ccccccc}
B(G) & \hookrightarrow & BZ(\Sigma) \cap ZB(\Sigma) & \subset & BZ(\Sigma) & \subset & B(\Sigma) \\
\cup &  & \cup & & \cup & & \cup \\
P(G) & \hookrightarrow & \,\,P(\Sigma) \cap ZB(\Sigma) & \subset & PZ(\Sigma) = P(\Sigma) \cap BZ(\Sigma) & \subset & P(\Sigma)
\end{array}
$$

\smallskip \noindent There are some inherent difficulties that one has to be able to handle before it will be possible to obtain a better picture. We illustrate this in the following observation. 
 
\begin{remark} \label{explcenter} 
It is obvious that the inclusion $B(G)\subset B(\Sigma) \cap M_0A(G)$ holds. It seems reasonable that the converse inclusion should hold, but it is in fact not easy to show this. Indeed,  
consider an equivariant representation $(\rho,v)$ of $\Sigma$ on a Hilbert $A$-module $X$,  let
$x,y \in X$ and assume that we have  $T:= T_{\rho,v,x,y}=T^\varphi$, where $\varphi(g)=\langle x, v(g)y \rangle \in {\mathbb C}$ for all $g \in G$. Then one readily checks by direct computation that $T \in ZB(\Sigma)$.
If we assume that $x=y$, then $T\in P(\Sigma)$, so it follows  from Remark \ref{pdclassic} that $\varphi \in P(G)$; in particular $T\in PZ(\Sigma) \subset BZ(\Sigma)$.
 However, when $x\neq y$, it is not clear that an element $T$  given as above has to lie in $BZ(\Sigma)$, or even in $B(G)$. Proposition \ref{trivcent} sheds some light on this problem in the case where $A$ has trivial center. \end{remark}

\begin{proposition}\label{trivcent}
Suppose that $A$ has trivial center. Then $ T \in ZB(\Sigma)$ if and only if $T = T_{\rho,v,x,y} = T^\varphi$, where $\varphi(g):= \langle x, v(g)y \rangle \in {\mathbb C}$ for all $g \in G$. In particular, $$BZ(\Sigma) \subset ZB(\Sigma)\,$$ Moreover, $P(G)$ and $B(G)$ are in bijective correspondences with $PZ(\Sigma)$ and $BZ(\Sigma)$, respectively, via the map $\phi \mapsto T^\phi$. Finally, 
$$P(\Sigma) \cap ZB(\Sigma) = PZ(\Sigma) = P(\Sigma) \cap BZ(\Sigma)\,.$$
\end{proposition}

\begin{proof}
The first statement follows from Proposition \ref{centprop} and Remark \ref{explcenter}. The second statement follows from Proposition \ref{BZprop} in conjugation with the first statement. Since a scalar-valued AD-positive definite function (w.r.t.\ $\Sigma$) is nothing but a positive definite function on $G$, it is clear from Proposition \ref{PZspan} that we may identify  $P(G)$ with $PZ(\Sigma)$ via  the map $\phi \mapsto T^\phi$, and then also $B(G)$ with $BZ(\Sigma)$, by polarization. The last claim follows then from the second statement in combination with Proposition \ref{BZprop}.
\end{proof}

\subsection{On algebras of completely bounded multipliers}

 Following M.\ Walter \cite{Wal1,Wal2}, we recall that
if $B$ is a $C^*$-algebra, then one may consider its {\it dual algebra} $\mathcal{D}(B)$,
that consists of all completely bounded maps from $B$ into itself,  the
product being given by composition. Equipped with the completely bounded norm, 
$\mathcal{D}(B)$ becomes a Banach algebra, with an isometric conjugation 
$\Phi \to \Phi^c$ given by 
$$\Phi^{c}(b) = (\Phi(b^*))^* $$
for each $\, b \in B$\, (cf.\  \cite[Proposition 1]{Wal1}). 

In the sequel, for brevity, we  focus on the reduced situation where $B=C_{\rm r}^*(\Sigma)$. Undoubtedly, many of the statements that follow  admit a full version where $B=C^*(\Sigma)$, mutatis mutandis. (See also Remark \ref{full-cb}).

\medskip 
For $T \in M_0A(\Sigma)$ we set $\| T \|_{\rm cb}  = \|M_T\|_{\rm cb}$.  We first show that $M_0A(\Sigma)$ embeds in a canonical way in $\mathcal{D}(C_{\rm r}^*(\Sigma))$. 

\begin{proposition} $M_0A(\Sigma)$ is a unital subalgebra of $L(\Sigma)$. Moreover, 
 $(M_0A(\Sigma), \|\cdot\|_{\rm cb})$ is a Banach algebra, with an isometric conjugation 
  $T\to T^c$  given by
  $$T^c(g,a)  = \sigma(g,g^{-1})^*\, \alpha_g\Big(T_{g^{-1}}\big(\alpha_g^{-1}\big(a^* \sigma(g,g^{-1})^*\big)\big)\Big)^*\,.$$ 
 The map $T\to M_T$ from $M_0A(\Sigma)$ into $\mathcal{D}(C_{\rm r}^*(\Sigma))$ is a unital injective  Banach algebra homomorphism that respects conjugation.
 \end{proposition}
 
\begin{proof}  
 If $T, T' \in M_0A(\Sigma)$,  then $M_T\circ M_T' \in \D(C_{\rm r}^*(\Sigma))$ and it is straightforward to check that $$(M_T\circ M_T')\big(\Lambda_\Sigma(f)\big) = \Lambda_\Sigma\big((T \times T')\cdot f \big)$$
for all $f\in C_c(\Sigma)$. Hence, it follows that $T \times T' \in M_0A(\Sigma)$, with $M_{T \times T'} = M_T\circ M_T'$. Proceeding in the same way with the other operations, and using that the map $T\to M_T$ is clearly injective, we see that $(M_0A(\Sigma), \|\cdot\|_{\rm cb})$ becomes a unital normed algebra with isometric conjugation  such that
 $T \mapsto M_T$ is a unital algebra homomorphism  that respects conjugation. 
  
Moreover, the range $R$ of the map $T\to M_T$ is closed in $\mathcal{D}(C^*_{\rm r}(\Sigma))$. Indeed, let $\{T_n\}$ be a sequence in $M_0A(\Sigma)$ such that 
$\|M_{T_n} - \Phi\|_{\rm cb} \to 0$ for some $\Phi \in \D(C^*_{\rm r}(\Sigma))$. Then for  $g\in G$ and $a\in A$, we have
$$\Phi\big(a \lambda_\Sigma(g)\big)\lambda_\Sigma(g)^* 
= \lim_n \, M_{T_n}\big(a \lambda_\Sigma(g)\big)\lambda_\Sigma(g)^* =  \lim_n \, T_n(g,a)\,.$$
 Set $T(g,a) = \Phi\big(a \lambda_\Sigma(g)\big)\lambda_\Sigma(g)^*$ for $g\in G$ and $a\in A$. Since $A$ is closed in $C_{\rm r}^*(\Sigma)$,  $T(g,a)$ lies in $A$. Thus, we get a map $T$ from $G\times A$ into $A$ that is linear in the second variable.  As we have $\Phi\big(a \lambda_\Sigma(g)\big) = T_g(a) \lambda_\Sigma(g)$ for all $g\in G$ and all $a\in A$, we get that $\Phi(\Lambda_\Sigma(f)) = \Lambda_\Sigma(T\cdot f)$ for all $f \in C_c(\Sigma)$. This shows that $T\in M_0A(\Sigma)$ and $\Phi = M_T$, that is, $\Phi \in R$.  
This shows that $R$ is closed in $\mathcal{D}(C^*_{\rm r}(\Sigma))$. Hence, $R$ is complete, and it is then clear  that $(M_0A(\Sigma), \|\cdot\|_{\rm cb})$ is also complete.

\end{proof}

\begin{remark} \label{BM0}
It is clear from Theorem \ref{equiv-coeff} that $B(\Sigma)$ is a subalgebra of $ M_0A(\Sigma) $, such that $ \| T \|_{\rm cb} \leq \|T\|\,$ for all $T\in B(\Sigma)$. However, it is not obvious that $B(\Sigma)$ is necessarily closed w.r.t.\ $\|\cdot\|_{\rm cb}$ in $M_0A(\Sigma)$. Of course, this will be the case if $ \| T \|_{\rm cb} = \|T\|\,$ for all $T\in B(\Sigma)$ or if $B(\Sigma) = M_0A(\Sigma)$. It would be interesting to know when one or both of these conditions are satisfied. For instance, when $A$ and $\sigma$ are trivial, so $B(\Sigma)=B(G)$ and $M_0A(\Sigma)=M_0A(G)$, it is known (as already recalled in Remark \ref{PiBoDCH}\,; see \cite[Corollary 1.8]{DCHa} and \cite[Theorem]{Bo}) that $B(G) = M_0A(G)$ if and only if $G$ is amenable, in which case the two norms agree. 
 
 We also note that $B(\Sigma)$ is closed under the conjugation in $M_0A(\Sigma)$. Indeed, if $T\in B(\Sigma)$ and $(\rho,v)$ is an equivariant representation of $\Sigma$ on $X$ such that
 $T=T_{\rho, v, x, y}$\, for some $x,\,y \in X$, then one checks by direct computation 
 that  $T^c = T_{\rho, v, y, x}$\,, so $T^c \in B(\Sigma)$. 
 
 Moreover,  
  $\|T^c\| \leq \|y\|\|x\|$. Taking the infimum over all possible $x$ and $y$ as above, we get $\|T^c\| \leq \|T\|\,.$
 As conjugation is involutive, the converse inequality also holds. Hence, the conjugation on $B(\Sigma)$ is also isometric  w.r.t. $\|\cdot\|$.
 
 It is clear that $ZB(\Sigma)$ is closed under conjugation, as is the case of the center of any algebra with a conjugation.  Finally, $BZ(\Sigma)$ also shares this property: if $T \in BZ(\Sigma)$, so $T=T_{\rho, v, x, y}$ with $x,y \in Z_X$, then 
 $T^c = T_{\rho, v, y, x} \in BZ(\Sigma).$  
 \end{remark}
 
\begin{remark}\label{M0AG} 
As shown in \cite[Corollary 4.7]{BeCo3}, we have $T^\varphi \in M_0A(\Sigma)$ whenever
$\varphi$ lies in $ M_0A(G)$, in which case
 $\| T^\varphi \|_{\rm cb} \, \leq \| \varphi \|_{\rm cb} $ (where $\| \varphi \|_{\rm cb}$ denotes the norm in $M_0A(G)$).
Hence, the map $\varphi \to T^\varphi$ provides a continuous embedding of $M_0A(G)$ into $M_0A(\Sigma)$. 
An interesting question is whether this map 
is isometric. Another open question is as follows. Assume that $G$ is not amenable, so that there exists $\varphi \in M_0A(G)\setminus B(G)$. Is then 
$T^\varphi \not\in  B(\Sigma)$ ? (See Remark \ref{explcenter} for a related problem).
\end{remark}

Next, we recall that a linear map $M: C^*_{\rm r}(\Sigma) \to C^*_{\rm r}(\Sigma)$ is called an {\it $A$-bimodule map} 
 if $$M(a\,x\,a') = a\,M(x)\,a'$$ for all $a,a' \in A$ and $x \in C^*_{\rm r}(\Sigma)$.
We set
$$M^{\rm bim}_0A(\Sigma) = \{ T \in M_0A(\Sigma) \ | \ M_T \ \mbox{is an $A$-bimodule map}\}\,. $$
As follows from Remark \ref{DongR} and Proposition \ref{PZspan}, we have $PZ(\Sigma) \subset M^{\rm bim}_0A(\Sigma)\,.$ Hence,
 $$BZ(\Sigma) \subset M^{\rm bim}_0A(\Sigma)\,$$
 since $M^{\rm bim}_0A(\Sigma)\,$ is obviously a subspace of $M_0A(\Sigma)$. In fact, we have:

\begin{proposition} 
$\Mb$ is a Banach subalgebra of $M_0A(\Sigma)$ which is closed under conjugation.
\end{proposition}

\begin{proof}
As the composition of two $A$-bimodule maps is an $A$-bimodule map,
it is clear that $\Mb$ is a subalgebra of $M_0A(\Sigma)$. 
Moreover, if $M$ is an $A$-bimodule map, then 
$$M^c(a x a') = M({a'}^* x^* a^*)^* = ({a'}^* M(x^*) a^*)^* = a M^c(x) a'$$  
for all
$a,a' \in A$ and $ x \in C^*_{\rm r}(\Sigma)$. Hence $M^c$
is an $A$-bimodule map.
It follows that if $T \in \Mb$, then $M_{T^c} = (M_T)^c$ is an $A$-bimodule map, so $T^c \in \Mb$.
Thus $\Mb$ is closed under conjugation. 
Finally, $\Mb$ is closed in $M_0A(\Sigma)$. Indeed, assume $\{T_n\}$ is a 
sequence in $\Mb$, $T \in M_0A(\Sigma)$ and $T_n \to T$ 
(in norm).
Then $$M_T(a x a') =\lim_n M_{T_n} (a x a') = \lim_n {a M_{T_n}(x) a'} = a M_T(x) a'$$ for all $a,a' \in A$ and $ x \in C^*_{\rm r}(\Sigma)$. Hence, $T \in \Mb $. 

\end{proof}

Recall that for $\varphi:G\to A$, we have defined $L^\varphi\,,\, R^\varphi: G \times A \to A$ by
 $$L^\varphi(g,a) = \varphi(g)\, a\,, \quad  R^\varphi(g,a) = a\, \varphi(g)$$ 
 for all $g \in G, a \in A$.
Set 
$$LM_0A(\Sigma) 
= \{L^\varphi \ | \ \varphi:G\to A \text{ and } L^\varphi \in M_0A(\Sigma)\} \,, $$
$$RM_0A(\Sigma) 
= \{R^\varphi \ | \ \varphi:G\to A \text{ and } R^\varphi \in M_0A(\Sigma)\} \,. $$
These subspaces of $M_0A(\Sigma)$ are non-empty since they both contain $BZ(\Sigma)$ (as follows from Proposition \ref{PZspan}).
Moreover, they both contain a copy of $M_0A(G)$: if $\varphi \in M_0A(G)$, regarding $\varphi$ as a function from $G$ to $\Complessi \cdot 1 \subset A$, we have $L^\varphi = R^\varphi$ \, =  $T^
\varphi \in M_0A(\Sigma)$ (cf.\ Remark \ref{M0AG}). 

\begin{proposition}
$LM_0A(\Sigma)$ and $RM_0A(\Sigma)$ are Banach subalgebras of $M_0A(\Sigma)$, satisfying
$$\big(LM_0A(\Sigma)\big)^c = RM_0A(\Sigma), \quad \big(RM_0A(\Sigma)\big)^c = LM_0A(\Sigma)\,.$$
We also have
$$LM_0A(\Sigma) \cap RM_0A(\Sigma)= \big\{ L^\varphi \mid \, \varphi:G\to Z(A) \ \text{and} \ L^\varphi\in M_0A(\Sigma)\big\}\,.$$
\end{proposition}

\begin{proof}
Assume $L^\varphi \in LM_0A(\Sigma)$.  Then we know that 
$\big(L^\varphi\big)^c \in M_0A(\Sigma)$.
Moreover, for all $g\in G$ and $a\in A$, we have
\begin{align*}
\big(L^\varphi\big)^c(g,a) 
& = \sigma(g,g^{-1})^* \alpha_g\Big(L^\varphi\big(g^{-1},\alpha_g^{-1}(a^* \sigma(g,g^{-1})^*)\big)\Big)^* \\
& = \sigma(g,g^{-1})^* \big(\alpha_g(\varphi(g^{-1}))\,a^*\, \sigma(g,g^{-1})^*\big)^* \\
& = a\, \alpha_g(\varphi(g^{-1}))^* \\
& = R^{\varphi^c}(g,a) \,,
\end{align*}
where 
$\varphi^c(g) := \alpha_g(\varphi(g^{-1}))^*$. So 
$R^{\varphi^c}=\big(L^\varphi\big)^c \in M_0A(\Sigma)$. 
Hence, 
$\big(L^\varphi\big)^c \in RM_0A(\Sigma)$. 
This shows that 
$ \big(LM_0A(\Sigma)\big)^c \subset RM_0A(\Sigma)$. 
The inclusion 
$\big(RM_0A(\Sigma)\big)^c \subset LM_0A(\Sigma)$
 may be shown in a similar way. By conjugation, we obtain that the opposite inclusions both hold. The final assertion is an easy exercise.
\end{proof}

\begin{remark} Let $L=L^\varphi \in LM_0A(\Sigma)$ and $R=R^\psi \in RM_0A(\Sigma)$. Then we have
$$(L \times R)(g,a) = \varphi(g) \,a \,\psi(g) = (R \times L)(g,a)  $$
for all $g \in G$ and $a \in A$, so $LM_0A(\Sigma)$ and $RM_0A(\Sigma)$ commute with each other.

Hence the subalgebra of $M_0A(\Sigma)$ generated by $LM_0A(\Sigma)$ and $RM_0A(\Sigma)$ is the span
of $LM_0A(\Sigma) \times \big(LM_0A(\Sigma)\big)^c$, 
and is closed under conjugation.
\end{remark}

\begin{remark}\label{DR1} Let $\varphi:G\to A$. Then it is easy to see that $L^\varphi \in  LM_0A(\Sigma) \cap \Mb$ if and only if $\varphi$ takes its values in $Z(A)$ and $L^\varphi \in M_0A(\Sigma)$. (When $\sigma=1$, this follows as in \cite[p. 436]{DoRu}; the argument is the same when $\sigma$ is nontrivial).  Hence, all in all, we get
\begin{align*}
BZ(\Sigma) &\subset \, LM_0A(\Sigma) \cap \Mb \\
&= \{L^\varphi \ | \, \varphi:G\to Z(A)
 \text{ and }  L^\varphi \in M_0A(\Sigma)\}\\
 &= LM_0A(\Sigma) \cap RM_0A(\Sigma)\\
 & =RM_0A(\Sigma) \cap \Mb\,.
\end{align*}

\vspace{-3ex} \noindent Setting 
$$LM_{\rm cp}A(\Sigma) 
= \{L^\varphi \ | \ \varphi:G\to A,\, L^\varphi \in M_0A(\Sigma) \text{\, and $M_{L^\varphi}$ is completely positive}\} \,, $$
we can also add to Proposition \ref{PZspan} that 
$$PZ(\Sigma) \, = \, LM_{\rm cp}A(\Sigma) \cap \Mb \,.$$ 
\end{remark}

 \begin{remark} \label{full-cb} What we have done so far in this subsection concerns subalgebras of the Banach algebra $(M_0A(\Sigma), \|\cdot\|_{\rm cb})$. As indicated earlier, one can also introduce analogous subalgebras  of the Banach algebra $M_{\rm cb}^{\rm u}(\Sigma) $ (with respect to the norm  $\|T\|^{\rm u}_{\rm cb} = \|\Phi_T\|_{\rm cb} $) that will satisfy similar properties. The Fourier-Stieltjes algebra $B(\Sigma)$ is of course a subalgebra of $M_{\rm cb}^{\rm u}(\Sigma) $ such that $\|T\| \leq \|T\|^{\rm u}_{\rm cb}$ for all $T\in B(\Sigma)$ (cf.\ Theorem \ref{equiv-coeff-full}). When $A$ and $\sigma$ are trivial, we have $B(\Sigma)=B(G) = M_{\rm cb}^{\rm u}(\Sigma) $ and the two norms agree (see\,\cite{Wal1, Pis2}). The general case is more elusive, cf.\ Remark \ref{PiBoDCH}. 
\end{remark}

\medskip Finally we mention that the algebra $M_0A(\Sigma)$ has a right-handed version: it is the Banach algebra $M_0A'(\Sigma)$ consisting of the maps $S\in L(\Sigma)$ having the property that there exists a
(necessarily unique) map $M'_S \in \D(C^*_{\rm r}(\Sigma))$ satisfying 
$$M'_S(\lambda_\Sigma(g)\,a) = \lambda_\Sigma(g)\, S_g(a)$$ for all $a \in A$ and  $g \in G$,  the norm of such a map $S$ being then defined by $\|S\| = \|M'_S\|_{\rm cb}$. Such a framework is for instance used by Dong and Ruan in \cite{DoRu} in the the special case where $\sigma=1$ and $S= L^\varphi$ for some $\varphi:G \to Z(A)$.
The resulting theory is parallel to the one we have outlined. For completeness, we describe below how the involved algebras are related and leave the proof of the following two propositions to the reader. 

\medskip 
We let $P'(\Sigma)$ consist of the maps $S\in L(\Sigma)$ that are such that for any $n\in \Naturali$, $g_1, \ldots, g_n \in G$ and $a_1, \ldots, a_n \in A$, the matrix
  $$\Big[ \,\sigma(g_i, g_i^{-1}g_j)\, \alpha_{g_j}\Big(S_{g_i^{-1}g_j}\big(\alpha_{g_j}^{-1}\big(\,\sigma(g_i, g_i^{-1}g_j)^*\,a_i^*\,a_j\big)\,\big)\Big)\, \Big]$$
is positive in $M_n(A)$. Moreover, we set $B'(\Sigma) ={\rm span} \, P'(\Sigma)$ and
$$\Mbp = \{ S \in M_0A'(\Sigma) \ | \ M'_S \ \mbox{is an $A$-bimodule map}\}\,. $$

\begin{proposition} Let  $T \in L(\Sigma)$ and define $\widetilde{T}\in L(\Sigma)$ by 
$$\widetilde{T}(g,a) = \alpha_g^{-1}(T(g,\alpha_g(a)))\,$$
for all $g\in G$ and $a\in A$.
The following statements hold:
\begin{itemize}
\item[$i)$] If $T \in M_0A(\Sigma)$, then $\widetilde{T} \in M_0A'(\Sigma)$ and $M'_{\widetilde{T}} = M_T$\,.
\item[$ii)$] If $T \in P(\Sigma)$, then $\widetilde{T} \in P'(\Sigma)$.
\item[$iii)$] The map $\tau: T \mapsto \widetilde{T}$  gives an isometric algebra isomorphism from $M_0A(\Sigma)$ onto $M_0A'(\Sigma)$, that maps $P(\Sigma)$ onto $P'(\Sigma)$, 
$B(\Sigma)$ onto $B'(\Sigma)$ and 
$\Mb$ onto $\Mbp$.
\item[$iv)$] $\Mbp$ is a Banach subalgebra of  $M_0A'(\Sigma)$ which is closed under conjugation.
\end{itemize}
\end{proposition}

The map $\tau$ can be used to transport the norm $\|\cdot\|$ on $B(\Sigma)$ to a norm on $B'(\Sigma)$, turning $B'(\Sigma)$ into a Banach algebra isometrically isomorphic to $B(\Sigma)$. Alternatively, we could have defined $B'(\Sigma)$ directly as coefficient maps of suitably defined equivariant representations of $\Sigma$ on left Hilbert $A$-modules, but the above approach is shorter. 

\medskip We will denote the inverse of $\tau$ just by $S \mapsto \widehat{S}$. We also set
$$LM_0A'(\Sigma) 
= \{L^\varphi \ | \ \varphi:G\to A \text{ and } L^\varphi \in M_0A'(\Sigma)\} \,, $$
$$RM_0A'(\Sigma) 
= \{R^\varphi \ | \ \varphi:G\to A \text{ and } R^\varphi \in M_0A'(\Sigma)\} \,. $$

\begin{proposition}
 $LM_0A'(\Sigma)$ and $RM_0A'(\Sigma)$
are Banach subalgebras of $M_0A'(\Sigma)$. They satisfy 
$$ \big(LM_0A'(\Sigma)\big)^c = RM_0A'(\Sigma)\,, \quad 
\big(RM_0A'(\Sigma)\big)^c = LM_0A'(\Sigma)\,,$$
$$\widetilde{LM_0A(\Sigma)} = LM_0A'(\Sigma)\,, 
\quad \widehat{LM_0A'(\Sigma)} = LM_0A(\Sigma) \,, $$
$$\widetilde{RM_0A(\Sigma)} = RM_0A'(\Sigma)\,, 
\quad \widehat{RM_0A'(\Sigma)} = RM_0A(\Sigma)\,. $$
\end{proposition}

\section{On Fourier-Stieltjes algebras and C$^*$-correspondences}  \label{corresp}

 The aim of this section is to give a description of $B(\Sigma)$  using C$^*$-correspondences over $C_{\rm r}^*(\Sigma)$. We  recall that if $B$ is a C$^*$-algebra, $Y$ is a Hilbert $B$-module and $\phi$ is a homomorphism from $B$ into $\L(Y)$, the triple $(Y,B,\phi)$ is called a {\it a nondegenerate $C^*$-correspondence over $B$} (see e.g.\ \cite[II.7.4.4]{Bl}), or sometimes a right Hilbert $B$-bimodule \cite{EKQR}. The nondegeneracy here is due to our standing assumption that all homomorphisms are unit-preserving. In the sequel, by a $C^*$-correspondence, we always mean a nondegenerate $C^*$-correspondence.

\begin{remark}  \label{eqrep-corresp} 
We first explain how $B(\Sigma)$ may be described in terms of C$^*$-correspondences over $A$. 
Let $(\rho, v)$  be an  equivariant representation of $\Sigma$ on a Hilbert $A$-module $X$. Then $(X,A,\rho)$ is a C$^*$-correspondence over $A$ and, as is usual, we set $a\cdot x = \rho(a)\, x $ for $a \in A$ and $ x \in X$. Property (i) of $(\rho,v)$ may then be rewritten as
\begin{itemize}
\item[(i')] \quad $v(g) (a\cdot x) = \alpha_g(a) \cdot (v(g)x)\,.$
\end{itemize}
Moreover, if $\tilde{\sigma}: G\times G \to \mathcal{I}(X)$ is defined by 
$$[\tilde{\sigma}(g,h)]\, x= \sigma(g,h)\cdot x \cdot \sigma(g,h)^*\,,$$
i.e., $\tilde{\sigma}(g,h)= {\rm ad}_\rho(\sigma(g,h))$, then property (ii) of $(\rho,v)$ says that 
\begin{itemize}
\item[(ii')] \quad $v(g) \, v(h) = \tilde\sigma(g,h) \, v(gh) \,. $
\end{itemize}
When $\sigma$ is trivial, this condition just means that $v$ is a homomorphism from $G$ into $\mathcal{I}(X)$, and $v$ is then a so-called $\alpha$-$\alpha$ compatible action of $G$ on $X$ in the terminology used in \cite{EKQR-0,EKQR}.
In the general case, the map $v$ may be considered as an  $(\alpha, \sigma)$-$(\alpha, \sigma)$ compatible action of $G$ on $X$. We may then say that  $B(\Sigma)$ consists of all the functions  in $L(\Sigma)$ of the form
$$(g,a)\to \big\langle x\,,\, a\cdot v(g)y\big\rangle $$  
where $v$ is some $(\alpha, \sigma)$-$(\alpha, \sigma)$ compatible action of $G$ on some C$^*$-correspondence $X$ over $A$ and $x, y \in X$. We also note  that the Fourier-Stieltjes algebra of a general twisted C$^*$-dynamical system, as considered in \cite{PaRa, PaRa1}, may be defined in a similar way by adapting the notion of compatible action used in \cite{EKQR-0, EKQR} to the twisted case.
\end{remark}

Let now $(\rho, v)$  be an  equivariant representation of $\Sigma$ on a Hilbert $A$-module $X$ and consider $(X,A,\rho)$ as a C$^*$-correspondence over $A$. One may then define the crossed product C$^*$-correspondence $X \rtimes_v G$ over $C^*(\Sigma)$ and its reduced version $X \rtimes_{v,r} G$ over $C_{\rm r}^*(\Sigma)$.
Indeed, as in \cite{EKQR-0,EKQR} when $\sigma$ is trivial, but now applying repeatedly the cocycle identities,
one may show that
the space $C_c(G, X)$ becomes a right pre-Hilbert $C_c(\Sigma)$-bimodule (cf. \cite[Definition 1.22]{EKQR}) 
when equipped with the operations
\begin{align*}
(f\cdot \xi)(h) &= \sum_{g\in G} \,f(g) \cdot \big(v(g) \xi(g^{-1} h)\big) \cdot \sigma(g,g^{-1}h)\\
(\xi\cdot f)(h) &= \sum_{g\in G} \,\xi(g) \cdot \big(\alpha_g\big(f(g^{-1} h)\big) \, \sigma(g,g^{-1}h)\big) \\
\big\langle \xi, \, \eta \rangle (h) &= \sum_{g\in G} \, \alpha^{-1}_g \Big(\big\langle \xi(g), \, \eta(gh)\big\rangle \, \sigma(g,h)^* \Big) \, 
\end{align*} 
for    $f \in C_c(\Sigma), \, \xi, \, \eta \in C_c(G,X)$ and $h\in G$. We skip the tedious computations, as they don't bring any additional information.
As in \cite{EKQR-0}, we may then complete $C_c(G, X)$ with respect to the norm given by $$\|\xi\|=\|\langle \xi,\xi \rangle\|^{1/2}_{\rm u}$$ for $\xi \in C_c(G,X)$, where $\|\cdot \|_{\rm u}$ denotes the norm on $C^*(\Sigma)$, and obtain a C$^*$-correspondence $X\rtimes_v G$ over $C^*(\Sigma)$. 
 Moreover, as in \cite{EKQR}, taking instead the completion with respect to the norm 
$$\|\xi\|'=\|\langle \xi,\xi \rangle\|^{1/2}_{\rm r}$$ on $C_c(G,X)$,
where $\|\cdot \|_{\rm r}$ denotes the norm on $C^*_{\rm r}(\Sigma)$, gives a C$^*$-correspondence  $X\rtimes_{v,r} G$ over $C_{\rm r}^*(\Sigma)$. 

\begin{example}
We consider the trivial equivariant representation $(\ell,\alpha)$ of $\Sigma$ on $X=A$. Then we have 
$X \rtimes_\alpha G = C^*(\Sigma)$ and $X \rtimes_{\alpha,r} G = C^*_{\rm r}(\Sigma)$, considered as C$^*$-correspondences over themselves in the canonical way. 

Indeed, for $f, \xi \in C_c(G,A)$, one gets that $f \cdot \xi = f * \xi$ and $\xi \cdot f = \xi * f$ (twisted convolutions).
Moreover, for $\xi,\eta \in C_c(G,A)$, we have
\begin{align*}
\big\langle \xi, \, \eta \rangle (h) &
= \sum_{g\in G} \, \alpha^{-1}_g \big(\langle \xi(g), \, \eta(gh)\rangle \, \sigma(g,h)^* \big) \\
& = \sum_{g\in G} \, \alpha^{-1}_g \big(\xi(g)^* \, \eta(gh) \big) \, \sigma(g^{-1},g)^*\sigma(g^{-1},gh) \\
& = \sum_{g \in G} \sigma(g^{-1},g)^* \, \alpha_{g^{-1}}\big(\xi(g)^*\eta(gh)\big) \, \sigma(g^{-1},gh) \\
& =  \sum_{g \in G} \sigma(g,g^{-1})^* \, \alpha_g\big(\xi(g^{-1})^*\eta(g^{-1}h)\big) \, \sigma(g,g^{-1}h) \\
& =  \sum_{g \in G} \sigma(g,g^{-1})^* \, \alpha_g\big(\xi(g^{-1})^*\big) \, \alpha_g\big(\eta(g^{-1}h)\big) \, \sigma(g,g^{-1}h) \\
& = \sum_{g \in G} \xi^*(g) \, \alpha_g(\eta(g^{-1}h)) \, \sigma(g,g^{-1}h)\\
& = \big(\xi^* * \eta\big)(h) \ ,  
\end{align*}
for all $h \in G$, hence $\langle \xi,\eta \rangle = \xi^* * \eta$. This means that $C_c(G,X) = C_c(G,A)$ is the canonical right pre-Hilbert bimodule  $C_c(\Sigma)$ over itself. Since
$\|\langle \xi,\xi\rangle\|^{1/2}_{\rm u} = \|\xi^* * \xi\|^{1/2}_{\rm u} = \|\xi\|_{\rm u}$ and
$\|\langle \xi,\xi\rangle\|^{1/2}_{\rm r} = \|\xi^* * \xi\|^{1/2}_{\rm r} = \|\xi\|_{\rm r}$, 
taking the corresponding completions, we thus get $C^*(\Sigma)$ and $C^*_{\rm r}(\Sigma)$ as correspondences over themselves, as asserted.
\end{example}

\medskip Now, 
let $Y$ be a C$^*$-correspondence over $B= C_{\rm r}^*(\Sigma)$. Since $E:B\to A$ is a faithful conditional expectation, we obtain a right Hilbert $A$-module $Y'$ by ''localizing'' via $E$ \cite{La1}. That is, we 
let $A$ act on $Y$ from the right in the obvious way (since $A$ is embedded in $B$) and set $$\langle y, z\rangle_A = E\big(\langle y, z\rangle_B\big) \in A$$
for $y, z \in Y$, where $ \langle y, z\rangle_B$ denotes the $B$-valued inner product on $Y$.
Equipped with $\langle\cdot, \cdot\rangle_A$  and the right action of $A$, $Y$ is a (right) pre-Hilbert $A$-module and  $Y'$ is the Hilbert $A$-module we get after completing  $Y$. 

To lighten our notation, we will write $\lambda(g)$ instead of $\lambda_\Sigma(g)$ in the sequel. For each $a \in A$ and $g \in G$,  we claim that there exist
 $\rho_Y(a) \in \L(Y')$ and $v_Y(g) \in \mathcal{I}(Y')$,  determined by 
$$\rho_Y(a) \, y = a\cdot y$$
$$v_Y(g) \, y = \lambda(g)\cdot y \cdot \lambda(g)^*$$
for all $y\in Y$. Indeed, we first define $\rho_Y(a)$ on $Y$ by the above formula. Then, for each $y, z\in Y$, we have
$$\langle \rho_Y(a) \, y \,,\, z\rangle_A = E\big( \langle a\cdot y\, , \, z\rangle_B\big) = E\big( \langle  y\, , \,a^*\cdot z\rangle_B\big)
= \langle  y \,,\, a^*\cdot z\rangle_A= \langle  y \,,\, \rho_Y(a^*)\, z\rangle_A\,.$$
Moreover, 
$$\|\rho_Y(a)\,y\|_A^2 = \| E(\langle a\cdot y\, , \, a\cdot \, y\rangle_B)\| \leq \|a\|^2 \, \|y\|_A^2$$
since $\langle a\cdot y\, , \, a\cdot \, y\rangle_B \leq \|a\|^2\, \langle y\, , \, y\rangle_B$ (cf.\ \cite[Proposition 1.2]{La1}). 
So $\rho_Y(a)$ is bounded as a linear map from $Y$ into itself, and it extends to a bounded linear map on $Y'$, also denoted by $\rho_Y(a)$. It is then easy to conclude that $\rho_Y(a)$ is adjointable on $Y'$, with $\rho_Y(a)^* = \rho_Y(a^*)$.

The formula for $v_Y(g)$ makes obviously sense on $Y$. Let $y\in Y$. Since
\begin{align*} 
\big\langle v_Y(g)y \, , \, v_Y(g)y\big\rangle_A &= E\big( \langle \lambda(g)\cdot y \cdot \lambda(g)^*\, , \, \lambda(g)\cdot y \cdot \lambda(g)^*\rangle_B\big) \\
&= \alpha_g\Big(E\big( \langle \lambda(g)\cdot y \, , \, \lambda(g)\cdot y \rangle_B\big)\Big) \\
&= \alpha_g\Big( E\big( \langle y \, , \, \lambda(g)^*\lambda(g)\cdot y \rangle_B\big)\Big) \\
&= \alpha_g\Big( E\big( \langle y \, , \, y \rangle_B\big)\Big) \\
&= \alpha_g\big( \langle y \, , \, y \rangle_A\big) \,
\end{align*}

\vspace{-3ex} we get 
$$\| v_Y(g)y \|_A^2 =\|\big\langle v_Y(g)y \, , \, v_Y(g)y\big\rangle_A\| = \|\alpha_g\big( \langle y \, , \, y \rangle_A\big)\| =  \| \langle y \, , \, y \rangle_A\| = \|y\|_A^2\,,$$
hence that $v_Y(g)$ is isometric on $Y$. So $v_Y(g)$ extends to an isometry on $Y'$, that we still denote by $v_Y(g)$.  Moreover, for each $g\in G$, we have
\begin{align*} 
v_Y(g^{-1})v_Y(g)\, y &= (\lambda(g^{-1})\lambda(g)\cdot y \cdot (\lambda(g)\lambda(g))^*\\
&= \sigma(g^{-1}, g) \cdot y \cdot \sigma(g^{-1}, g)^* \\ 
&= {\rm ad}_{\rho_Y} (\sigma(g^{-1},g)) \, y
\end{align*}
for each $y\in Y$. Similarly, we get $v_Y(g)v_Y(g^{-1})\, y = {\rm ad}_{\rho_Y} (\sigma(g,g^{-1})) \, y$. It follows  that  $v_Y(g)$ is invertible on $Y'$, with $(v_Y(g))^{-1} = {\rm ad}_{\rho_Y}(\sigma(g^{-1},g)^*)\, v_Y(g^{-1})
= v_Y(g^{-1}) \, {\rm ad}_{\rho_Y}(\sigma(g,g^{-1})^*)$.

\medskip We then have:

\begin{proposition}\label{rv}
The pair  $(\rho_Y, v_Y)$ is an equivariant representation of $\Sigma$ on $Y'$. 
\end{proposition}

\begin{proof} We write $\rho$ instead of $\rho_Y$ and $v$ instead of $v_Y$ in the proof. 
It is enough to check that the conditions (i)-(iv) hold on $Y$. As we have already done this for (iii), we check the others.

\smallskip
Let $g,\, h \in G, \, a \in A,\, y, z \in Y$. We have
\begin{align*}
\rho (\alpha_g(a))\, v(g) \,y & = \alpha_g(a)\cdot \big(\lambda(g) \cdot y \cdot \lambda(g)^*\big) 
=\big(\alpha_g(a)\lambda(g)\big) \cdot y \cdot \lambda(g)^* \\
& = \big(\lambda(g) \,a \big) \cdot y \cdot \lambda(g)^* 
= \lambda(g)  \cdot ( a\cdot y) \cdot \lambda(g)^* \\
& = v(g)\big(a\cdot y\big) = v(g) \rho(a) \, y\,,
\end{align*}

\vspace{-5ex}\begin{align*}
v(g)\, v(h) \, y & =  \lambda(g) \cdot \big(v(h)\,y\big) \cdot \lambda(g)^* = 
\lambda(g) \cdot \lambda(h)\cdot y \cdot \lambda(h)^*\cdot \lambda(g)^* \\
& = \big(\sigma(g,h) \, \lambda(gh) \big) \cdot y \cdot \big(\lambda(gh)^*\,\sigma(g,h)^*\big)
=  \sigma(g,h) \cdot\big( v(gh) \,y\big) \cdot \sigma(g,h)^* \\
& = \big(\rho(\sigma(g,h))\,v(gh) \,y\big) \cdot \sigma(g,h)^*
={\rm ad}_\rho(\sigma(g,h)) \,v(gh)\,y \,, 
\end{align*}

\vspace{-4ex}\begin{align*}
v(g) (y\cdot a) & = \lambda(g)\cdot (y\cdot a)\cdot \lambda(g)^* = \lambda(g)\cdot y\cdot \big(\lambda(g)^*\,\lambda(g) \,a \lambda(g)^*\big) \\
& =  \lambda(g)\cdot y\cdot \lambda(g)^* \cdot \alpha_g(a) = \big(v(g)\, y\big) \cdot \alpha_g(a)\,.
\end{align*}

\vspace{-2ex}\end{proof}

\begin{example} We consider $Y = C^*_{\rm r}(\Sigma)$ as a $C^*$-correspondence over itself in the obvious way.
For $\xi,\eta \in C_c(G,A)$ we have 
$$\langle \xi, \eta \rangle_A = E(\xi^* * \eta) = (\xi^* * \eta)(e) = \sum_{g \in G}\alpha^{-1}_g\big(\xi(g)^* \eta(g)\big) \ . $$
Thus, the inner product 
$\langle \cdot, \cdot \rangle_A$ coincides with  the inner product of the Hilbert A-module $A^\Sigma$ when both are restricted to functions in $C_c(G,A)$.  
As the C$^*$-module norm $\| \cdot \|_A$  associated with $\langle \cdot, \cdot \rangle_A$ is majorized by the $C^*$-algebra norm on $Y=C^*_{\rm r}(\Sigma)$, it follows 
that the completion $Y'$ of $Y$, as a $C^*$-correspondence over $A$, may be identified with $A^\Sigma$.

Thus, using Proposition \ref{rv}, we obtain an equivariant representation $(\rho_Y,v_Y)$ of $\Sigma$ on $Y' = A^\Sigma$. By the very definition of $\rho_Y$, we get $\rho_Y = \ell_\Sigma$. 
Also, on the dense subspace $C_c(G,A)$ of $A^\Sigma$, each $v_Y(g)$ is given by 
$$ [v_Y(g)\, \xi](h) = \alpha_g\big(\xi(g^{-1}hg)\big)\, \sigma(g,g^{-1}hg) \, \sigma(h,g)^*\,. $$ 

Indeed, $v_Y(g)$ is given on $Y = C^*_{\rm r}(\Sigma)$ by
$v_Y(g)\, y = \lambda(g) \, y \, \lambda(g)^*$, and a straightforward computation making use of the cocycle identities
yields the above formula.

Note that if $A = {\mathbb C}$, so $Y = C^*_{\rm r}(G,\sigma)$ for a scalar-valued two-cocycle $\sigma$, we have $Y' = A^\Sigma = \ell^2(G)$
with its usual norm and 
the  $v_Y$   we get in this way is in fact a  unitary representation of $G$ on $\ell^2(G)$ 
(since in this case ${\rm ad}_{\rho_Y}(\sigma(g,h))$ is the identity operator for every $g,h \in G$).
It is a kind of conjugate representation modified by $\sigma$; denoting it by $v_\sigma$, it is given by
$$ [v_\sigma(g)\, \xi](h) = \sigma(g,g^{-1}hg) \, \overline{\sigma(h,g)} \, \xi(g^{-1}hg) $$ 
for $g,h \in G$ and $\xi \in \ell^2(G)$.
In particular, when $G$ is abelian, we just get $[v_\sigma(g)\, \xi](h) = \tilde{\sigma}(g,h)\,\xi(h)$, where $\tilde{\sigma}$ is the bicharacter on $G$ obtained by symmetrizing $\sigma$. 
\end{example}

\medskip Let again $Y$ be a $C^*$-correspondence over $B = C^*_{\rm r}(\Sigma)$ and let $y,z \in Y$. Then define\\ $T: G \times A \to A$ by
$$T(g,a) = E\big(\big\langle y, (a\, \lambda(g)) \cdot z \big\rangle_B \, \lambda(g)^*\,\big) \quad \text{for } g\in G \text{ and } a \in A \,. $$
Since 
$$T(g,a) =\big \langle y, \rho_Y(a) \, v_Y(g) \, z \big\rangle_A \,, $$ where  $(\rho_Y,v_Y)$ is the equivariant representation of $\Sigma$ on $Y'$ constructed in Proposition \ref{rv}, it is clear that $T \in B(\Sigma)$.

\smallskip Conversely, we will show that any $T \in B(\Sigma)$ may be written as above. So assume that $T \in B(\Sigma)$ is given by
$T(g,a) = \big\langle x, \, \rho(a) \, v(g) \, x' \big\rangle$ for some equivariant representation $(\rho,v)$ of $\Sigma$ on a Hilbert $A$-module $X$
and $x,x' \in X$. We can then form the associated crossed product $C^*$-correspondence $X \rtimes_{{\rm r},v} G$ over $B=C^*_{\rm r}(\Sigma)$ constructed previously.
Set $y=x \odot e$  and $y' = x' \odot e$. Then  we have
$$T(g,a) = E\big(\big\langle y , (a \lambda(g)) \cdot y'  \big\rangle \, \lambda(g)^* \big)$$
for all $g \in G$ and $a \in A$. In order to verify this, we first observe that $(a \lambda(g)) \cdot y' \cdot \lambda(g)^*$ corresponds to the function in $C_c(G,X)$ equal to $(a \odot g) \cdot (x' \odot e) \cdot(\sigma(g^{-1},g)^* \odot g^{-1})$. Hence, 
\begin{align*}
E\big(\big\langle y , (a \lambda(g)) \cdot y'  \big\rangle \, \lambda(g)^* \big)
 & = E\big(\big\langle y , (a \lambda(g)) \cdot y'  \cdot \lambda(g)^*  \big\rangle \big) \\
& = \big\langle (x \odot e), (a \odot g) \cdot (x' \odot e) \cdot(\sigma(g^{-1},g)^* \odot g^{-1}) \big\rangle (e) \\
& = \big\langle x, \big[(a \odot g) \cdot (x' \odot e) \cdot(\sigma(g^{-1},g)^* \odot g^{-1})\big](e) \big\rangle \ .
\end{align*}
Now, one easily checks that
$$(a \odot g) \cdot (x' \odot e) = \big( a \cdot(v(g)x')\big) \odot g \  , $$
so that
\begin{align*}
\big[(a \odot g) & \cdot (x' \odot e) \cdot(\sigma(g^{-1},g)^* \odot g^{-1})\big](e) \\
& = \big[ \big(\big( a \cdot(v(g)x')\big) \odot g\big) \cdot(\sigma(g^{-1},g)^* \odot  g^{-1})\big](e) \\
& = a \cdot\Big( v(g) \big( (x' \odot e) \cdot (\sigma(g^{-1},g)^* \odot g^{-1})\big) (g^{-1}) \Big) \cdot \sigma(g,g^{-1}) \\
& = a \cdot \Big( v(g) \big( x' \cdot \sigma(g^{-1},g)^* \big) \Big) \cdot \sigma(g,g^{-1}) \\
& = a \cdot \big( v(g) \, x' \big) \cdot \alpha_g\big(\sigma(g^{-1},g)^*\big) \cdot \sigma(g,g^{-1})  \\
& = a \cdot \big(v(g) \, x'\big) = \rho(a) \, v(g) \, x' \,.
\end{align*}
Therefore we get
$$ E\big(\big\langle y , (a \lambda(g)) \cdot y'  \big\rangle \, \lambda(g)^* \big) =  \big\langle x,\rho(a) \, v(g) \, x' \big\rangle = T(g,a) \,, $$
as desired. In conclusion, we have shown the following result:

\begin{proposition}\label{FS-Cstar}
$B(\Sigma)$ consists of all functions from $G \times A$ into $A$ of the form
$$(g,a) \mapsto E\big(\big\langle y, (a \lambda(g)) \cdot z \big\rangle \, \lambda(g)^*\big)$$
where $y$ and $z$ belong to some $C^*$-correspondence $Y$ over $C^*_{\rm r}(\Sigma)$.
\end{proposition}

\begin{example}
 When $A = {\mathbb C}$ and $\sigma = 1$, Proposition \ref{FS-Cstar} gives that the Fourier-Stieltjes algebra $B(G)$ consists of all complex functions on $G$ of the form
$$g \mapsto \tau\big(\big\langle y, \lambda(g) \cdot z \big\rangle \, \lambda(g)^*\big)\,, $$
where $\tau$ denotes the canonical tracial state on $C^*_{\rm r}(G)$ and $y,z$ belong to some $C^*$-correspon-dence $Y$ over $C^*_{\rm r}(G)$. As a consequence, if $\varphi \in B(G)$, then there exists a bounded family $\{x_g\}_{g \in G}$ in $C^*_{\rm r}(G)$ such that $\varphi(g) = \widehat{x_g}(g)$ for all $g \in G$. We don't know whether the converse statement is true.
\end{example}

 \begin{remark} 
 By proceeding in a  similar way, one can show that $B(\Sigma)$ also consists of all functions from $G \times A$ into $A$ of the form
$$(g,a) \mapsto \big(E\circ\Lambda_\Sigma\big)\big(\,\big\langle y, [i_A(a)i_G(g)] \cdot z \big\rangle \, i_G(g)^*\big)$$
where $y$ and $z$ belong to some $C^*$-correspondence $Y$ over $C^*(\Sigma)$. We leave the reader to check this claim.  
\end{remark}

\bigskip \noindent{\bf Acknowledgements.} 
 Most of the present work has been done during several visits
 made by E.B. at the Sapienza University of Rome  and by R.C. at the 
University of Oslo in the period 2013--2015.
 Both authors  thank these institutions for their kind hospitality. They are also very grateful to the referee for his/her helpful comments and suggestions.

\medskip 
{\parindent=0pt Addresses of the authors:\\

\smallskip 
Erik B\'edos, Institute of Mathematics, University of
Oslo, \\
P.B. 1053 Blindern, N-0316 Oslo, Norway.\\ E-mail: bedos@math.uio.no \\

\smallskip 
\noindent
Roberto Conti, 
Universit\`{a} Sapienza di Roma, \\
 Dipartimento di Scienze di Base e Applicate per l'Ingegneria,\\
 via A. Scarpa 16, I-00166 Roma, Italy.
\\ E-mail: roberto.conti@sbai.uniroma1.it\par}

\end{document}